\documentclass[11pt]{article}
\usepackage[letterpaper,hmargin=1in,vmargin=1.25in]{geometry}

\usepackage{url}
\usepackage{hyperref}

\usepackage{amsmath}
\usepackage{amssymb}

\usepackage{amsthm}

\theoremstyle{plain}
\newtheorem{theorem}{Theorem}
\newtheorem{lemma}[theorem]{Lemma}
\newtheorem{corollary}[theorem]{Corollary}
\newtheorem{proposition}[theorem]{Proposition}
\newtheorem{definition}[theorem]{Definition}

\newtheorem{postulate}{Postulate}
\newtheorem{thesis}{Thesis}

\theoremstyle{definition}

\newcommand{\abs}[1]{\left\lvert#1\right\rvert}

\DeclareMathOperator{\Dom}{dom}

\newcommand{\rest}[2]{#1\!\!\restriction_{#2}}
\newcommand{\reste}[2]{#1\restriction_{#2}}
\newcommand{\osg}[1]{\left[#1\right]^{\prec}}

\newcommand{\N}{\mathbb{N}}

\newcommand{\Q}{\mathbb{Q}}
\newcommand{\R}{\mathbb{R}}

\newcommand{\XI}{\{0,1\}^\infty}

\newcommand{\Bm}[2]{\lambda_{#1}\left(#2\right)}

\newcommand{\PS}{\mathbb{P}}%
\newcommand{\charaps}[2]{C\!\left(#1,#2\right)}

\newcommand{\chara}[2]{\mathrm{C}_{#1}\left(#2\right)}
\newcommand{\cond}[2]{\mathrm{Filtered}_{#1}\left(#2\right)}

\newcommand{\noi}{\noindent}

\begin{document}


\begin{center}
{\Large \textbf{An operational characterization of \\
the notion of probability by algorithmic randomness II:
\vspace{1.7mm}\\
Discrete probability spaces}}
\end{center}

\vspace{0mm}

\begin{center}
Kohtaro Tadaki
\end{center}

\vspace{-5mm}

\begin{center}
Department of Computer Science, College of Engineering, Chubu University\\
1200 Matsumoto-cho, Kasugai-shi, Aichi 487-8501, Japan\\
E-mail: \textsf{tadaki@isc.chubu.ac.jp}\\
\url{http://www2.odn.ne.jp/tadaki/}
\end{center}

\vspace{-2mm}

\begin{quotation}
\noi\textbf{Abstract.}
The notion of probability plays an important role in almost all areas of science and technology.
In modern mathematics, however, probability theory means nothing other than measure theory,
and the operational characterization of the notion of probability is not established yet.
In this paper, based on the toolkit of algorithmic randomness we present
an operational characterization of the notion of probability, called an \emph{ensemble},
for
general
discrete probability spaces
whose sample space is countably infinite.
Algorithmic randomness, also known as algorithmic information theory,
is a field of mathematics which enables us to consider the randomness of an individual infinite sequence.
We use
an extension of
Martin-L\"of randomness with respect to
a generalized
Bernoulli measure
over
the
Baire space,
in order
to present the operational characterization.
In our former work~[K. Tadaki, arXiv:1611.06201],
we developed an operational characterization of the notion of probability
for an arbitrary finite probability space, i.e., a probability space whose sample space is
a finite set.
We then gave a natural operational characterization of the notion of conditional probability
in terms of ensemble for a finite probability space,
and gave equivalent characterizations of the notion of independence between two events based on it.
Furthermore,
we gave equivalent characterizations of the notion of independence of
an arbitrary number of events/random variables in terms of ensembles for finite probability spaces.
In particular, we showed that the independence between events/random variables is equivalent
to the independence in the sense of van Lambalgen's Theorem,
in the case where the underlying finite probability space is computable.
In this paper, we show that we can certainly extend
these results
over
general
discrete probability spaces whose sample space is countably infinite.
\end{quotation}

\begin{quotation}
\noi\textit{Key words\/}:
probability,
algorithmic randomness,
operational characterization,
discrete probability space,
Baire space,
Martin-L\"of randomness,
Bernoulli measure,
conditional probability,
independence,
van Lambalgen's Theorem
\end{quotation}


\newpage
\tableofcontents 
\newpage

\section{Introduction}

The notion of probability plays an important role in almost all areas of science and technology.
In modern mathematics, however, probability theory means nothing other than \emph{measure theory},
and the operational characterization of the notion of probability is not established yet.
In our former work~\cite{T14,T15,T16arXiv},
based on the toolkit of \emph{algorithmic randomness},
we presented an operational characterization of the notion of probability
for a finite probability space,
i.e., a probability space whose sample space is a finite set.

Algorithmic randomness
is a field of mathematics which enables us to consider the randomness of an individual infinite sequence.
In the former work~\cite{T14,T15,T16arXiv}
we used the notion of
\emph{Martin-L\"of randomness with respect to Bernoulli measure}
to present the operational characterization
for a finite probability space.

To clarify our motivation
and standpoint,
and the meaning of the operational characterization,
let us consider
a familiar example of a probabilistic phenomenon.
We
here
consider the repeated throwings of a fair die.
In this probabilistic phenomenon,
as throwings progressed,
a specific infinite sequence such as
\begin{equation*}
  3,5,6,3,4,2,2,3,6,1,5,3,5,4,1,\dotsc\dotsc\dotsc
\end{equation*}
is being generated,
where each number is the outcome of the corresponding throwing of the die.
Then the following
naive
question
may
arise naturally. 

\begin{quote}
\textbf{Question}:
What property should this infinite sequence
satisfy as a
probabilistic
phenomenon?
\end{quote}

In our former work~\cite{T14,T15,T16arXiv}
we tried to answer this question
for finite probability spaces in general,
including
the throwing of a fair die.
In the former work
we
characterized
the notion of probability
as
an infinite sequence of outcomes in a probabilistic phenomenon of a \emph{specific mathematical property}.
We called such an infinite sequence of outcomes
the \emph{operational characterization of the notion of probability}.
As the specific mathematical property,
in the work~\cite{T14,T15,T16arXiv}
we adopted
the notion of \emph{Martin-L\"of randomness with respect to Bernoulli measure},
a notion in algorithmic randomness.

In the work~\cite{T15,T16arXiv}
we put forward
this proposal
as a thesis (see Thesis~\ref{thesis} in Section~\ref{FPS} below),
in particular, for \emph{finite probability spaces in general}.
We
then
checked
the validity of
the thesis
based on
\emph{our intuitive understanding of the notion of probability}.
Furthermore,
we characterized \emph{equivalently}
the basic notions in probability theory in terms of the operational characterization.
Namely, we equivalently characterized the notion of the \emph{independence} of random variables/events
in terms of the operational characterization,
and represented the notion of \emph{conditional probability}
in terms of the operational characterization in a natural way.
The existence of these equivalent characterizations confirms further the validity of the thesis.
See Tadaki~\cite{T16arXiv}
for the detail of
our framework \cite{T14,T15,T16arXiv},
which was
developed especially for finite probability spaces in general.

The results above are about finite probability spaces.
In this paper, we show that we can certainly extend
the results above
over
\emph{general discrete probability spaces
whose sample space is countably infinite}.

\subsection{Historical background}

In the past century,
there was a comprehensive attempt to provide
an operational characterization of the notion of probability.
Namely, von Mises
developed
a mathematical theory of repetitive events which was aimed at reformulating
the theory of probability and statistics
based on an operational characterization of the notion of probability \cite{vM57,vM64}.
In a series of his comprehensive works which began in 1919,
von Mises developed this theory and, in particular,
introduced the notion of \emph{collective} as a mathematical idealization of a long sequence of
outcomes of experiments or observations repeated under a set of invariable conditions,
such as the repeated tossings of a coin or of a pair of dice.

The collective plays a role as an operational characterization of the notion of probability,
and is an infinite sequence of sample points in the sample space of a probability space.
As the randomness property of the collective, von Mises assumes that
all ``reasonable'' infinite subsequences of a collective satisfy the law of large numbers
with the identical limit value,
where the subsequences are selected using ``acceptable selection rules.''
Wald \cite{Wa36,Wa37} later showed that for any countable collection of selection rules,
there are sequences
which
are collectives in the sense of von Mises.
However, at the time it was unclear exactly what types of selection rules should be acceptable.
There seemed to von Mises to be no canonical choice.

Later, with the development of computability theory and the introduction of
generally accepted precise mathematical definitions of the notions of algorithm and computable function,
Church \cite{Ch40}
suggested
that a selection rule be considered acceptable if and only if it is computable.
In 1939, however, Ville \cite{Vi39} revealed the defect of the notion of collective.
Namely, he showed that for any countable collection of selection rules, there is a sequence that
is random in the sense of von Mises but has properties that make it clearly nonrandom.
In the first place,
the collective has an \emph{intrinsic defect} that
it cannot exclude the possibility that an event with probability zero may occur.
(For the development of the theory of collectives from the point of view of the definition of randomness,
see Downey and Hirschfeldt
\cite{DH10}.)

In 1966, Martin-L\"of \cite{M66} introduced the definition of random sequences,
which is called \emph{Martin-L\"of randomness} nowadays, and
plays a central role in the recent development of algorithmic randomness.
At the same time,
he introduced the notion of \emph{Martin-L\"of randomness with respect to Bernoulli measure} \cite{M66}.
He then pointed out that this notion overcomes the defect of the collective in the sense of von Mises, and
this can be regarded precisely as the collective which von Mises wanted to define.
However, he did not develop probability theory
based on Martin-L\"of random sequence with respect to Bernoulli measure.

Algorithmic randomness is a field of mathematics which studies the definitions of random sequences and
their property
(see \cite{N09,DH10} for the recent developments of
the field).
However, the recent research on algorithmic randomness would seem only interested in
the notions of randomness themselves and their
interrelation,
and not seem to have
made an attempt
to develop probability theory based on Martin-L\"of randomness with respect to Bernoulli measure in an operational manner so far.

\subsection{Contribution of the paper}

The subject of this paper is to make such an attempt for
\emph{general discrete probability spaces whose sample space is countably infinite},
as a sequel to our former work~\cite{T14,T15,T16arXiv} where we developed
a framework for an operational characterization of the notion of probability
for general finite probability spaces.
In the former work we did this,
precisely
based on Martin-L\"of randomness with respect to Bernoulli measure.
In contrast,
in this paper we present an operational characterization of the notion of probability
for general discrete probability spaces,
based on an \emph{extension} of Martin-L\"of randomness with respect to
a \emph{generalized} Bernoulli measure over the \emph{Baire space}.
Thus,
the core mathematical concept of this paper is
a Martin-L\"of random infinite sequence over the sample space of a discrete probability space,
with respect to a generalized Bernoulli measure on the Baire space.
In this paper
we call
it
an \emph{ensemble}, instead of collective for distinction.
The name ``ensemble'' comes from physics, in particular, from quantum mechanics and statistical mechanics.
We propose to identify it with
an infinite sequence of outcomes
resulting from
the infinitely repeated trials in a probabilistic phenomenon
described by
the discrete probability space.
We show that
the
ensemble has enough properties to regard it as an operational characterization of
the notion of probability
for a discrete probability space,
from the point of view of our intuitive understanding of the notion of probability.

Actually,
in a similar manner to our former work~\cite{T14,T15,T16arXiv} for finite probability spaces, 
in this paper
we
can
give a natural operational characterization of the notion of conditional probability
in terms of ensemble
for a discrete probability space,
and give equivalent characterizations of the notion of independence between two events based on it.
Furthermore,
we
can
give equivalent characterizations of the notion of independence of
an arbitrary number of events/random variables in terms of ensembles.
In particular, we
can
show that the independence
of
events/random variables is equivalent
to the independence in the sense of van Lambalgen's Theorem \cite{vL87},
in the case where the underlying
discrete
probability space is \emph{computable}.

From the operational point of view, we
must
be able to determine \emph{effectively}
whether each outcome of a trial is in the sample space of  the underlying discrete probability space,
or not.
Thus,
from that point of view,
we
must
only
consider
discrete probability spaces whose
sample spaces are \emph{recursive} infinite sets.
For mathematical generality,
however,
in this paper
we make a weaker assumption
about the sample spaces.
Namely, we assume that
the sample spaces
of discrete probability spaces
which we consider in this paper
are simply \emph{recursively enumerable} infinite sets.
We think this recursive enumerability of the sample space to be sufficiently general for our purpose.
On the other hand,
we emphasize that
a discrete probability space \emph{itself}
which we consider in this paper is not required to be computable at all
(except for in the results related to van Lambalgen's Theorem).
Therefore
the generalized Bernoulli measure which we consider
in this paper is
not necessarily computable
while the measures considered in
the field of
algorithmic randomness so far are \emph{usually} computable.
Thus, the
central
results in this paper hold for any
discrete
probability space
whose sample space is
a recursively enumerable infinite set.

Modern probability theory originated from the \emph{axiomatic approach} to probability theory,
introduced by Kolmogorov~\cite{K50} in 1933,
where the probability theory is precisely \emph{measure theory}.
One of the important roles of modern probability theory is, of course, in its applications to
the general areas of science and technology.
As we have already pointed out,
however,
an operational characterization of the notion of probability is still missing in modern probability theory.
Thus, when we apply the results of modern probability theory,
we have no choice but to
make such applications \emph{thoroughly based on our intuition without formal means}.

The aim of this paper,
as well as of our former work~\cite{T14,T15,T16arXiv},
is to try to fill in this gap between modern probability theory and its applications.
We present the operational characterization of the notion of probability as
a \emph{rigorous interface} between theory and practice,
without appealing to our intuition for filling in the gap.
Anyway,
in our framework
we \emph{keep} modern probability theory \emph{in its original form} without any modifications,
and propose the operational characterization of the notion of probability
as an \emph{additional mathematical structure} to it,
which
provides modern probability theory
with more comprehensive and rigorous 
opportunities for applications.

\subsection{Organization of the paper}

The paper is organized as follows.
We begin in Section~\ref{preliminaries} with some preliminaries to measure theory,
computability theory, and algorithmic randomness.
In Section~\ref{FPS},
we review the fundamental framework of the operational characterization of the notion of probability
for a finite probability space,
which was
introduced and developed by our former work~\cite{T14,T15,T16arXiv}.

We start our investigation
to provide an operational characterization of the notion of probability
for a discrete probability space in Section~\ref{Measure-theory-on-Baire-space}.
We there develop measure theory on
the
Baire space.
Although
the
Baire space is not compact and therefore it is hard to handle,
we can certainly develop measure theory on it.
In Section~\ref{DPS}, we introduce the notion of discrete probability space
for which
the operational characterization of the notion of probability is presented.
On this basis,
we
introduce
the extension of
Martin-L\"of randomness with respect to
a
generalized Bernoulli measure over
the
Baire space in Section~\ref{MLRP-Baire}.

In Section~\ref{OPT-Baire} we
introduce the notion of ensemble,
and put forward a thesis
which states to identify the ensemble
as an operational characterization of the notion of probability
for a discrete probability space.
We then
check
the validity of the thesis.
In Section~\ref{CPITW} we start to construct our framework
for developing the operational characterization,
by characterizing operationally the notions of conditional probability and
the independence between two events,
in terms of ensembles.
We then characterize operationally the notion of the independence
of an arbitrary number of events/random variables
in terms of ensembles in Section~\ref{IANERV}.
In Section~\ref{FENICFPS} we show that
the independence notions, introduced in the preceding sections, are further equivalent to
the notion of the independence in the sense of van Lambalgen's Theorem,
in the case where the underlying discrete probability space is computable,
by generalizing van Lambalgen's Theorem
over our framework.
Thus we show that the three independence notions, considered in this paper,
are all equivalent in this case.
We conclude this paper with
a mention of
the major application of our framework,
i.e.,
the
application
to quantum mechanics,
in Section~\ref{Concluding}.

\section{Preliminaries}
\label{preliminaries}

\subsection{Basic notation and definitions}
\label{basic notation}

We start with some notation about numbers and strings which will be used in this paper.
$\#S$ is the cardinality of $S$ for any set $S$.
$\N=\left\{0,1,2,3,\dotsc\right\}$ is the set of \emph{natural numbers},
and $\N^+$ is the set of \emph{positive integers}.
$\Q$ is the set of \emph{rationals}, and $\R$ is the set of \emph{reals}.

An \emph{alphabet} is a nonempty
set.
Let $\Omega$ be an arbitrary alphabet throughout the rest of this subsection.
A \emph{finite string over $\Omega$} is a finite sequence of elements from the alphabet $\Omega$.
We use $\Omega^*$ to denote the set of all finite strings over $\Omega$,
which contains the \emph{empty string} denoted by $\lambda$.
We use $\Omega^+$ to denote the set $\Omega^*\setminus\{\lambda\}$.
For any $\sigma\in\Omega^*$, $\abs{\sigma}$ is the \emph{length} of $\sigma$.
Therefore $\abs{\lambda}=0$.
For any $\sigma\in\Omega^+$ and $k\in\N^+$ with $k\le\abs{\sigma}$,
we use $\sigma(k)$ to denote the $k$th element in $\sigma$.
Therefore, we have $\sigma=\sigma(1)\sigma(2)\dots\sigma(\abs{\sigma})$
for every $\sigma\in\Omega^+$.
For any $n\in\N$,
we define the sets $\Omega^n$, $\Omega^{\le n}$, and $\Omega^{\ge n}$ as follows:
\begin{align*}
  \Omega^n&:=\{\,x\mid x\in\Omega^*\;\&\;\abs{x}=n\}, \\
  \Omega^{\le n}&:=\{\,x\mid x\in\Omega^*\;\&\;\abs{x}\le n\}, \\
  \Omega^{\ge n}&:=\{\,x\mid x\in\Omega^*\;\&\;\abs{x}\ge n\}.
\end{align*}
A subset $S$ of $\Omega^*$ is called
\emph{prefix-free}
if no string in $S$ is a prefix of another string in $S$.

An \emph{infinite sequence over $\Omega$} is an infinite sequence of elements from the alphabet $\Omega$,
where the sequence is infinite to the right but finite to the left.
We use $\Omega^\infty$ to denote the set of all infinite sequences over $\Omega$.

Let $\alpha\in\Omega^\infty$.
For any $n\in\N$ we denote by $\rest{\alpha}{n}\in\Omega^*$ the first $n$ elements
in the infinite sequence $\alpha$, and
for any $n\in\N^+$ we denote by $\alpha(n)$ the $n$th element in $\alpha$.
Thus, for example, $\rest{\alpha}{4}=\alpha(1)\alpha(2)\alpha(3)\alpha(4)$, and $\rest{\alpha}{0}=\lambda$.

For any $S\subset\Omega^*$, the set
$\{\alpha\in\Omega^\infty\mid\exists\,n\in\N\;\rest{\alpha}{n}\in S\}$
is denoted by $\osg{S}$.
Note that
(i) $\osg{S}\subset\osg{T}$ for every $S\subset T\subset\Omega^*$, and
(ii) for every set $S\subset\Omega^*$ there exists a prefix-free set $P\subset\Omega^*$ such that
$\osg{S}=\osg{P}$.
For any $\sigma\in\Omega^*$, we denote by $\osg{\sigma}$ the set $\osg{\{\sigma\}}$, i.e.,
the set of all infinite sequences over $\Omega$ extending $\sigma$.
Therefore $\osg{\lambda}=\Omega^\infty$.

For any function $f$,
the domain of definition of $f$ is denoted by $\Dom f$.

\subsection{Measure theory on infinite sequences over a finite alphabet}
\label{MR}

A \emph{finite alphabet} is a non-empty finite set.
Let $\Omega$ be an arbitrary finite alphabet throughout the rest of this subsection.
We briefly review measure theory
on $\Omega^\infty$
according to Nies~\cite[Section 1.9]{N09}.
See also Billingsley~\cite{B95}
for measure theory in general.

\begin{definition}[Outer measure]\label{def_outer-measure}
Let $\Gamma$ be a nonempty set.
A real-valued function $\mu$ defined on the class of all subsets of $\Gamma$ is called
an \emph{outer measure on $\Gamma$} if the following conditions hold.
\begin{enumerate}
\item $\mu\left(\emptyset\right)=0$;
\item $\mu\left(\mathcal{C}\right)\le\mu\left(\mathcal{D}\right)$
  for every subsets $\mathcal{C}$ and $\mathcal{D}$ of $\Gamma$
  with $\mathcal{C}\subset\mathcal{D}$;
\item $\mu\left(\bigcup_{i}\mathcal{C}_i\right)\le\sum_{i}\mu\left(\mathcal{C}_i\right)$
  for every sequence $\{\mathcal{C}_i\}_{i\in\N}$ of subsets of $\Gamma$.\qed
\end{enumerate}
\end{definition}

A \emph{probability measure representation over $\Omega$} is
a function $r\colon\Omega^*\to[0,1]$ such that
\begin{enumerate}
  \item $r(\lambda)=1$ and
  \item for every $\sigma\in\Omega^*$ it holds that
    \begin{equation}\label{pmr}
       r(\sigma)=\sum_{a\in\Omega}r(\sigma a).
    \end{equation}
\end{enumerate}
A probability measure representation $r$ over $\Omega$ \emph{induces}
an outer measure $\mu_r$ on $\Omega^\infty$ in the following manner:
A subset $\mathcal{R}$ of $\Omega^\infty$ is \emph{open} if
$\mathcal{R}=\osg{S}$ for some $S\subset\Omega^*$.
Let
$r$
be an arbitrary probability measure representation over $\Omega$. 
For each open subset $\mathcal{A}$ of $\Omega^\infty$, we define $\mu_r(\mathcal{A})$ by
\begin{equation*}
  \mu_r(\mathcal{A}):=\sum_{\sigma\in E}r(\sigma),
\end{equation*}
where $E$ is a prefix-free subset of $\Omega^*$ with $\osg{E}=\mathcal{A}$.
Due to the equality~\eqref{pmr} the sum is independent of the choice of  the prefix-free set $E$,
and therefore the value $\mu_r(\mathcal{A})$ is well-defined.
Then, for any subset $\mathcal{C}$ of $\Omega^\infty$, we define $\mu_r(\mathcal{C})$ by
\begin{equation*}
  \mu_r(\mathcal{C}):=
  \inf\{\mu_r(\mathcal{A})\mid
  \mathcal{C}\subset\mathcal{A}\text{ \& $\mathcal{A}$ is an open subset of $\Omega^\infty$}\}.
\end{equation*}
We can then show that $\mu_r$ is an \emph{outer measure} on $\Omega^\infty$  such that
$\mu_r(\Omega^\infty)=1$.

A class $\mathcal{F}$ of subsets of $\Omega^\infty$ is called
a \emph{$\sigma$-field on $\Omega^\infty$}
if  $\mathcal{F}$ includes $\Omega^\infty$, is closed under complements,
and is closed under the formation of countable unions.
The \emph{Borel class} $\mathcal{B}_{\Omega}$ is the $\sigma$-field \emph{generated by}
all open sets on $\Omega^\infty$.
Namely, the Borel class $\mathcal{B}_{\Omega}$ is defined
as the intersection of all the $\sigma$-fields on $\Omega^\infty$ containing
all open sets
on
$\Omega^\infty$.
A real-valued function $\mu$ defined on the Borel class $\mathcal{B}_{\Omega}$ is called
a \emph{probability measure on $\Omega^\infty$} if the following conditions hold.
\begin{enumerate}
\item $\mu\left(\emptyset\right)=0$ and $\mu\left(\Omega^\infty\right)=1$;
\item $\mu\left(\bigcup_{i}\mathcal{D}_i\right)=\sum_{i}\mu\left(\mathcal{D}_i\right)$
  for every sequence $\{\mathcal{D}_i\}_{i\in\N}$ of sets in $\mathcal{B}_{\Omega}$ such that
  $\mathcal{D}_i\cap\mathcal{D}_i=\emptyset$ for all $i\neq j$.
\end{enumerate}
Then, for every probability measure representation $r$ over $\Omega$,
we can show that the restriction of the outer measure $\mu_r$ on $\Omega^\infty$
to the Borel class $\mathcal{B}_{\Omega}$ is
a probability measure on $\Omega^\infty$.
We denote the restriction of $\mu_r$ to $\mathcal{B}_{\Omega}$ by
$\mu_r$
just the same.

Then it is easy to see that
\begin{equation}\label{mr}
  \mu_r\left(\osg{\sigma}\right)=r(\sigma)
\end{equation}
for every probability measure representation $r$ over $\Omega$ and every $\sigma\in \Omega^*$.

\subsection{Computability}
\label{Computability}

A \emph{partial computable function} is a function $f$ such that
there exists a deterministic Turing machine $\mathcal{M}$ with the properties that
\begin{enumerate}
  \item $\Dom f\subset D$, where $D$ denotes the set of all the inputs for $\mathcal{M}$, and
  \item for each input $x\in D$, when executing $\mathcal{M}$ with the input $x$,
  \begin{enumerate}
    \item if $x\in\Dom f$ then the computation of $\mathcal{M}$ eventually terminates
      and then $\mathcal{M}$ outputs $f(x)$;
    \item if $x\notin\Dom f$ then the computation of $\mathcal{M}$ does not terminate.
  \end{enumerate}
\end{enumerate}
A partial computable function is also called a \emph{partial recursive function}.
A \emph{computable function} is a partial computable function $f$ such that
$\Dom f$ equals to $D$ in the above definition of partial computable function.
Namely,
a \emph{computable function} is a function $f$ such that
there exists a deterministic Turing machine $\mathcal{M}$ with the properties that
\begin{enumerate}
  \item $\Dom f$ equals to the set of all the inputs for $\mathcal{M}$, and
  \item for each $x\in\Dom f$, when executing $\mathcal{M}$ with the input $x$,
    the computation of $\mathcal{M}$ eventually terminates and then $\mathcal{M}$ outputs $f(x)$.
\end{enumerate}
A computable function is also called a \emph{total recursive function}.

We say that $\alpha\in\Omega^\infty$ is \emph{computable}
if the mapping $\N\ni n\mapsto\rest{\alpha}{n}$ is a computable function.
A real $a$ is called \textit{computable} if there exists a computable function $g\colon\N\to\Q$ such that $\abs{a-g(k)} < 2^{-k}$ for all $k\in\N$.
A real $a$ is called \emph{left-computable} if
there exists a computable, increasing sequence of rationals which converges to $a$, i.e., if
there exists a computable function $h\colon\N\to\Q$ such that
$h(n)\le h(n+1)$ for every $n\in\N$ and $\lim_{n\to\infty}h(n)=a$.
On the other hand, a real $a$ is called \textit{right-computable} if
$-a$ is left-computable.
It is then easy to see that, for every $a\in\R$,
$a$ is computable if and only if $a$ is both left-computable and right-computable.

A \emph{recursively enumerable set} is a set $S$
such that there exists a deterministic Turing machine $\mathcal{M}$ with the properties that
\begin{enumerate}
  \item $S\subset D$, where $D$ denotes the set of all the inputs for $\mathcal{M}$, and
  \item for each input $x\in D$, when executing $\mathcal{M}$ with the input $x$,
  \begin{enumerate}
    \item if $x\in S$ then the computation of $\mathcal{M}$ eventually terminates;
    \item if $x\notin S$ then the computation of $\mathcal{M}$ does not terminate.
  \end{enumerate}
\end{enumerate}
We write ``r.e.'' instead of ``recursively enumerable.''
A \emph{recursive set} is a set $S$
such that there exists a deterministic Turing machine $\mathcal{M}$ with the properties that
\begin{enumerate}
  \item $S\subset D$, where $D$ denotes the set of all the inputs for $\mathcal{M}$, and
  \item for each input $x\in D$, when executing $\mathcal{M}$ with the input $x$,
    the computation of $\mathcal{M}$ eventually terminates and then
    $\mathcal{M}$ outputs $1$ if $x\in S$ and
    $0$ otherwise.
\end{enumerate}
Note that every recursive set is an r.e.~set, and every r.e.~set is a countable set. 

\subsection{Martin-L\"of randomness with respect to an arbitrary probability measure}
\label{MLRam}

In this subsection,
we introduce the notion of \emph{Martin-L\"of randomness} \cite{M66}
in a general setting.

Let $\Omega$ be an arbitrary finite alphabet,
and $\mu$ be an arbitrary probability measure on $\Omega^\infty$.
The basic idea of Martin-L\"of randomness
(with respect to the probability measure $\mu$)
is as follows.
\begin{quote}
\textbf{Basic idea of Martin-L\"of randomness}:
The \emph{random} infinite sequences over $\Omega$ are precisely sequences
which are not contained in any \emph{effective null set} on $\Omega^\infty$.
\end{quote}
Here, an \emph{effective null set} on $\Omega^\infty$ is
a set $\mathcal{S}\in\mathcal{B}_\Omega$ such that $\mu(\mathcal{S})=0$
and moreover $\mathcal{S}$ has some type of \emph{effective} property.
As a specific implementation of the idea of effective null set, we introduce the following notion.

\begin{definition}[Martin-L\"{o}f test with respect to a probability measure]
\label{ML-testM}
Let $\Omega$ be a finite alphabet, and let $\mu$ be a probability measure on $\Omega^\infty$.
A subset $\mathcal{C}$ of $\N^+\times\Omega^*$ is called a
\emph{Martin-L\"{o}f test with respect to $\mu$} if
$\mathcal{C}$ is
an r.e.~set
such that for every $n\in\N^+$ it holds that $\mathcal{C}_n$ is a prefix-free subset of $\Omega^*$ and
\begin{equation}\label{muocn<2n}
  \mu\left(\osg{\mathcal{C}_n}\right)<2^{-n},
\end{equation}
where $\mathcal{C}_n$ denotes the set
$\left\{\,
    \sigma\mid (n,\sigma)\in\mathcal{C}
\,\right\}$.
\qed
\end{definition}

Let $\mathcal{C}$ be a Martin-L\"{o}f test with respect to $\mu$.
Then,
it follows from \eqref{muocn<2n}
that $\mu\left(\bigcap_{n=1}^{\infty}\osg{\mathcal{C}_n}\right)=0$.
Therefore,
the set $\bigcap_{n=1}^{\infty}\osg{\mathcal{C}_n}$ serves as an effective null set.
In this manner, the notion of an effective null set is implemented
as
a Martin-L\"{o}f test with respect a probability measure
in Definition~\ref{ML-testM}.

Then,
the notion of \emph{Martin-L\"of randomness with respect to a probability measure} is
defined as follows, according to the basic idea of Martin-L\"{o}f randomness stated above.

\begin{definition}[Martin-L\"{o}f randomness with respect to a probability measure]
\label{ML-randomness-wrtm}
Let $\Omega$ be a finite alphabet, and let $\mu$ be a probability measure on $\Omega^\infty$.
For any $\alpha\in\Omega^\infty$, we say that $\alpha$ is
\emph{Martin-L\"{o}f random with respect to $\mu$} if
$$\alpha\notin\bigcap_{n=1}^{\infty}\osg{\mathcal{C}_n}$$
for every Martin-L\"{o}f test $\mathcal{C}$ with respect to $\mu$.\qed
\end{definition}

\section{Operational characterization of the notion of probability for a finite probability space}
\label{FPS}

In our former work~\cite{T14,T15,T16arXiv}
we
provided
an operational characterization of the notion of probability for a \emph{finite probability space}.
In this section we review
the fundamental framework of the operational characterization for a finite probability space.

First,
a finite probability space is
defined as follows.

\begin{definition}[Finite probability space]\label{def-FPS}
Let $\Omega$ be a finite alphabet. A \emph{finite probability space on $\Omega$} is a function $P\colon\Omega\to\R$
such that
\begin{enumerate}
  \item $P(a)\ge 0$ for every $a\in \Omega$, and
  \item $\sum_{a\in \Omega}P(a)=1$.
\end{enumerate}
The set of all finite probability spaces on $\Omega$ is denoted by $\PS(\Omega)$.

Let $P\in\PS(\Omega)$.
The set $\Omega$ is called the \emph{sample space} of $P$,
and elements
of
$\Omega$ are called \emph{sample points} or \emph{elementary events}
of $P$.
For each $A\subset\Omega$, we define $P(A)$ by
$$P(A):=\sum_{a\in A}P(a).$$
A subset of $\Omega$ is called an \emph{event} on $P$, and
$P(A)$ is called the \emph{probability} of $A$
for every event $A$
on $P$.
\qed
\end{definition}

In the framework~\cite{T14,T15,T16arXiv},
a finite alphabet  $\Omega$ plays
a role of the set of all possible outcomes of stochastic trials such experiments or observations.
An operational characterization of the notion of probability which we
provide
for a finite probability space
on $\Omega$ is an infinite sequence over $\Omega$.
In order to provide
it,
we use the notion of Martin-L\"of randomness with respect to \emph{Bernoulli measure}.
A Bernoulli measure is introduced
in the following manner.

Let $\Omega$ be a finite alphabet, and let $P\in\PS(\Omega)$.
For each $\sigma\in\Omega^*$, we use $P(\sigma)$ to denote
$P(\sigma_1)P(\sigma_2)\dots P(\sigma_n)$
where $\sigma=\sigma_1\sigma_2\dots\sigma_n$ with $\sigma_i\in\Omega$.
Therefore $P(\lambda)=1$, in particular.
For each subset $S$ of $\Omega^*$, we use $P(S)$ to denote
$$\sum_{\sigma\in S}P(\sigma).$$
Therefore $P(\emptyset)=0$, in particular.

Consider a function $r\colon\Omega^*\to[0,1]$ such that $r(\sigma)=P(\sigma)$ for every $\sigma\in\Omega^*$.
It is then easy to see that the function $r$ is a probability measure representation over $\Omega$.
The probability measure $\mu_r$ induced by $r$ is
called
a \emph{Bernoulli measure on $\Omega^\infty$}, denoted
$\lambda_{P}$.
The Bernoulli measure $\lambda_{P}$ on $\Omega^\infty$ has
the following property:
For every $\sigma\in \Omega^*$,
\begin{equation*}%
  \Bm{P}{\osg{\sigma}}=P(\sigma),
\end{equation*}
which results from \eqref{mr}.

Martin-L\"of randomness with respect to Bernoulli measure,
which is called \emph{Martin-L\"of $P$-randomness} in
our framework,
is defined as follows.
This notion was, in essence, introduced by Martin-L\"{o}f~\cite{M66},
as well as the notion of Martin-L\"of randomness
with respect to Lebesgue measure.

\begin{definition}[%
Martin-L\"of $P$-randomness,
Martin-L\"{o}f \cite{M66}]\label{ML_P-randomness}
Let $\Omega$ be a finite alphabet, and let $P\in\PS(\Omega)$.
For any $\alpha\in\Omega^\infty$,
we say that $\alpha$ is \emph{Martin-L\"{o}f $P$-random} if
$\alpha$ is Martin-L\"{o}f random with respect to $\lambda_{P}$.
\qed
\end{definition}

Let $\Omega$ be a finite alphabet, and let $P\in\PS(\Omega)$.
In the work~\cite{T14,T15,T16arXiv},
we propose to regard a Martin-L\"of $P$-random sequence of sample points
as an \emph{operational characterization of the notion of probability}
for a finite probability space $P$ on $\Omega$.
Namely,
we propose to identify a Martin-L\"of $P$-random sequence of sample points
with the \emph{substance} of the notion of probability for a finite probability space $P$.
Thus, since the notion of Martin-L\"of $P$-random sequence plays a central role in our framework,
in particular we call it an \emph{ensemble}, as in Definition~\ref{ensemble-finite},
instead of collective for distinction.

\begin{definition}[Ensemble]\label{ensemble-finite}
Let $\Omega$ be a finite alphabet, and let $P\in\PS(\Omega)$.
A Martin-L\"of $P$-random
infinite sequence over $\Omega$
is called an \emph{ensemble} for the finite probability space $P$ on $\Omega$.
\qed
\end{definition}

Consider an infinite sequence $\alpha\in\Omega^\infty$ of outcomes which is being generated
by infinitely repeated trials \emph{described by} the finite probability space $P$.
The operational characterization of the notion of probability for the finite probability space $P$
is thought to be completed
if the property which the infinite sequence $\alpha$ has to satisfy is determined.
In the work~\cite{T15,T16arXiv}
we thus proposed the following thesis.

\begin{thesis}[Tadaki~\cite{T15,T16arXiv}]\label{thesis}
Let $\Omega$ be a finite alphabet, and let $P\in\PS(\Omega)$.
An infinite sequence of outcomes in $\Omega$ which is
being generated by infinitely repeated trials \emph{described by} the finite probability space $P$ on $\Omega$
is an ensemble for $P$.
\qed
\end{thesis}

In the work~\cite{T14,T15,T16arXiv},
we confirmed the validity of Thesis~\ref{thesis} from the various aspects.

\section{Measure theory on the Baire space}
\label{Measure-theory-on-Baire-space}

From now on, we start our investigation
to provide an operational characterization of the notion of probability
for a \emph{discrete probability space}.

First, we develop measure theory on
the
\emph{Baire space}.
An \emph{countable alphabet} is a countably infinite set.
Let $\Omega$ be an arbitrary countable alphabet throughout the rest of this section.
We introduce a measure on
$\Omega^\infty$
by generalizing the argument given in Nies~\cite[Section 1.9]{N09}, i.e.,
by generalizing
the argument reviewed in Subsection~\ref{MR}
above.%
\footnote{The Baire space is the set of all infinite sequences of natural numbers.
Since $\Omega$ is countably infinite,
the set $\Omega^\infty$
is, in essence,
the
Baire space.}
Although the space $\Omega^\infty$ is not compact
and therefore it is hard to handle,
we can certainly develop measure theory on $\Omega^\infty$.

\begin{definition}[Measure representation]
A \emph{measure representation over $\Omega$} is
a function $r\colon\Omega^*\to[0,1]$ such that for every $\sigma\in\Omega^*$ it holds that
\begin{equation}\label{pmr-Bs}
  r(\sigma)=\sum_{a\in\Omega}r(\sigma a).
\end{equation}
\qed
\end{definition}

A measure representation $r$ over $\Omega$ \emph{induces}
an outer measure $\mu_r$ on $\Omega^\infty$ in the following manner:
For any subset $S$ of $\Omega^*$, we use $r(S)$ to denote
\[
  \sum_{\sigma\in S}r(\sigma)
\]
(may be $\infty$).
For any subset $S$ of $\Omega^*$ and $\rho\in\Omega^*$, we use $S[\rho]$ to denote
the set of all $\sigma\in S$ such that $\rho$ is a prefix of $\sigma$.

First, we show the following theorem.
For any $S\subset \Omega^*$ and $n\in\N$, we denote by $\rest{S}{n}$
the set of all $\sigma\in\Omega^n$ such that $\sigma$ is a prefix of some element of $S$.

\begin{theorem}\label{thm_rErlerr}
Let $r$ be a measure representation over $\Omega$.
Let $\rho\in\Omega^*$ and let $E$ be a prefix-free subset of
$\Omega^*$.
Then $r(E[\rho])$ converges and satisfies that $r(E[\rho])\le r(\rho)$.
\end{theorem}

\begin{proof}
First, we show the result in the case of $\rho=\lambda$, i.e., we show that
$r(E)\le r(\lambda)$.
For any $n\in\N$, we use $E_{\le n}$, $E_n$, and $E_{\ge n}$
to denote the sets  $E\cap\Omega^{\le n}$,  $E\cap\Omega^n$, and  $E\cap\Omega^{\ge n}$,
respectively.
In particular, we denote the set $\rest{E_{\ge n+1}}{n}$ by $\widetilde{E}_{>n}$ for any $n\in\N$.
Note that $E_{\le n}\cap \widetilde{E}_{>n}=\emptyset$ for every $n\in\N$, since $E$ is prefix-free.

We prove the following inequality by induction on $n\in\N$:
\begin{equation}\label{rencupenlerl}
  r(E_{\le n}\cup \widetilde{E}_{>n})\le r(\lambda).
\end{equation}
Since $E_{\le 0}\cup \widetilde{E}_{>0}\subset\{\lambda\}$, the inequality~\eqref{rencupenlerl} holds for $n=0$, obviously.

For an arbitrary $k\in\N$, assume that the inequality~\eqref{rencupenlerl} holds for $n=k$.
Then, we show that
\begin{equation}\label{rE<=k+1uwE>k+1<=rE<=kuwE>k}
  r(E_{\le k+1}\cup \widetilde{E}_{>k+1})\le r(E_{\le k}\cup \widetilde{E}_{>k}).
\end{equation}
First, we see that
\begin{equation}\label{rEkuwE>k}
  r(E_{\le k}\cup \widetilde{E}_{>k})
  =r(E_{\le k})+r(\widetilde{E}_{>k})
  =r(E_{\le k})+r(\rest{E_{\ge k+1}}{k}),
\end{equation}
where the first equality follows from the fact that $E_{\le k}\cap \widetilde{E}_{>k}=\emptyset$.
We then see that
\begin{equation}\label{rs=saOrsa=gstArt+stBnrt}
  r(\rest{E_{\ge k+1}}{k})
  =\sum_{\sigma\in\reste{E_{\ge k+1}}{k}}r(\sigma)
  =\sum_{\sigma\in\reste{E_{\ge k+1}}{k}}\sum_{a\in\Omega}r(\sigma a)
  \ge\sum_{\tau\in\reste{E_{\ge k+1}}{k+1}}r(\tau)
  =r(\rest{E_{\ge k+1}}{k+1}),
\end{equation}
where the second equality follows from \eqref{pmr-Bs}.
Since
$r(\rest{E_{\ge k+1}}{k+1})=r(E_{k+1})+r(\rest{E_{\ge k+2}}{k+1})$,
it follows from \eqref{rEkuwE>k} and \eqref{rs=saOrsa=gstArt+stBnrt} that
\begin{align*}
  r(E_{\le k}\cup \widetilde{E}_{>k})
  &\ge r(E_{\le k})+r(E_{k+1})+r(\rest{E_{\ge k+2}}{k+1})
  =r(E_{\le k+1})+r(\widetilde{E}_{>k+1}) \\
  &=r(E_{\le k+1}\cup \widetilde{E}_{>k+1}),
\end{align*}
where the last equality follows from the fact that $E_{\le k+1}\cap \widetilde{E}_{>k+1}=\emptyset$.
Thus, we have the inequality~\eqref{rE<=k+1uwE>k+1<=rE<=kuwE>k}, as desired.
Hence, from the assumption we have
that the inequality~\eqref{rencupenlerl} holds for $n=k+1$.

Thus, the inequality~\eqref{rencupenlerl} holds for all $n\in\N$.
It follows that $r(E_{\le n})\le r(E_{\le n+1})\le r(\lambda)$ for all $n\in\N$.
Thus, $r(E)$ converges and satisfies that $r(E)\le r(\lambda)$.

Next, we show the result $r(E[\rho])\le r(\rho)$ in the general case of an arbitrary $\rho\in\Omega^*$.
Since $E$ is prefix free, there exists a prefix-free subset $F$ of $\Omega^*$ such that
$E[\rho]=\{\rho\sigma\mid \sigma\in F\}$.
Consider a function $q\colon\Omega^*\to[0,1]$ defined by $q(\sigma):=r(\rho\sigma)$.
Since $r$ is a measure representation over $\Omega$, it is easy to see that
$q$ is also a measure representation over $\Omega$.
Applying the result above to $q$ and $F$, we have $q(F)\le q(\lambda)$,
which implies that $r(E[\rho])\le r(\rho)$, as desired.
This completes the proof.
\end{proof}

\begin{theorem}\label{orsoEr-rEr=rr}
Let $r$ be a measure representation over $\Omega$.
Let $\rho\in\Omega^*$ and let $E$ be a prefix-free subset of
$\Omega^*$.
Suppose that $\osg{\rho}\subset\osg{E[\rho]}$.
Then $r(E[\rho])=r(\rho)$.
\end{theorem}

\begin{proof}
First, we show the result in the case of $\rho=\lambda$, i.e.,
we show that $r(E)=r(\lambda)$ if $\Omega^\infty=\osg{E}$.
Note that $\Omega$ is well-ordered, since it is a countably infinite set.
Thus, every non-empty subset of $\Omega$ has a least element.

Now, let us assume contrarily that $\Omega^\infty=\osg{E}$ but $r(E)\neq r(\lambda)$.
It follows from Theorem~\ref{thm_rErlerr} that
\begin{equation}\label{rE<rl}
  r(E)<r(\lambda).
\end{equation}
Based on this,
we choose an infinite sequence $\tau_0,\tau_1,\tau_2,\tau_3,\dotsc$ of elements of $\Omega^*$
such that
\begin{enumerate}
  \item $\abs{\tau_n}=n$,
  \item $r(E[\tau_n])<r(\tau_n)$, and
  \item there exists $a\in\Omega$ with the properties that
    $\tau_n a=\tau_{n+1}$ and $a$ is the least element of $\Omega$ for which $r(E[\tau_na])<r(\tau_na)$
\end{enumerate}
for all $n\in\N$, inductively, in the following manner.

First, we set $\tau_0:=\lambda$.
Obviously, $\abs{\tau_0}=0$ and we have $r(E[\tau_0])<r(\tau_0)$ due to \eqref{rE<rl}.
Assume that the sequence $\tau_0,\tau_1,\tau_2,\dots,\tau_k$ satisfying
the properties~(i), (ii), and (iii) above
has already been chosen.
Then
\begin{equation}\label{rEtauk<rtauk}
  r(E[\tau_k])<r(\tau_k)
\end{equation}
holds, in particular.
On the other hand, it follows
from Theorem~\ref{thm_rErlerr} that $r(E[\tau_ka])\le r(\tau_ka)$ for every $a\in\Omega$.
Assume contrarily that $r(E[\tau_ka])=r(\tau_ka)$ for every $a\in\Omega$.
Then, since
$r$ is a measure representation over $\Omega$, we have that
\begin{equation*}
  r(E[\tau_k])\ge\sum_{a\in\Omega}r(E[\tau_ka])=\sum_{a\in\Omega}r(\tau_ka)=r(\tau_k).
\end{equation*}
However, this contradicts the inequality~\eqref{rEtauk<rtauk}.
Thus, we have that $r(E[\tau_ka_0])<r(\tau_ka_0)$ for some $a_0\in\Omega$.
We then choose a least $a\in\Omega$ such that $r(E[\tau_ka])<r(\tau_ka)$,
and set $\tau_{k+1}:=\tau_ka$.
As a result,
the properties~(i), (ii), and (iii) hold for the sequence $\tau_0,\tau_1,\tau_2,\dots,\tau_k,\tau_{k+1}$,
certainly.

In this manner, we can generate
an infinite sequence $\tau_0,\tau_1,\tau_2,\tau_3,\dotsc$ of elements of $\Omega^*$
satisfying the properties~(i), (ii), and (iii) above.

Then,
due to the properties~(i) and (iii),
there exists an infinite sequence $\alpha\in\Omega^\infty$ such that
$\rest{\alpha}{n}=\tau_n$ for all
$n\in\N$.
Since $\Omega^\infty=\osg{E}$, we have that $\alpha\in\osg{E}$
and therefore $\rest{\alpha}{n_0}\in E$ for some
$n_0\in\N$.
This implies that $\tau_{n_0}\in E$.
Therefore, $r(E[\tau_{n_0}])=r(\tau_{n_0})$ since $E$ is prefix-free.
However, this contracts
the property~(ii) above
which implies
that $r(E[\tau_{n_0}])<r(\tau_{n_0})$.
Hence, we have that $r(E)=r(\lambda)$
if $\Omega^\infty=\osg{E}$, as desired.

Next, we show the result in the general case,
i.e., we show that $r(E[\rho])=r(\rho)$ if $\osg{\rho}\subset\osg{E[\rho]}$.
Since $E$ is prefix free, there exists a prefix-free subset $F$ of $\Omega^*$ such that
$E[\rho]=\{\rho\sigma\mid \sigma\in F\}$.
It follows that if $\osg{\rho}\subset\osg{E[\rho]}$ then $\Omega^\infty=\osg{F}$.
On the other hand,
consider a function $q\colon\Omega^*\to[0,1]$ defined by $q(\sigma):=r(\rho\sigma)$.
Since $r$ is a measure representation over $\Omega$, it is easy to see that
$q$ is also a measure representation over $\Omega$.
Applying the result above to $q$ and $F$, we have that
if $\osg{\rho}\subset\osg{E[\rho]}$ then $q(F)=q(\lambda)$,
which implies that $r(E[\rho])=r(\rho)$, as desired.
This completes the proof.
\end{proof}

\begin{theorem}\label{oEleoFimprElerF}
Let $r$ be a measure representation over $\Omega$.
Let $E$ and $F$ be prefix-free subsets of $\Omega^*$.
Suppose that $\osg{E}\subset\osg{F}$. Then $r(E)\le r(F)$.
\end{theorem}

\begin{proof}
In the case where $E$ is an empty set, the result is obvious.
Thus, we assume that  $E$ is a nonempty set, in what follows.
Therefore, since $\osg{E}\subset\osg{F}$, $F$ is also nonempty.
 
Let $G$ be the set of all $\sigma\in\Omega^*$ such that (i) $\osg{\sigma}\subset\osg{F}$
and (ii) $\osg{\rho}\not\subset\osg{F}$ for every proper prefix $\rho$ of $\sigma$.
Since $F$ is a nonempty set, $G$ is also nonempty.

First, we show that $r(F)=r(G)$.
On the one hand, it is easy to see that $\osg{\sigma}\subset\osg{F[\sigma]}$ for every $\sigma\in G$.
Thus, it follows from Theorem~\ref{orsoEr-rEr=rr} that
\begin{equation}\label{oEsoFimprElerF_rFslers}
  r(F[\sigma])=r(\sigma)
\end{equation}
for every $\sigma\in G$.
On the other hand, note that $G$ is a prefix-free set.
Therefore, we have that $F[\sigma_1]\cap F[\sigma_2]=\emptyset$
for every $\sigma_1,\sigma_2\in G$ with $\sigma_1\neq\sigma_2$, in particular.
Since some prefix of $\rho$ is in $G$ for every $\rho\in F$, we have
\begin{equation*}
  F=\bigcup_{\sigma\in G}F[\sigma].
\end{equation*}
Hence, using \eqref{oEsoFimprElerF_rFslers} we have that
\[
  r(F)=\sum_{\sigma\in G}r(F[\sigma])=\sum_{\sigma\in G}r(\sigma)=r(G),
\]
as desired.

Next, we
show that $r(E)\le r(G)$.
As above, since $G$ is prefix-free, we have that $E[\sigma_1]\cap E[\sigma_2]=\emptyset$
for every $\sigma_1,\sigma_2\in G$ with $\sigma_1\neq\sigma_2$.
Since $\osg{E}\subset\osg{F}$, for each $\rho\in E$ we see that $\osg{\rho}\subset\osg{F}$ and
therefore some prefix of $\rho$ is in $G$.
Thus we have 
\begin{equation*}
  E=\bigcup_{\sigma\in G}E[\sigma].
\end{equation*}
Hence, it follows from Theorem~\ref{thm_rErlerr} that
\[
  r(E)=\sum_{\sigma\in G}r(E[\sigma])\le \sum_{\sigma\in G}r(\sigma)=r(G),
\]
as desired.

Thus, we have $r(E)\le r(G)=r(F)$.
This completes the proof. 
\end{proof}

The following is immediate from Theorem~\ref{oEleoFimprElerF}.

\begin{corollary}\label{oE=oE'imprE=rE'}
Let $r$ be a measure representation over $\Omega$.
Let $E$ and $E'$ be prefix-free subsets of $\Omega^*$.
Suppose that $\osg{E}=\osg{E'}$. Then $r(E)=r(E')$.
\qed
\end{corollary}

A subset $\mathcal{R}$ of $\Omega^\infty$ is \emph{open} if
$\mathcal{R}=\osg{S}$ for some $S\subset\Omega^*$.
It is easy to see that for every open subset $\mathcal{A}$ of $\Omega^\infty$
there exists a prefix-free subset $E$ of $\Omega^*$ such that $\mathcal{A}=\osg{E}$.
For
the set of all $\sigma\in\Omega^*$ such that (i) $\osg{\sigma}\subset\mathcal{A}$
and (ii) $\osg{\rho}\not\subset\mathcal{A}$ for every proper prefix $\rho$ of $\sigma$
serves as such a prefix-free set $E$.

Let $r$ be an arbitrary measure representation over $\Omega$. 
For each open subset $\mathcal{A}$ of $\Omega^\infty$, we define $r(\mathcal{A})$ by
\begin{equation*}
  r(\mathcal{A}):=r(E),
\end{equation*}
where $E$ is a prefix-free subset of $\Omega^*$ with $\osg{E}=\mathcal{A}$.
Due to Corollary~\ref{oE=oE'imprE=rE'},
the real value
$r(E)$ is independent of the choice of  the prefix-free set $E$
and therefore the
real
value $r(\mathcal{A})$ is well-defined.

Then, for any subset $\mathcal{C}$ of $\Omega^\infty$, we define $\mu_r(\mathcal{C})$ by
\begin{equation}\label{mu_rC=infrAmidCsubsetAopen}
  \mu_r(\mathcal{C}):=
  \inf\{r(\mathcal{A})\mid
  \mathcal{C}\subset\mathcal{A}\text{ \& $\mathcal{A}$ is an open subset of $\Omega^\infty$}\}.
\end{equation}
We can then show the following theorem.

\begin{theorem}\label{thm_mur-outer-measure-Bs}
Let $r$ be a measure representation over $\Omega$.
Then $\mu_r$ is an outer measure on $\Omega^\infty$ such that
$\mu_r(\mathcal{A})=r(\mathcal{A})$
for every open subset $\mathcal{A}$ of $\Omega^\infty$.
\end{theorem}

\begin{proof}
First, note that $\mathcal{C}\subset\osg{\{\lambda\}}$ for every $\mathcal{C}\subset\Omega^\infty$,
since $\osg{\{\lambda\}}=\Omega^\infty$.
Therefore, since $\osg{\{\lambda\}}$ is an open subset of $\Omega^\infty$ and $r(\osg{\{\lambda\}})=r(\lambda)$,
for each $\mathcal{C}\subset\Omega^\infty$
we see that the infimum in the right-hand side of \eqref{mu_rC=infrAmidCsubsetAopen} exists
as a non-negative real
at most $r(\lambda)$.
Thus, $\mu_r(\mathcal{C})$ is a non-negative real for every $\mathcal{C}\subset\Omega^\infty$.

Secondly, it follows from Theorem~\ref{oEleoFimprElerF} that,
for every open subsets $\mathcal{A}$ and $\mathcal{B}$ of $\Omega^\infty$,
if $\mathcal{A}\subset\mathcal{B}$ then $r(\mathcal{A})\le r(\mathcal{B})$.
This implies that $\mu_r(\mathcal{A})=r(\mathcal{A})$
for every open subset $\mathcal{A}$ of $\Omega^\infty$, as desired.

Since $\emptyset$ is an open subset of $\Omega^\infty$, we have $\mu_r(\emptyset)=r(\emptyset)=0$.
It is also easy to show that $\mu_r(\mathcal{C})\le\mu_r(\mathcal{D})$
for every subsets $\mathcal{C}$ and $\mathcal{D}$ of $\Omega^\infty$
with $\mathcal{C}\subset\mathcal{D}$.
To see this, let $\mathcal{C}$ and $\mathcal{D}$ be arbitrary subsets of $\Omega^\infty$
with $\mathcal{C}\subset\mathcal{D}$, and let $\varepsilon$ be an arbitrary positive real.
Then
there exists an open subset $\mathcal{A}$ of $\Omega^\infty$
such that $\mathcal{D}\subset\mathcal{A}$ and $r(\mathcal{A})<\mu_r(\mathcal{D})+\varepsilon$.
Since $\mathcal{C}\subset\mathcal{D}\subset\mathcal{A}$,
it follows that $\mu_r(\mathcal{C})<\mu_r(\mathcal{D})+\varepsilon$.
Since $\varepsilon$ is arbitrary, we have $\mu_r(\mathcal{C})\le\mu_r(\mathcal{D})$, as desired.

Finally, we show that $\mu_r(\bigcup_{i}\mathcal{C}_i)\le\sum_{i}\mu_r(\mathcal{C}_i)$
for every sequence $\{\mathcal{C}_i\}_{i\in\N}$ of subsets of $\Omega^\infty$.
Let $\{\mathcal{C}_i\}_{i\in\N}$ be an arbitrary sequence of subsets of $\Omega^\infty$.
In the case where $\sum_{i}\mu_r(\mathcal{C}_i)$ diverges, the result is obvious.
Thus, we assume that $\sum_{i}\mu_r(\mathcal{C}_i)$ converges, in what follows.
Let $\varepsilon$ be an arbitrary positive real.
Then, for each $i$ there exists an open subset $\mathcal{A}_i$ of $\Omega^\infty$ such that
$\mathcal{C}_i\subset\mathcal{A}_i$ and
\begin{equation}\label{thm_mromBs_eq0}
  r(\mathcal{A}_i)<\mu_r(\mathcal{C}_i)+\varepsilon 2^{-i}.
\end{equation}
Let $E$ be the set of all $\sigma\in\Omega^*$ such that
(i) $\osg{\sigma}\subset\mathcal{A}_i$ for some $i$
and (ii) $\osg{\rho}\not\subset\mathcal{A}_i$ for every proper prefix $\rho$ of $\sigma$ and every $i$.
Then, $E$ is a prefix-free subset of $\Omega^*$ and $\osg{E}=\bigcup_{i}\mathcal{A}_i$.
Thus, we have
\begin{equation}\label{thm_mromBs_eq1}
  r\left(\bigcup_{i}\mathcal{A}_i\right)=r(E).
\end{equation}
For each $i$, let $E_i$ be the set of all $\sigma\in E$ such that (i) $\osg{\sigma}\subset\mathcal{A}_i$
but (ii) $\osg{\sigma}\not\subset\mathcal{A}_k$ for every $k<i$.
It follows that $E=\bigcup_{i}E_i$ and $E_i\cap E_j=\emptyset$ for every $i\neq j$.
Thus, we have
\begin{equation}\label{thm_mromBs_eq2}
  r(E)=\sum_{i}r(E_i)
\end{equation}
On the other hand,
for each $i$,
since $\osg{E_i}\subset\mathcal{A}_i$ and $E_i$ is prefix-free,
it follows from Theorem~\ref{oEleoFimprElerF} that
\begin{equation}\label{thm_mromBs_eq3}
  r(E_i)\le r(\mathcal{A}_i).
\end{equation}
Hence, since $\bigcup_{i}\mathcal{C}_i\subset\bigcup_{i}\mathcal{A}_i$,
using \eqref{thm_mromBs_eq1}, \eqref{thm_mromBs_eq2}, \eqref{thm_mromBs_eq3},
\eqref{thm_mromBs_eq0} we have that
\[
  \mu_r\left(\bigcup_{i}\mathcal{C}_i\right)\le r\left(\bigcup_{i}\mathcal{A}_i\right)
  <\sum_{i}\left\{\mu_r(\mathcal{C}_i)+\varepsilon 2^{-i}\right\}
  =\sum_{i}\mu_r(\mathcal{C}_i)+\varepsilon.
\]
Thus, since $\varepsilon$ is an arbitrary positive real,
we have $\mu_r(\bigcup_{i}\mathcal{C}_i)\le\sum_{i}\mu_r(\mathcal{C}_i)$,
as desired.
\end{proof}

\begin{definition}[$\sigma$-field and measure]
Let $\Gamma$ be a nonempty set.
A class $\mathcal{F}$ of subsets of $\Gamma$ is called a \emph{$\sigma$-field in $\Gamma$}
if  $\mathcal{F}$ includes $\Gamma$, is closed under complements,
and is closed under the formation of countable unions.
A real-valued function $\mu$ defined on a $\sigma$-field $\mathcal{F}$ in $\Gamma$ is called
a \emph{measure on $\mathcal{F}$} if the following conditions hold.
\begin{enumerate}
\item $\mu\left(\emptyset\right)=0$;
\item $\mu\left(\bigcup_{i}\mathcal{D}_i\right)=\sum_{i}\mu\left(\mathcal{D}_i\right)$
  for every sequence $\{\mathcal{D}_i\}_{i\in\N}$ of sets in $\mathcal{F}$ such that
  $\mathcal{D}_i\cap\mathcal{D}_i=\emptyset$ for all $i\neq j$.\qed
\end{enumerate}
\end{definition}

\begin{definition}
[Carath\'{e}odory~\cite{Cara68}]
\label{def_mu-measurable-set}
Let $\Gamma$ be a nonempty set, and let $\mu$ be an outer measure on $\Gamma$.
A subset $\mathcal{G}$ of $\Gamma$ is called \emph{$\mu$-measurable} if
\[
  \mu(\mathcal{C}\cap\mathcal{G})+\mu(\mathcal{C}\setminus\mathcal{G})=\mu(\mathcal{C})
\]
for every subset $\mathcal{C}$ of $\Gamma$.
The class of all $\mu$-measurable sets is denoted by $\mathcal{M}(\mu)$.
\qed
\end{definition}

Carath\'{e}odory~\cite{Cara68} showed the following central result of measure theory.

\begin{theorem}[Carath\'{e}odory~\cite{Cara68}]\label{thm_mu-measurable-set}
Let $\Gamma$ be a nonempty set, and let $\mu$ be an outer measure on $\Gamma$.
Then $\mathcal{M}(\mu)$ is a $\sigma$-field in $\Gamma$,
and $\mu$ restricted to $\mathcal{M}(\mu)$ is a measure on $\mathcal{M}(\mu)$.
\qed
\end{theorem}

The \emph{Borel class} $\mathcal{B}_{\Omega}$ is the $\sigma$-field \emph{generated by}
all open sets on $\Omega^\infty$.
Namely, the Borel class $\mathcal{B}_{\Omega}$ is defined
as the intersection of all the $\sigma$-fields in $\Omega^\infty$ containing
all open sets on $\Omega^\infty$.

\begin{theorem}\label{thm_Brlsbmumsbl}
Let $r$ be a measure representation over $\Omega$.
Then $\mathcal{B}_{\Omega}\subset\mathcal{M}(\mu_r)$.
\end{theorem}

\begin{proof}
First, note from Theorems~\ref{thm_mur-outer-measure-Bs} and \ref{thm_mu-measurable-set}
that $\mathcal{M}(\mu_r)$ is
a $\sigma$-field in $\Omega^\infty$.
Since the Borel class $\mathcal{B}_{\Omega}$ is the $\sigma$-field generated by
all open sets on $\Omega^\infty$,
it is sufficient to show that
all open sets on $\Omega^\infty$ are $\mu_r$-measurable.
For showing this in turn, it is sufficient to prove that $\osg{\sigma}$ is $\mu_r$-measurable
for every $\sigma\in\Omega^*$,
since $\mathcal{M}(\mu_r)$ is a $\sigma$-field in $\Omega^\infty$
and every subset of $\Omega^*$ is at most countable.

Let $\sigma\in\Omega^*$ and let $\mathcal{C}$ be a subset of $\Omega^\infty$.
We show that
$\mu_r(\mathcal{C}\cap\osg{\sigma})+\mu_r(\mathcal{C}\setminus\osg{\sigma})\le\mu_r(\mathcal{C})$. 
Let $\varepsilon$ be an arbitrary positive real.
Then, there exists an open subset $\mathcal{A}$ of $\Omega^\infty$ such that
$\mathcal{C}\subset\mathcal{A}$ and
\begin{equation}\label{thm_Brlsbmumsbl_eq1}
  r(\mathcal{A})<\mu_r(\mathcal{C})+\varepsilon.
\end{equation}
Note that if $\mathcal{D}_1$ and $\mathcal{D}_2$ are open subsets of $\Omega^\infty$
then $\mathcal{D}_1\cap\mathcal{D}_2$ is also an open subset of $\Omega^\infty$.
This can be confirmed by the equality
\[
  \mathcal{D}_1\cap\mathcal{D}_2
  =\osg{\{\rho\in\Omega^*\mid\osg{\rho}\subset\mathcal{D}_1\cap\mathcal{D}_2\}}.
\]
Thus, $\mathcal{A}\cap\osg{\sigma}$ is an open set, in particular.
Since
$\mathcal{A}\setminus\osg{\sigma}=\mathcal{A}\cap\osg{\Omega^{\abs{\sigma}}\setminus\{\sigma\}}$,
we see that $\mathcal{A}\setminus\osg{\sigma}$ is also an open set.
Since $(\mathcal{A}\cap\osg{\sigma})\cap(\mathcal{A}\setminus\osg{\sigma})=\emptyset$,
it follows from Lemma~\ref{ropnAcupopnB=rA+rB} below that
\begin{equation}\label{thm_Brlsbmumsbl_eq2}
  r(\mathcal{A}\cap\osg{\sigma})+r(\mathcal{A}\setminus\osg{\sigma})=r(\mathcal{A}).
\end{equation}
Hence, since $\mathcal{C}\cap\osg{\sigma}\subset\mathcal{A}\cap\osg{\sigma}$ and
$\mathcal{C}\setminus\osg{\sigma}\subset\mathcal{A}\setminus\osg{\sigma}$,
using \eqref{thm_Brlsbmumsbl_eq2} and \eqref{thm_Brlsbmumsbl_eq1} we have that
\[
  \mu_r(\mathcal{C}\cap\osg{\sigma})+\mu_r(\mathcal{C}\setminus\osg{\sigma})
  \le r(\mathcal{A}\cap\osg{\sigma})+r(\mathcal{A}\setminus\osg{\sigma})
  <\mu_r(\mathcal{C})+\varepsilon.
\]
Thus, since $\varepsilon$ is an arbitrary positive real,
we have
$\mu_r(\mathcal{C}\cap\osg{\sigma})+\mu_r(\mathcal{C}\setminus\osg{\sigma})\le\mu_r(\mathcal{C})$,
as desired.

Then, it follows from the conditions~(i) and (iii) of Definition~\ref{def_outer-measure} that
\[
  \mu_r(\mathcal{C}\cap\osg{\sigma})+\mu_r(\mathcal{C}\setminus\osg{\sigma})=\mu_r(\mathcal{C}).
\]
Therefore, $\osg{\sigma}$ is $\mu_r$-measurable.
This completes the proof.
\end{proof}

\begin{lemma}\label{ropnAcupopnB=rA+rB}
Let $r$ be a measure representation over $\Omega$.
For every open subsets $\mathcal{A}$ and $\mathcal{B}$ of $\Omega^\infty$,
if $\mathcal{A}\cap\mathcal{B}=\emptyset$ then
$r(\mathcal{A}\cup\mathcal{B})=r(\mathcal{A})+r(\mathcal{B})$.
\end{lemma}

\begin{proof}
Let $\mathcal{A}$ and $\mathcal{B}$ be open subsets of $\Omega^\infty$.
Then there exist prefix-free subsets $E$ and $F$ of $\Omega^*$ such that
$\mathcal{A}=\osg{E}$ and $\mathcal{B}=\osg{F}$.
Since $\mathcal{A}\cap\mathcal{B}=\emptyset$, we see that
$E\cap F=\emptyset$ and $E\cup F$ is prefix-free.
Therefore, since $\mathcal{A}\cup\mathcal{B}=\osg{E\cup F}$,
we have
$r(\mathcal{A}\cup\mathcal{B})=r(E\cup F)=r(E)+r(F)=r(\mathcal{A})+r(\mathcal{B})$.
\end{proof}

Thus, for every measure representation $r$ over $\Omega$,
based on Theorem~\ref{thm_mur-outer-measure-Bs},
\ref{thm_Brlsbmumsbl}, and \ref{thm_mu-measurable-set} we see that
the restriction of the outer measure $\mu_r$ on $\Omega^\infty$
to the Borel class $\mathcal{B}_{\Omega}$ is a measure on
$\mathcal{B}_{\Omega}$.
We denote the restriction of $\mu_r$ to $\mathcal{B}_{\Omega}$ by
$\mu_r$
just the same
in what follows.

Then it follows from Theorem~\ref{thm_mur-outer-measure-Bs} that
\begin{equation}\label{mr-Bs}
  \mu_r\left(\osg{\sigma}\right)=r(\sigma)
\end{equation}
for every measure representation $r$ over $\Omega$ and every $\sigma\in \Omega^*$.

\begin{definition}[Probability measure representation]
A \emph{probability measure representation over $\Omega$} is
a measure representation $r$ over $\Omega$ with $r(\lambda)=1$.
\qed
\end{definition}

\begin{definition}[Probability measure]
Let $\Gamma$ be a nonempty set, and let $\mathcal{F}$ be a $\sigma$-field in $\Gamma$.
A \emph{probability measure on $\mathcal{F}$}
is a measure $\mu$ on $\mathcal{F}$ with $\mu(\Gamma)=1$.
\qed
\end{definition}

Using \eqref{mr-Bs},
we see that, for every probability measure representation $r$ over $\Omega$,
the measure $\mu_r$ on $\mathcal{B}_{\Omega}$ is a probability measure on $\mathcal{B}_{\Omega}$.

\section{Discrete probability spaces}
\label{DPS}

In this paper
we give an operational characterization of the notion of probability for a \emph{discrete probability space}.%
\footnote{Normaly, a discrete probability space is a probability space whose
sample space is finite or countably infinite.
For distinction, a discrete probability space in this paper means a discrete probability space whose
sample space is countably infinite.}
A discrete probability space is
defined as follows.

\begin{definition}[Discrete probability space]\label{def-DPS}
Let $\Omega$ be a countable alphabet.
A \emph{discrete probability space on $\Omega$} is a function $P\colon\Omega\to\R$
such that
\begin{enumerate}
  \item $P(a)\ge 0$ for every $a\in \Omega$, and
  \item $\sum_{a\in \Omega}P(a)=1$.
\end{enumerate}
The set of all discrete probability spaces on $\Omega$ is denoted by $\PS(\Omega)$.

Let $P\in\PS(\Omega)$.
The set $\Omega$ is called the \emph{sample space} of $P$,
and elements
of
$\Omega$ are called \emph{sample points} or \emph{elementary events}
of $P$.
For each $A\subset\Omega$, we define $P(A)$ by
$$P(A):=\sum_{a\in A}P(a).$$
A subset of $\Omega$ is called an \emph{event} on $P$, and
$P(A)$ is called the \emph{probability} of $A$
for every event $A$
on $P$.
\qed
\end{definition}

Let $\Omega$ be an arbitrary countable alphabet through out the rest of this section.
It plays a role of the set of all possible outcomes of
a stochastic trial.
An operational characterization of the notion of probability which we give for a discrete probability space
on $\Omega$ is an infinite sequence over $\Omega$.

In order to provide
such
an operational characterization of the notion of probability
we use an extension of Martin-L\"of randomness over a
countable
alphabet.
For that purpose,
we first introduce the notion of a \emph{generalized} Bernoulli measure on $\Omega^\infty$
as follows.

Let $P\in\PS(\Omega)$.
For each $\sigma\in\Omega^*$, we use $P(\sigma)$ to denote
$P(\sigma_1)P(\sigma_2)\dots P(\sigma_n)$
where $\sigma=\sigma_1\sigma_2\dots\sigma_n$ with $\sigma_i\in\Omega$.
Therefore $P(\lambda)=1$, in particular.
For each subset $S$ of $\Omega^*$, we use $P(S)$ to denote
\[
  \sum_{\sigma\in S}P(\sigma).
\]
Therefore $P(\emptyset)=0$, in particular.

Consider a function $r\colon\Omega^*\to[0,1]$ such that $r(\sigma)=P(\sigma)$ for every $\sigma\in\Omega^*$.
It is then easy to see that the function $r$ is a probability measure representation over $\Omega$.
The probability measure $\mu_r$
on
$\mathcal{B}_{\Omega}$,
induced by $r$,
is called a \emph{generalized Bernoulli measure on $\Omega^\infty$},
denoted $\lambda_{P}$.
The generalized Bernoulli measure $\lambda_{P}$ on $\Omega^\infty$ has the following property:
For every $\sigma\in \Omega^*$,
\begin{equation}\label{pBm-Bs}
  \Bm{P}{\osg{\sigma}}=P(\sigma),
\end{equation}
which results from \eqref{mr-Bs}.

In this paper, we develop an operational characterization of the notion of probability
for discrete probability spaces,
whose sample space is \emph{countably infinite.}
From the operational point of view, we
must
be able to determine \emph{effectively}
whether each outcome of a trial is in the sample space of  the discrete probability space,
or not.
Thus, in this paper we consider discrete probability spaces whose
sample spaces are \emph{recursive} infinite sets.
For the same reason,
we
must
be able to determine \emph{effectively}
whether each outcome of a trial is in a given event of
a
discrete probability space,
or not.
Thus, in this paper we consider \emph{recursive} events of
discrete probability spaces.
For mathematical generality,
however,
we
make a weaker assumption especially about the sample spaces.
Namely, we assume that
the sample spaces are simply \emph{recursively enumerable} infinite sets,
when stating definitions and results
throughout the rest of this paper.

It is
convenient to introduce the notion of \emph{computable discrete probability space} as follows.

\begin{definition}[Computability of discrete probability space]
Let $\Omega$ be an r.e.~infinite set,
and let $P\in\PS(\Omega)$.
We say that $P$ is \emph{computable} if
there exists
a partial recursive function $f$ such that
(i) $\Dom f=\Omega\times\N$,
(ii) $f(\Dom f)\subset\Q$, and
(iii) $\abs{P(a)-f(a,k)}\le 2^{-k}$ for every $a\in\Omega$ and $k\in\N$.
\qed
\end{definition}

We may try to weaken the notion of the computability for a discrete probability space as follows:
Let $\Omega$ be an r.e.~infinite set, and let $P\in\PS(\Omega)$.
We say that $P$ is \emph{left-computable} if
there exists a partial recursive function
$f$
such that
(i) $\Dom f=\Omega\times\N$,
(ii) $f(\Dom f)\subset \Q$,
(iii) $P(a)\ge f(a,k)$ for every $a\in\Omega$ and $k\in\N$, and
(iv) $\lim_{k\to\infty} f(a,k)=P(a)$ for every $a\in\Omega$.
On the other hand,
we say that $P$ is \emph{right-computable} if
there exists a partial recursive function
$f$
such that
(i) $\Dom f=\Omega\times\N$,
(ii) $f(\Dom f)\subset \Q$,
(iii) $P(a)\le f(a,k)$ for every $a\in\Omega$ and $k\in\N$, and
(iv) $\lim_{k\to\infty} f(a,k)=P(a)$ for every $a\in\Omega$.
However, using the condition~(ii) of Definition~\ref{def-DPS} we can see that
these three computable notions for a discrete probability space coincide with one another,
as the following proposition states.

\begin{proposition}\label{computable-left-computable-right-computable-equivalent}
Let $\Omega$ be an r.e.~infinite set,
and let $P\in\PS(\Omega)$.
The following conditions are equivalent to one another.
\begin{enumerate}
  \item $P$ is computable.
  \item $P$ is left-computable.
  \item $P$ is right-computable.\qed
\end{enumerate}
\end{proposition}

\section{\boldmath Extension of Martin-L\"of randomness
over discrete probability spaces}
\label{MLRP-Baire}

In order to provide an operational characterization of the notion of probability
we use an extension of Martin-L\"of randomness over a generalized Bernoulli measure.

Martin-L\"of randomness with respect to a generalized Bernoulli measure,
which is called \emph{Martin-L\"of $P$-randomness} in this paper, is defined as follows.

\begin{definition}
[Martin-L\"of $P$-randomness]\label{ML_P-randomness-Bs}
Let $\Omega$ be an r.e.~infinite set, and let $P\in\PS(\Omega)$.
\begin{enumerate}
  \item A subset $\mathcal{C}$ of $\N^+\times \Omega^*$ is called a \emph{Martin-L\"{o}f $P$-test} if
    $\mathcal{C}$ is an r.e.~set such that for every $n\in\N^+$ it holds that
    $\mathcal{C}_n$ is a prefix-free subset of $\Omega^*$ and
    $$\Bm{P}{\osg{\mathcal{C}_n}}<2^{-n},$$
    where
    $\mathcal{C}_n
    :=
    \left\{\,
      \sigma\bigm|(n,\sigma)\in\mathcal{C}
    \,\right\}$.
  \item For any $\alpha\in\Omega^\infty$ and Martin-L\"{o}f $P$-test $\mathcal{C}$,
    we say that $\alpha$ \emph{passes} $\mathcal{C}$ if there exists $n\in\N^+$ such that
    $\alpha\notin\osg{\mathcal{C}_n}$.
  \item For any $\alpha\in\Omega^\infty$, we say that $\alpha$ is \emph{Martin-L\"{o}f $P$-random} if
    for every Martin-L\"{o}f $P$-test $\mathcal{C}$
    it holds that $\alpha$ passes $\mathcal{C}$.\qed
\end{enumerate}
\end{definition}

Note that we do not require $P$ to be computable in Definition~\ref{ML_P-randomness-Bs}.
Thus, the generalized Bernoulli measure $\lambda_{P}$ itself is not necessarily computable
in Definition~\ref{ML_P-randomness-Bs}.
Here, we say that a generalized Bernoulli measure $\lambda_{P}$ is \emph{computable} if there exists
a partial recursive function $g$ such that
(i) $\Dom g=\Omega^*\times\N$, (ii) $g(\Dom g)\subset\Q$, and
(iii) $\abs{\Bm{P}{\osg{\sigma}}-g(\sigma,k)} < 2^{-k}$ for all $\sigma\in\Omega^*$ and $k\in\N$.
Note also that in Definition~\ref{ML_P-randomness-Bs} we do not require that
$P(a)>0$ for all $a\in\Omega$.
Therefore, $P(a_0)$ may be $0$ for some $a_0\in\Omega$.

In Definition~\ref{ML_P-randomness-Bs}, we require that
the set $\mathcal{C}_n$ is prefix-free in the definition of a Martin-L\"{o}f $P$-test $\mathcal{C}$.
However, we can eliminate this
requirement
while keeping the notion of Martin-L\"of $P$-randomness
the same.
Namely, we can show the following theorem.

\begin{theorem}\label{eliminate-prefix-freeness}
Let $\Omega$ be an r.e.~infinite set, and let $P\in\PS(\Omega)$.
For every r.e.~subset $\mathcal{C}$ of $\N^+\times \Omega^*$ such that
$\Bm{P}{\osg{\mathcal{C}_n}}<2^{-n}$ for every $n\in\N^+$,
then there exists a Martin-L\"{o}f $P$-test $\mathcal{D}\subset\N^+\times \Omega^*$ such that
$\osg{\mathcal{C}_n}=\osg{\mathcal{D}_n}$ for every $n\in\N^+$.
\qed
\end{theorem}

Actually, from Theorem~\ref{eliminate-prefix-freeness} we have the following theorem.

\begin{theorem}\label{ML_P-randomness_eliminated-prefix-freeness}
Let $\Omega$ be an r.e.~infinite set, and let $P\in\PS(\Omega)$.
Let $\alpha\in\Omega^\infty$.
Then the following conditions are equivalent to each other.
\begin{enumerate} 
  \item The infinite sequence $\alpha$ is Martin-L\"{o}f $P$-random.
  \item For every r.e.~subset $\mathcal{C}$ of $\N^+\times \Omega^*$ such that
    $\Bm{P}{\osg{\mathcal{C}_n}}<2^{-n}$ for every $n\in\N^+$,
    there exists $n\in\N^+$ such that $\alpha\notin\osg{\mathcal{C}_n}$.\qed
\end{enumerate}
\end{theorem}

Since there are only countably infinitely many algorithms,
we can show the following
theorem, as is
shown for the usual Martin-L\"of randomness
for infinite binary sequences with respective to Lebesgue measure.

\begin{theorem}\label{Bmae-Baire}
Let $\Omega$ be an r.e.~infinite set, and let $P\in\PS(\Omega)$.
Then $\mathrm{ML}_P\in\mathcal{B}_{\Omega}$ and $\Bm{P}{\mathrm{ML}_P}=1$,
where $\mathrm{ML}_P$ is the set of all Martin-L\"of $P$-random sequences over $\Omega$.
\end{theorem}

\begin{proof}
Since there are only countably infinitely many Turing machines,
there are only countably infinitely many Martin-L\"{o}f $P$-tests
$\mathcal{C}^1,\mathcal{C}^2,\mathcal{C}^3,\dotsc$.
For each $i\in\N^+$, let
$\mathrm{NML}_P^i$
be the set of all $\alpha\in\Omega^\infty$ which does not pass $\mathcal{C}_i$.

Let $i\in\N^+$.
We see that $\mathrm{NML}_P^i=\bigcap_{n=1}^\infty\osg{\mathcal{C}^i_n}$ and therefore
$\mathrm{NML}_P^i\in\mathcal{B}_{\Omega}$.
Since
\[
  \Bm{P}{\mathrm{NML}_P^i}\le\Bm{P}{\osg{\mathcal{C}^i_n}}<2^{-n}
\]
for every $n\in\N^+$,
we have $\Bm{P}{\mathrm{NML}_P^i}=0$.
Thus, since $\Omega^\infty\setminus\mathrm{ML}_P=\bigcup_{i=1}^\infty \mathrm{NML}_P^i$,
it follows that $\mathrm{ML}_P\in\mathcal{B}_{\Omega}$ and
$\Bm{P}{\Omega^\infty\setminus\mathrm{ML}_P}=0$.
In particular, the latter
implies that $\Bm{P}{\mathrm{ML}_P}=1$, as desired.
\end{proof}

\section{Ensemble}
\label{OPT-Baire}

Let $\Omega$ be an arbitrary r.e.~infinite set throughout this section.
In this section
we present an operational characterization of the notion of probability for a discrete probability space,
and consider its validity.
We propose to regard a Martin-L\"of $P$-random sequence of sample points
as an \emph{operational characterization of the notion of probability}
for a discrete probability space $P$ on $\Omega$.
Namely,
we propose to identify a Martin-L\"of $P$-random sequence of sample points
with the \emph{substance} of the notion of probability for a discrete probability space $P$.
Thus, since the notion of Martin-L\"of $P$-random sequence plays a central role in our framework,
in particular we call it an \emph{ensemble}, as in Definition~\ref{ensemble}.
The name ``ensemble'' comes from physics, in particular, from quantum mechanics and
statistical mechanics.%
\footnote{The notion of ensemble plays a fundamental role in quantum mechanics and
statistical mechanics. However, the notion is very vague in physics from a mathematical point of view.
We propose to regard a Martin-L\"of $P$-random sequence of quantum states
as a formal definition of the notion of ensemble in quantum mechanics and statistical mechanics
\cite{T15Kokyuroku,T15WiNF-Tadaki_rule,T16QIT35,T18arXiv}.}

\begin{definition}[Ensemble]\label{ensemble}
Let $P\in\PS(\Omega)$.
A Martin-L\"of $P$-random
infinite sequence over $\Omega$
is called an \emph{ensemble} for the discrete probability space $P$ on $\Omega$.
\qed
\end{definition}

Let $P\in\PS(\Omega)$.
Consider an infinite sequence $\alpha\in\Omega^\infty$ of outcomes which is being generated
by infinitely repeated trials \emph{described by} the discrete probability space $P$.
The operational characterization of the notion of probability for the discrete probability space $P$
is thought to be completed
if the property which the infinite sequence $\alpha$ has to satisfy is determined.
We thus propose the following thesis.

\begin{thesis}\label{thesis-Bs}
Let $P\in\PS(\Omega)$.
An infinite sequence of outcomes in $\Omega$ which is being generated by
infinitely repeated trials \emph{described by} the discrete probability space $P$ on $\Omega$
is an ensemble for $P$.
\qed
\end{thesis}

Let us
check
the validity of Thesis~\ref{thesis-Bs}.
First of all, what is ``probability''?
It would seem very difficult to answer this question
\emph{completely} and \emph{sufficiently}.
However, we may enumerate the \emph{necessary} conditions which the notion of probability is
considered to have to satisfy
\emph{according to
our intuitive understanding of the notion of probability}.
In the subsequent subsections,
we check that the notion of ensemble satisfies these necessary conditions.

\subsection{Event with probability one}

Let $P\in\PS(\Omega)$, and
let us consider
an infinite sequence $\alpha\in\Omega^\infty$ of outcomes
which is being generated by infinitely repeated trials described
by the discrete probability space $P$ on $\Omega$.
The first necessary condition which the notion of probability for the discrete probability space $P$
is considered to have to satisfy is
the condition that
\emph{an elementary event with probability one always occurs in the infinite sequence $\alpha$},
i.e., the condition that for every $a\in\Omega$ if $P(a)=1$ then $\alpha$ is of the form
$\alpha=aaaaa\dotsc\dotsc$.
This intuition that \emph{an elementary event with probability one occurs certainly}
is particularly supported by the notion of probability in \emph{quantum mechanics}, as we will see in what follows.

In our former work~\cite{T16arXiv},
we confirmed the fact that an elementary event with probability one occurs certainly,
in particular, in quantum measurements over a \emph{finite-dimensional} quantum system, i.e.,
a quantum system whose state space is a finite-dimensional Hilbert space.
Note that the number of
possible
measurement outcomes is normally \emph{finite}
in measurements over a finite-dimensional quantum system.
To be specific,
projective measurements over a finite-dimensional quantum system gives a measurement outcome
from a \emph{finite} set.

In contrast,
in this paper
we consider a
stochastic
trial where the number of elementary events is \emph{countable infinite}.
Nonetheless, we can still confirm
the fact that
an elementary event with probability one occurs certainly
in quantum measurements,
even in the case
where
the number of elementary events is
\emph{infinite}.
In order to
see
this,
we consider quantum measurements over an \emph{infinite-dimensional} quantum system,
where the number of
possible
measurement outcomes is normally \emph{infinite}.

First, we recall some of the central postulates of quantum mechanics.
Due to the above reasons, we here consider the postulates of quantum mechanics
for an \emph{infinite-dimensional} quantum system, i.e.,
for
a quantum system whose state space is an infinite-dimensional Hilbert space,
in particular.
See von Neumann~\cite{vN55},
Prugove\v{c}ki~\cite{P81},
Arai and Ezawa~\cite{AE99}, Blank, Exner, and Havl\'{i}\v{c}ek~\cite{BEH08},
Hall~\cite{H13}, Teschl~\cite{Tesch14},
and Moretti~\cite{Mor17}
for the detail of the formulation of the postulates of quantum mechanics
in the infinite-dimensional case as well as the related mathematical notions and results
such as self-adjoint operators, spectral measures, and spectral theorem.

The first postulate of quantum mechanics is about \emph{state space} and \emph{state vector}.

\begin{postulate}[State space and state vector]\label{state_space}
Associated to any isolated physical system is
a
(separable complex)
Hilbert space
known as the \emph{state space} of the system.
The system is completely described by its \emph{state vector},
which is a
non-zero
vector in the system's state space.
\qed
\end{postulate}

The second postulate of quantum mechanics is about observables of quantum systems.

\begin{postulate}[Observables]\label{observables}
A physical quantity of a quantum system,
called an \emph{observable},
is described by a self-adjoint operator
on the state space of the system.
\qed
\end{postulate}

Let $H$ be a
(separable complex)
Hilbert space.
We denote by $(\cdot, \cdot)$ the inner-product defined on $H$.
The domain of definition of an operator $A$ on $H$ is denoted by $D(A)$.
We use $\mathcal{P}(H)$ to denote the set of projectors on
$H$.
The Borel class on $\R$ is denoted by $\mathcal{B}$.
Then,
in order to state the third postulate of quantum mechanics, we need
the spectral theorem below 
(see e.g.~Arai and Ezawa~\cite[Section~2.9.4]{AE99} for this form of the spectral theorom).

\begin{theorem}[The spectral theorem]\label{spectral-theorem}
For every self-adjoint operator $A$ on a Hilbert space $H$,
there exists a unique spectral measure $E\colon\mathcal{B}\to\mathcal{P}(H)$ such that
\begin{equation*}
  D(A)=\left\{\Psi\in H\middle|\int_{\R}\lambda^2 d\langle\Psi, E(\lambda)\Psi\rangle<\infty\right\}
\end{equation*}
and
\begin{equation*}
  \langle\Phi,A\Psi\rangle=\int_{\R}\lambda d\langle\Phi, E(\lambda)\Psi\rangle
\end{equation*}
for every $\Psi\in D(A)$ and $\Phi\in H$.
The spectral measure $E\colon\mathcal{B}\to\mathcal{P}(H)$ is called
the \emph{spectral measure of $A$}. 
\qed
\end{theorem}

The third postulate of quantum mechanics is about measurements on quantum systems.
This is the so-called \emph{Born rule}, i.e, \emph{the probability interpretation of the wave function}.

\begin{postulate}[The Born rule]\label{measurements}
Consider measurements of an observable of a quantum system.
Let $A$ be a self-adjoint operator describing the observable.
If the state of the quantum system is described by a state vector $\Psi$
immediately before the measurement,
then the \emph{probability} that
the measured value of the observable is found in a Borel set $J$ on $\R$
is given by
\[
  \frac{\langle\Psi,E(J)\Psi\rangle}{\langle\Psi,\Psi\rangle},
\]
where $E$ is the spectral measure of $A$. 
\qed
\end{postulate}

Postulate~\ref{measurements} describes the effects of measurements on quantum systems
using the notion of \emph{probability},
whereas it does not mention the \emph{operational definition} of the notion of probability. 
On the other hand,
there is a postulate about quantum measurements with no reference to the notion of probability.
This is given in Dirac \cite[Section 10]{D58},
and describes a spacial case of  quantum measurements which are performed upon a quantum system
in an \emph{eigenstate} of an observable, i.e., a state represented by an eigenvector of an observable.
\begin{postulate}[Dirac \cite{D58}]\label{Dirac}
If the dynamical system is in an eigenstate of a real dynamical variable $\xi$, belonging to the eigenvalue $\xi'$,
then a measurement of $\xi$ will certainly gives as result the number $\xi'$.
\qed
\end{postulate}
Here, the ``dynamical system'' means quantum system.

Based on Postulates~\ref{state_space}, \ref{observables}, \ref{measurements}, and \ref{Dirac} above,
we can show that an elementary event \emph{with probability one} occurs certainly in quantum mechanics.
To see this, let us consider a quantum system with infinite-dimensional state space,
and measurements of an observable of the quantum system described by a self-adjoint operator $A$.
Suppose that the probability that
the measured value of the observable is equal to a real $\lambda_0$ is \emph{one}
in the measurement of the observable performed upon
the system in
a state represented by a state vector $\Psi_0$.
Then, it follows from Postulate~\ref{measurements} that
\begin{equation*}
  \frac{\langle\Psi_0,E(\{\lambda_0\})\Psi_0\rangle}{\langle\Psi_0,\Psi_0\rangle}=1,
\end{equation*}
where $E$ is the spectral measure
of
$A$.
Thus, since $E(\{\lambda_0\})$ is a projector on $H$, we have that
\begin{equation}\label{E(lambda0)Psi_0=Psi_0}
  E(\{\lambda_0\})\Psi_0=\Psi_0.
\end{equation}
We here note the following theorem
(see Arai and Ezawa~\cite[Theorem~2.84 (i)]{AE99}).

\begin{theorem}\label{Arai-Ezawa-Th2-84i}
Let $A$ be a self-adjoint operator on a Hilbert space $H$, and let $E$ be
the spectral measure of $A$.
Then we have that $\{\Psi\in H\mid A\Psi=\lambda\Psi\}=\{E(\{\lambda\})\Psi\mid \Psi\in H \}$
for every real $\lambda$.
\qed
\end{theorem}

It follows from \eqref{E(lambda0)Psi_0=Psi_0} and Theorem~\ref{Arai-Ezawa-Th2-84i}
that $\Psi_0$ is an eigenvector of $A$ belonging to the eigenvalue $\lambda_0$.
Therefore, we have that immediately before the measurement, the quantum system is
in an eigenstate of the observable $A$, belonging to the eigenvalue $\lambda_0$.
While Postulate~\ref{Dirac} is mathematically vague,
it is natural to
identify
the ``real dynamical variable'' referred to in Postulate~\ref{Dirac}
with
an observable
in our terminology above.
Thus, under this identification,
it follows from Postulate~\ref{Dirac} that
the measurement of $A$ will \emph{certainly} gives as result the number $\lambda_0$.
Hence, it turns out that
\emph{an elementary event with probability one occurs certainly in quantum mechanics}.

The above consideration can be generalized
to show that
an \emph{arbitrary} event
\emph{with probability one}
occurs certainly in quantum mechanics.
To see this,
let us again consider a quantum system with infinite-dimensional state space,
and measurements of an observable of the quantum system described by a self-adjoint operator $A$.
Suppose that the probability that
the measured value of the observable is found in a Borel set $J$ on $\R$
is \emph{one}
in the measurement of the observable performed upon
the system in
a state represented by a state vector $\Psi_0$.
Then, it follows from Postulate~\ref{measurements} that
\begin{equation*}
  \frac{\langle\Psi_0,E(J)\Psi_0\rangle}{\langle\Psi_0,\Psi_0\rangle}=1,
\end{equation*}
where $E$ is the spectral measure of $A$.
Thus, since $E(J)$ is a projector on $H$, we have that
\begin{equation}\label{E(J)Psi=Psi}
  E(J)\Psi_0=\Psi_0.
\end{equation}
Now, we define a Borel function $f\colon\R\to\R$ by the condition that
$f(\lambda)$ is $1$ if $\lambda\in J$ and $0$ otherwise.
Then, we can define a operator $f(A)$ such that
\begin{enumerate}
\item
\begin{equation*}
  D(f(A))
  =\left\{\Psi\in H\middle|\int_{\R}\abs{f(\lambda)}^2 d\langle\Psi, E(\lambda)\Psi\rangle<\infty\right\}
\end{equation*}
and
\begin{equation*}
  \langle\Phi,f(A)\Psi\rangle=\int_{\R}f(\lambda) d\langle\Phi, E(\lambda)\Psi\rangle
\end{equation*}
for every $\Psi\in D(A)$ and $\Phi\in H$,
and
\item
if $\Psi$ is an eigenvector of $A$ belonging to an eigenvalue $\lambda$
then $\Psi$ is an eigenvector of $f(A)$ belonging to the eigenvalue $f(\lambda)$.
\end{enumerate}
See e.g.~Prugove\v{c}ki~\cite[Chapter~2]{P81} or Arai and Ezawa~\cite[Chapter~2]{AE99}
for the detail of the definition and property of the operator $f(A)$.
It is then easy to show that
\begin{equation}\label{f(A)=E(J)}
  f(A)=E(J).
\end{equation}
Thus, $f(A)$ is a self-adjoint operator.
It describes \emph{an observable whose measurement is done
by first performing the measurement of the observable described by $A$,
and then simply applying the function $f$ to the measured value}.
On the other hand, it follows from \eqref{E(J)Psi=Psi} and \eqref{f(A)=E(J)} that
$\Psi_0$ is an eigenvector of $f(A)$ belonging to the eigenvalue $1$.
Therefore, we have that immediately before the measurement, the quantum system is
in an eigenstate of the observable
$f(A)$, belonging to the eigenvalue $1$.
Thus,
applying Postulate~\ref{Dirac} under the identification of
the ``real dynamical variable'' referred to in Postulate~\ref{Dirac}
with the observable described by $f(A)$ as above,
we have that
the measurement of $f(A)$ will \emph{certainly} gives as result the number $1$.
Since $f(\lambda)=1$ if and only if $\lambda\in J$,
this can be rephrased as that
the measurement of $A$ will \emph{certainly} gives as result a number in $J$.
Hence, it turns out that
\emph{an event with probability one occurs certainly in quantum mechanics}.

Theorem~\ref{one_probability-Bs} below states that
an elementary event with probability one always occurs in an ensemble,
and thus shows that the notion of ensemble coincides with our intuition about
the notion of probability, in particular, in quantum mechanics.

\begin{theorem}\label{one_probability-Bs}
Let $\Omega$ be an r.e.~infinite set, and let $P\in\PS(\Omega)$.
Let $a\in\Omega$.
Suppose that $\alpha$ is an ensemble for the discrete probability space $P$ and $P(a)=1$.
Then $\alpha$ consists only of $a$,
i.e.,
$\alpha=aaaaaa\dotsc\dotsc$.
\qed
\end{theorem}

Theorem~\ref{one_probability-Bs} follows immediately from a more general result,
Theorem~\ref{zero_probability-Bs} below,
which states that an elementary event with probability zero never occurs in an ensemble.

\begin{theorem}\label{zero_probability-Bs}
Let $\Omega$ be an r.e.~infinite set,
and let $P\in\PS(\Omega)$.
  Let $a\in\Omega$.
  Suppose that $\alpha$ is an ensemble for the discrete probability space $P$ and $P(a)=0$.
  Then $\alpha$ does not contain $a$.
\end{theorem}

\begin{proof}
We define $\mathcal{C}$
as the set $\{(n,\rho a)\mid n\in\N^+\;\&\;\rho\in(\Omega\setminus \{a\})^*\}$.
Then $\mathcal{C}_n$ is a prefix-free subset of $\Omega^*$ for every $n\in\N^+$.
Since $\Omega$ is an r.e.~set,
the set $\mathcal{C}$ is an r.e.~subset of $\N^+\times \Omega^*$.
For each $n\in\N^+$,
since $P(\sigma)=0$ for every $\sigma\in\mathcal{C}_n$,
we have $\Bm{P}{\osg{\mathcal{C}_n}}=P(\mathcal{C}_n)=0<2^{-n}$.
Hence, $\mathcal{C}$ is Martin-L\"of $P$-test.

Since $\alpha$ is Martin-L\"of $P$-random, it passes $\mathcal{C}$.
Let us assume contrarily that $\alpha$ contains $a$.
Then there exists a prefix $\sigma$ of $\alpha$
such that $\sigma=\rho a$ for some $\rho\in(\Omega\setminus \{a\})^*$.
Since $\sigma\in\mathcal{C}_n$ for all $n\in\N^+$,
we have that $\alpha\in\osg{\mathcal{C}_n}$ for all $n\in\N^+$.
Therefore $\alpha$ does not pass $\mathcal{C}$.
Thus,
we have a contradiction, and the proof is completed.
\end{proof}

\subsection{The law of large numbers}

Let $P\in\PS(\Omega)$, and
let us consider
an infinite sequence $\alpha\in\Omega^\infty$ of outcomes
which is being generated by infinitely repeated trials described
by the discrete probability space $P$ on $\Omega$.
The second necessary condition which the notion of probability for the discrete probability space $P$
is considered to have to satisfy is the condition that
\emph{the law of large numbers holds for $\alpha$}.
Theorem~\ref{LLN-Bs} below confirms that this certainly holds.
Note here that we have to prove that the law of large numbers holds for $\alpha$
even in the case where $P$ is \emph{not computable}.
This is because a discrete probability space is not computable, in general.
However, we can certainly prove it, as shown in Theorem~\ref{LLN-Bs}.

\begin{theorem}[The law of large numbers]\label{LLN-Bs}
Let $\Omega$ be an r.e.~infinite set, and let $P\in\PS(\Omega)$.
For every $\alpha\in\Omega^\infty$, if $\alpha$ is an ensemble for $P$
then for every $a\in\Omega$ it holds that
\begin{equation*}
  \lim_{n\to\infty}\frac{N_a(\rest{\alpha}{n})}{n}=P(a),
\end{equation*}
where $N_a(\sigma)$ denotes the number of the occurrences of $a$ in $\sigma$ for every $a\in\Omega$ and $\sigma\in\Omega^*$.
\qed
\end{theorem}

In order to prove Theorem~\ref{LLN-Bs}, we need Theorem~\ref{LLN} below,
which is Theorem~11 of Tadaki~\cite{T16arXiv}.

\begin{theorem}[The law of large numbers, Tadaki~\cite{T14,T15,T16arXiv}]\label{LLN}
Let $\Theta$ be
a finite
alphabet, and let $Q\in\PS(\Theta)$.
For every $\alpha\in\Theta^\infty$, if $\alpha$ is an ensemble for $Q$
then for every $a\in\Theta$ it holds that
\begin{equation*}
  \lim_{n\to\infty}\frac{N_a(\rest{\alpha}{n})}{n}=Q(a),
\end{equation*}
where $N_a(\sigma)$ denotes the number of the occurrences of $a$ in $\sigma$
for every $a\in\Theta$ and $\sigma\in\Theta^*$.
\qed
\end{theorem}

In order to prove Theorem~\ref{LLN-Bs}, we also need the following theorem.

\begin{theorem}\label{contraction-Bs}
Let $\Omega$ be an r.e.~infinite set, and let $P\in\PS(\Omega)$.
Let $A_1,\dots,A_L$ be r.e.~subsets of
$\Omega$
such that
$\Omega=\bigcup_{i=1}^L A_i$ and $A_i\cap A_j=\emptyset$ for every $i\neq j$.
Let $\Theta=\{a_1,\dots,a_L\}$ be
a finite
alphabet such that $a_i\neq a_j$ for every $i\neq j$.
Suppose that $\alpha$ is an ensemble for $P$.
Let $\beta$ be an infinite sequence over $\Theta$
obtained by replacing all occurrences of elements of $A_i$ in $\alpha$ by $a_i$ for each $i=1,\dots,L$.
Then $\beta$ is an ensemble for $Q$,
where $Q\in\PS(\Theta)$ such that
$Q(a_i):=P(A_i)$
for every $i=1,\dots,L$.
\end{theorem}

\begin{proof}
We show the contraposition. Suppose that $\beta$ is not Martin-L\"of $Q$-random.
Then there exists a Martin-L\"of $Q$-test $\mathcal{T}\subset\N^+\times\Theta^*$ such that
\begin{equation}\label{betainosgmTn_LemmaLLN}
  \beta\in\osg{\mathcal{T}_n}
\end{equation}
for every $n\in\N^+$.
For each $\tau\in\Theta^*$, let $f(\tau)$ be the set of all $\sigma\in\Omega^*$ such that,
when replacing all occurrences of elements of $A_i$ in $\sigma$ by $a_i$ for each $i=1,\dots,L$,
the resulting finite string equals to $\tau$.
Then, since $Q(a_i)=\sum_{a\in A_i}P(a)$ for every $i=1,\dots,L$, we have that
\begin{equation}\label{BmQ=Q=P=BmP_LemmaLLN}
  \Bm{Q}{\osg{\tau}}=Q(\tau)=P(f(\tau))=\Bm{P}{\osg{f(\tau)}}
\end{equation}
for each $\tau\in\Theta^*$.
We then define $\mathcal{S}$ to be a subset of $\N^+\times \Omega^*$ such that
$\mathcal{S}_n=\bigcup_{\tau\in\mathcal{T}_n} f(\tau)$ for every $n\in\N^+$.
Since $\mathcal{T}_n$ is a prefix-free subset of $\Theta^*$ for every $n\in\N^+$,
we see that $\mathcal{S}_n$ is a prefix-free subset of $\Omega^*$ for every $n\in\N^+$.
For each $n\in\N^+$, we also see that
\[
  \Bm{P}{\osg{\mathcal{S}_n}}
  \le\sum_{\tau\in\mathcal{T}_n}\Bm{P}{\osg{f(\tau)}}
  =\sum_{\tau\in\mathcal{T}_n}\Bm{Q}{\osg{\tau}}
  =\Bm{Q}{\osg{\mathcal{T}_n}}<2^{-n},
\]
where the first equality follows from \eqref{BmQ=Q=P=BmP_LemmaLLN} and
the second equality follows from the prefix-freeness of $\mathcal{T}_n$.
Moreover, since all of $A_1,\dots,A_L$, and $\mathcal{T}$ are r.e., $\mathcal{S}$ is also r.e.
Thus, $\mathcal{S}$ is a Martin-L\"of $P$-test.

On the other hand,
note that, for every $n\in\N^+$, if $\beta\in\osg{\mathcal{T}_n}$ then $\alpha\in\osg{\mathcal{S}_n}$.
Thus, it follows from \eqref{betainosgmTn_LemmaLLN}
that $\alpha\in\osg{\mathcal{S}_n}$ for every $n\in\N^+$.
Hence, $\alpha$ is not Martin-L\"of $P$-random.
This completes the proof.
\end{proof}

Theorem~\ref{LLN-Bs} is then proved as follows.

\begin{proof}[Proof of Theorem~\ref{LLN-Bs}]
Let $a\in\Omega$.
We define $Q\in\PS(\{0,1\})$ by the condition that $Q(1)=P(a)$ and $Q(0)=1-P(a)$.
Let $\beta$ be the infinite binary sequence obtained from $\alpha$
by replacing all $a$ by $1$ and all other elements of $\Omega$ by $0$ in $\alpha$.
Note that $\Omega\setminus\{a\}$ is r.e., since $\Omega$ is r.e.
Thus, since $Q(0)=\sum_{x\in\Omega\setminus\{a\}}P(x)$,
it follows from Theorem~\ref{contraction-Bs} that $\beta$ is Martin-L\"of $Q$-random.
On the other hand, obviously, we have that
$N_1(\rest{\beta}{n})=N_a(\rest{\alpha}{n})$ for every $n\in\N^+$.
Thus, using Theorem~\ref{LLN} we have that
\begin{equation*}
  \lim_{n\to\infty}\frac{N_a(\rest{\alpha}{n})}{n}=\lim_{n\to\infty}\frac{N_1(\rest{\beta}{n})}{n}=Q(1)=P(a).
\end{equation*}
This completes the proof.
\end{proof}

The following is immediate from Theorem~\ref{LLN-Bs}.

\begin{corollary}\label{uniquness}
Let $\Omega$ be an r.e.~infinite set, and let $P,Q\in\PS(\Omega)$.
If there exists $\alpha\in\Omega^\infty$ which is both
an ensemble for $P$ and an ensemble for $Q$,
then $P=Q$.
\qed
\end{corollary}

\subsection{Computable shuffling}

This subsection considers the third necessary condition which
the notion of probability for a discrete probability space is considered to have to satisfy.

Let $P\in\PS(\Omega)$.
Assume that an observer $A$ performs an infinite reputation of trials described by
the discrete probability space $P$,
and thus is generating an infinite sequence $\alpha\in\Omega^\infty$ of outcomes of
the trials as
$$\alpha=a_1 a_2 a_3 a_4 a_5 a_6 a_7 a_8 \dotsc\dotsc$$
with $a_i\in\Omega$. 
According to our thesis, Thesis~\ref{thesis-Bs},
$\alpha$ is an ensemble for $P$.
Consider another observer $B$ who wants to adopt the following subsequence $\beta$ of $\alpha$
as the outcomes of
the trials:
$$\beta=a_2 a_3 a_5 a_7 a_{11} a_{13} a_{17} \dotsc\dotsc,$$
where the observer $B$ only takes into account
the $n$th elements $a_n$ in the original sequence $\alpha$
such that $n$ is a prime number.
According to Thesis~\ref{thesis-Bs},
$\beta$ has to be an ensemble for $P$, as well.
However, is this true?

Consider this problem in a general setting.
Assume as before that an observer $A$ performs an infinite reputation of trials described
by the discrete probability space $P$,
and thus is generating an infinite sequence $\alpha\in\Omega^\infty$ of outcomes of
the trials.
According to Thesis~\ref{thesis-Bs},
$\alpha$ is an ensemble for $P$.
Now, let $f\colon\N^+\to\N^+$ be an injection.
Consider another observer $B$ who wants to adopt the following sequence $\beta$
as the outcomes of
the trials:
$$\beta=\alpha(f(1))\alpha(f(2))\alpha(f(3))\alpha(f(4))\alpha(f(5))\dotsc\dotsc$$
instead of $\alpha$.
According to Thesis~\ref{thesis-Bs}, $\beta$ has to be an ensemble for $P$, as well.
However, is this true?

We can confirm this \emph{by restricting the ability of $B$}, that is, by assuming that
every observer can select elements from the original sequence $\alpha$
\emph{only in an effective manner}.
This means that the function $f\colon\N^+\to\N^+$ has to be a computable function.
Theorem~\ref{cpucs-Bs} below shows this result.

In other words, Theorem~\ref{cpucs-Bs} states that ensembles for $P$ are
\emph{closed under computable shuffling}.

\begin{theorem}[Closure property under computable shuffling%
]\label{cpucs-Bs}
Let $\Omega$ be an r.e.~infinite set, and let $P\in\PS(\Omega)$.
Suppose that $\alpha$ is an ensemble for $P$.
Then, for every injective function $f\colon\N^+\to\N^+$, if $f$ is computable then the infinite sequence
\begin{equation*}
  \alpha_f:=\alpha(f(1)) \alpha(f(2)) \alpha(f(3)) \alpha(f(4)) \dotsc\dotsc\dotsc
\end{equation*}
is an ensemble for $P$.
\end{theorem}

\begin{proof}
We show the contraposition.
Suppose that $\alpha_f$ is not Martin-L\"of $P$-random.
Then there exists a Martin-L\"of $P$-test $\mathcal{C}\subset\N^+\times \Omega^*$ such that
\begin{equation}\label{afinosgmCn_CS}
  \alpha_f\in\osg{\mathcal{C}_n}
\end{equation}
for every $n\in\N^+$.
For each $\sigma\in\Omega^+$, let $F(\sigma)$ be the set of all $\tau\in\Omega^+$ such that
\begin{enumerate}
  \item $\abs{\tau}=\max f(\{1,2,\dots,\abs{\sigma}\})$, and
  \item for every $k=1,2,\dots,\abs{\sigma}$ it holds that $\sigma(k)=\tau(f(k))$.
\end{enumerate}
Then, since $f$ is an injection and $\sum_{a\in\Omega}P(a)=1$, we have that
\begin{equation}\label{BPoFslePFs=Ps=BPos_CS}
  \Bm{P}{\osg{F(\sigma)}}=P(F(\sigma))=P(\sigma)=\Bm{P}{\osg{\sigma}}
\end{equation}
for each $\sigma\in\Omega^+$.
We then define $\mathcal{D}$ to be a subset of $\N^+\times \Omega^*$ such that
$\mathcal{D}_n=\bigcup_{\sigma\in\mathcal{C}_n} F(\sigma)$ for every $n\in\N^+$.
Note here that, for each $n\in\N^+$,
$\lambda\notin\mathcal{C}_n$ since $\Bm{P}{\osg{\mathcal{C}_n}}<2^{-n}<1$.
Then, since $\mathcal{C}_n$ is a prefix-free subset of $\Omega^*$ for every $n\in\N^+$,
we see that $\mathcal{D}_n$ is also a prefix-free subset of $\Omega^*$ for every $n\in\N^+$.
For each $n\in\N^+$, we see that
\[
  \Bm{P}{\osg{\mathcal{D}_n}}
  \le\sum_{\sigma\in\mathcal{C}_n}\Bm{P}{\osg{F(\sigma)}}
  =\sum_{\sigma\in\mathcal{C}_n}\Bm{P}{\osg{\sigma}}
  =\Bm{P}{\osg{\mathcal{C}_n}}<2^{-n},
\]
where the first equality follows from \eqref{BPoFslePFs=Ps=BPos_CS} and
the second equality follows from the prefix-freeness of $\mathcal{C}_n$.
Moreover,
since $f$ is an injective computable function and both $\Omega$ and $\mathcal{C}$ are r.e.,
it is easy to see that $\mathcal{D}$ is r.e.
Thus, $\mathcal{D}$ is a Martin-L\"of $P$-test.

On the other hand, we see that,
for every $n\in\N^+$, if $\alpha_f\in\osg{\mathcal{C}_n}$ then $\alpha\in\osg{\mathcal{D}_n}$.
Thus, it follows from \eqref{afinosgmCn_CS} that
$\alpha\in\osg{\mathcal{D}_n}$ for every $n\in\N^+$.
Hence, $\alpha$ is not Martin-L\"of $P$-random.
This completes the proof.
\end{proof}

\subsection{Selection by
partial computable
generalized
selection functions}

As the forth necessary condition which
the notion of probability for a discrete probability space $P$ on $\Omega$ is considered to have to satisfy,
in this subsection
we consider the condition that infinite sequences over $\Omega$ of outcomes each of which
is obtained by an infinite reputation of the trials described by the discrete probability space $P$
are \emph{closed under the selection by a partial computable
generalized
selection function on $\Omega^*$},
which is a
\emph{generalization} of
the notion of partial computable selection function
used in the definition of \emph{von Mises-Wald-Church stochasticity}
over an \emph{r.e.~infinite alphabet}.
The notion of von Mises-Wald-Church stochasticity \emph{itself}
is investigated in the theory of collectives \cite{vM57,vM64,Wa36,Wa37,Ch40}.%
\footnote{%
See Downey and Hirschfeldt~\cite[Section 7.4]{DH10}
for a treatment of the mathematics of the notion of von Mises-Wald-Church stochasticity
\emph{itself}
from a modern point of view.}
For motivating the forth necessary condition,
we carry out a thought experiment
in what follows,
as in the preceding subsection.

Let $P\in\PS(\Omega)$,
and let us assume that an observer $A$ performs an infinite reputation of trials described by
the discrete probability space $P$,
and thus is generating an infinite sequence $\alpha\in\Omega^\infty$ of outcomes of
the trials as
$$\alpha=a_1 a_2 a_3 a_4 a_5 a_6 \dotsc\dotsc$$
with $a_i\in\Omega$.
According to Thesis~\ref{thesis-Bs},
$\alpha$ is an ensemble for $P$.

Consider another observer $B$ who wants to \emph{refute} Thesis~\ref{thesis-Bs}.
For that purpose, the observer $B$ adopts a subsequence
$\beta=b_1 b_2 b_3 b_4 \dotsc\dotsc$ with $b_i\in\Omega$ of $\alpha$ in the following manner:
Whenever a new outcome $a_n$ is generated by the observer $A$,
the observer $B$
investigates
the prefix $a_1 a_2 a_3 \dots a_n$ of $\alpha$ generated so far
by the observer $A$.
Then, based on the prefix, the observer $B$ decides whether the next outcome $a_{n+1}$ should be appended
to the tail of $b_1 b_2 b_3 \dots b_k$
which have been adopted so far by $B$ as a prefix of $\beta$.
In this manner the observer $B$ is generating  the subsequence $\beta$ of $\alpha$.
Note that the length of $\beta$ may or may not be infinite.

On the other hand, the observer $A$ is a \emph{defender} of Thesis~\ref{thesis-Bs}.
Therefore, the observer $A$ tries to inhibit the observer $B$ from breaking Thesis~\ref{thesis-Bs}.
For that purpose, the observer $A$ never generates
the next
outcome $a_{n+1}$
before the observer $B$ decides whether
this
$a_{n+1}$ should be appended
to the tail of $b_1 b_2 b_3 \dots b_k$.
This is because if for each $n$ the observer $B$ knows the outcome $a_{n+1}$
before the decision for $a_{n+1}$ to be appended or to be ignored,
then the observer $B$ can easily generate an infinite subsequence $\beta$ of $\alpha$
which does not satisfy Thesis~\ref{thesis-Bs}.
Thus, due to this careful behavior of the observer $A$,
the observer $B$ has to make the decision of the choice of the next outcome $a_{n+1}$,
based only on the prefix $a_1 a_2 a_3 \dots a_n$ of $\alpha$, without knowing the outcome $a_{n+1}$.
Then, according to Thesis~\ref{thesis-Bs},
$\beta$ has to be an ensemble for $P$, as well as $\alpha$.
However, is this true?

We can confirm this by restricting the ability of $B$, that is, by assuming that
the observer $B$ can make the decision of the choice of the next outcome,
\emph{only in an effective manner}
based on the prefix $a_1 a_2 a_3 \dots a_n$ of $\alpha$ generated so far by the observer $A$.

Put more mathematically,
we introduce some notations.
A \emph{generalized selection function}
is
a function $f$ such that
$\Dom f\subset\Omega^*$ and $f(\Dom f)\subset\{\mathrm{YES}, \mathrm{NO}\}$.
We think of $f$ as the decision of $B$ whether or not to choose the next outcome $\alpha(n+1)$
based on the prefix
$\rest{\alpha}{n}$
of $\alpha$ in generating $\beta$.
For any $\gamma\in\Omega^\infty$, $k\in\N^+$, and
generalized
selection function $g$,
let $s_g (\gamma, k)$ be the $k$th number $\ell\in\N$ such that $g(\rest{\gamma}{\ell})=\mathrm{YES}$,
i.e.,
the least number $\ell\in\N$ such that $\#\{m\le\ell\mid g(\rest{\gamma}{m})=\text{YES}\}=k$,
if such $\ell$ exists.

First, consider the case where $f(\rest{\alpha}{n})$ is not defined for some $n\in\N$.
Let $m$ be the least number of such $n$.
Then, this case means that
the observer $B$ does not make the decision of the choice of the next outcome $\alpha(m+1)$
based on the prefix $\rest{\alpha}{m}$, and is stalled.
Therefore, the length of $\beta$
remains
finite in this case.
Thus, the observer $B$ \emph{cannot refute} Thesis~\ref{thesis-Bs} in this case,
since Thesis~\ref{thesis-Bs} only refers to the property of an \emph{infinite} sequence of outcomes
which is being generated by \emph{infinitely} repeated trials.
Hence, Thesis~\ref{thesis-Bs} survives in this case.

Secondly, consider the case where
$f(\rest{\alpha}{n})$ is defined for all $n\in\N$ and
$\{n\in\N\mid f(\rest{\alpha}{n})=\mathrm{YES}\}$ is a finite set.
In this case, the length of $\beta$
remains also finite.
Thus, the observer $B$ does not refute Thesis~\ref{thesis-Bs}, and therefore
Thesis~\ref{thesis-Bs} survives also in this case.

Finally, consider the remaining case, where $f(\rest{\alpha}{n})$ is defined for all $n\in\N$ and
the set $\{n\in\N\mid f(\rest{\alpha}{n})=\mathrm{YES}\}$ is infinite.
Then, $s_f (\alpha, k)$ is defined and $\beta(k)=\alpha(s_f (\alpha, k)+1)$ for all $k\in\N^+$.
Hence, $\beta$ is an infinite sequence over $\Omega$, 
and thus Thesis~\ref{thesis-Bs} can be applied to $\beta$ in this case.
Therefore,
according to Thesis~\ref{thesis-Bs}, $\beta$ has to be an ensemble for $P$, as well as $\alpha$.
However, is this true?
Actually, we can confirm this \emph{by restricting the ability of $B$}, that is,
\emph{by assuming that
$f$ has to be a partial computable
generalized
selection function}.
Here, a \emph{partial computable
generalized
selection function} is a
generalized
selection function which is
a partial
recursive
function.
Theorem~\ref{cpuscsf-Bs} below shows this result.
It states that ensembles for an arbitrary discrete probability space are
\emph{closed under the selection by a partial computable
generalized
selection function}.
Hence, Thesis~\ref{thesis-Bs} survives in this case as well.

In this way, based on Theorem~\ref{cpuscsf-Bs}, we confirm that
the forth condition certainly holds for ensembles for an arbitrary discrete probability space.

\begin{theorem}[Closure property under the selection by a partial computable
generalized
selection function%
]\label{cpuscsf-Bs}
Let $\Omega$ be an r.e.~infinite set, and let $P\in\PS(\Omega)$.
Suppose that $\alpha$ is an ensemble for $P$.
Let $f $ be a partial computable
generalized
selection function with $\Dom f\subset\Omega^*$.
Suppose that $f(\rest{\alpha}{k})$ is defined for all $k\in\N$ and
$\{k\in\N\mid f(\rest{\alpha}{k})=\mathrm{YES}\}$ is an infinite set. Then
an infinite sequence $\beta$ such that $\beta(k)=\alpha(s_f (\alpha, k)+1)$ for all $k\in\N^+$
is an ensemble for $P$.
\end{theorem}

\begin{proof}
We show the contraposition.
Suppose that $\beta$ is not Martin-L\"of $P$-random.
Then there exists a Martin-L\"of $P$-test $\mathcal{C}\subset\N^+\times \Omega^*$ such that
\begin{equation}\label{betainosgmathCn_CPUSPCF}
  \beta\in\osg{\mathcal{C}_n}
\end{equation}
for every $n\in\N^+$.
For
any
$\sigma,\tau\in\Omega^+$, we say that \emph{$\sigma$ is selected by $f$ from $\tau$} if
$f(\rest{\tau}{k})$ is defined for all $k=0,1,\dots,\abs{\tau}-1$ and
there exists a strictly increasing function $h\colon\{1,\dots,\abs{\sigma}\}\to\N$ such that
\begin{enumerate}
  \item $\{k\in\{1,\dots,\abs{\tau}\}\mid f(\rest{\tau}{k-1})=\mathrm{YES}\}=h(\{1,\dots,\abs{\sigma}\})$,
  \item $h(\abs{\sigma})=\abs{\tau}$, and
  \item $\tau(h(k))=\sigma(k)$ for all $k=1,\dots,\abs{\sigma}$.
\end{enumerate}
For each $\sigma\in\Omega^+$, let $F(\sigma)$ be the set of all $\tau\in\Omega^*$ such that
$\sigma$ is selected by $f$ from $\tau$.
We also set $F(\lambda):=\{\lambda\}$.
It is then easy to see that $F(\sigma)$ is a prefix-free
subset of $\Omega^*$
for every $\sigma\in\Omega^*$.

We show that
\begin{equation}\label{PFleP-Bs}
  \Bm{P}{\osg{F(\sigma)}}\le\Bm{P}{\osg{\sigma}}
\end{equation}
for all $\sigma\in\Omega^*$ by the induction on the length of $\abs{\sigma}$.
First, the inequality \eqref{PFleP-Bs} holds for the case of $\abs{\sigma}=0$, obviously.
For an arbitrary $n\in\N$, assume that \eqref{PFleP-Bs} holds for all $\sigma\in\Omega^n$.
Let $\sigma\in\Omega^{n+1}$.
We then denote the prefix of $\sigma$ of length $n$ by $\rho$, and denote $\sigma(\abs{\sigma})$ by $a$.
Therefore $\sigma=\rho a$.
Note that
$$G(\tau):=\{\upsilon\in\Omega^*\mid \tau\upsilon a\in F(\sigma)\}$$
is a prefix-free
subset of $\Omega^*$
for every $\tau\in\Omega^*$.
Therefore, we have that
\begin{equation}\label{lPGosgtle1}
  \sum_{\upsilon\in G(\tau)}\Bm{P}{\osg{\upsilon}}=\Bm{P}{\osg{G(\tau)}}\le 1
\end{equation}
for each $\tau\in\Omega^*$.
Thus,
for each $\sigma\in\Omega^*$,
we see that
\begin{align*}
  \Bm{P}{\osg{F(\sigma)}}
  &=\sum_{\nu\in F(\sigma)}\Bm{P}{\osg{\nu}}
  =\sum_{\tau\in F(\rho)}\sum_{\upsilon\in G(\tau)}\Bm{P}{\osg{\tau\upsilon a}} \\
  &=\sum_{\tau\in F(\rho)}\sum_{\upsilon\in G(\tau)}\Bm{P}{\osg{\tau}}\Bm{P}{\osg{\upsilon}}P(a) \\
  &\le\sum_{\tau\in F(\rho)}\Bm{P}{\osg{\tau}}P(a)
  =\Bm{P}{\osg{F(\rho)}}P(a) \\
  &\le\Bm{P}{\osg{\rho}}P(a)
  =\Bm{P}{\osg{\sigma}},
\end{align*}
where the second equality follows from the fact that the mapping
$$\{(\tau,\upsilon)\mid\tau\in F(\rho)\;\&\; \upsilon\in G(\tau)\}\ni(\tau,\upsilon)\mapsto
\tau \upsilon a\in F(\sigma)$$
is a bijection,
the first inequality follows from \eqref{lPGosgtle1},
and the second inequality follows from the assumption.
Therefore \eqref{PFleP-Bs} holds for all $\sigma\in\Omega^{n+1}$.
Hence, \eqref{PFleP-Bs} holds for all $\sigma\in\Omega^*$, as desired.

We then define $\mathcal{D}$ to be a subset of $\N^+\times \Omega^*$ such that
$\mathcal{D}_n=\bigcup_{\sigma\in \mathcal{C}_n}F(\sigma)$ for every $n\in\N^+$.
Since $\mathcal{C}_n$ is a prefix-free subset of $\Omega^*$ for every $n\in\N^+$,
we see that $\mathcal{D}_n$ is also a prefix-free subset of $\Omega^*$ for every $n\in\N^+$.
For each $n\in\N^+$, we see that
\begin{equation*}
  \Bm{P}{\osg{\mathcal{D}_n}}\le\sum_{\sigma\in \mathcal{C}_n}\Bm{P}{\osg{F(\sigma)}}
  \le\sum_{\sigma\in \mathcal{C}_n}\Bm{P}{\osg{\sigma}}
  =\Bm{P}{\osg{\mathcal{C}_n}}<2^{-n},
\end{equation*}
where the second inequality follows from \eqref{PFleP-Bs} and
the equality follows from the prefix-freeness of $\mathcal{C}_n$.
Moreover, since $\Omega$ and $\mathcal{C}$ are r.e.,
we see
that $\mathcal{D}$ is also r.e.
Thus, $\mathcal{D}$ is a Martin-L\"of $P$-test.

On the other hand, we see that,
for every $n\in\N^+$, if $\beta\in\osg{\mathcal{C}_n}$ then $\alpha\in\osg{\mathcal{D}_n}$.
Thus, it follows from \eqref{betainosgmathCn_CPUSPCF} that
$\alpha\in\osg{\mathcal{D}_n}$ for every $n\in\N^+$.
Hence, $\alpha$ is not Martin-L\"of $P$-random.
This completes the proof.
\end{proof}

Theorem~\ref{cpucs-Bs} and Theorem~\ref{cpuscsf-Bs} show that
certain closure properties hold for ensembles for an arbitrary discrete probability space.
In the subsequent sections,
we will see that various strong closure properties of another type hold for the ensembles.

\section{Conditional probability and the independence between two events}
\label{CPITW}

In this section we operationally characterize the notions of \emph{conditional probability} and
the \emph{independence between two events} on a discrete probability space,
in terms of ensembles.

Let $\Omega$ be a countable alphabet, and let $P\in\PS(\Omega)$.
Let $A\subset\Omega$ be
an
event on the discrete probability space $P$.
For each ensemble $\alpha$ for $P$,
we use $\chara{A}{\alpha}$ to denote the infinite binary sequence such that,
for every $n\in\N^+$,
its $n$th element $(\chara{A}{\alpha})(n)$ is $1$ if $\alpha(n)\in A$ and $0$ otherwise.
The pair $(P,A)$ induces a finite probability space $\charaps{P}{A}\in\PS(\{0,1\})$ such that
$(\charaps{P}{A})(1)=P(A)$ and $(\charaps{P}{A})(0)=1-P(A)$.
Note that the notions of $\chara{A}{\alpha}$ and $\charaps{P}{A}$ in our theory together correspond to
the notion of \emph{mixing} in the theory of collectives by von Mises \cite{vM64}.
We can then show the following theorem.

\begin{theorem}\label{charaA-Bs}
Let $\Omega$ be an r.e.~infinite set, and let $P\in\PS(\Omega)$.
Let $A\subset\Omega$ be a recursive event on the discrete probability space $P$.
Suppose that $\alpha$ is an ensemble for the discrete probability space $P$.
Then $\chara{A}{\alpha}$ is an ensemble for the finite probability space $\charaps{P}{A}$.
\end{theorem}

\begin{proof}
We show the result using Theorem~\ref{contraction-Bs}.
First, since $\Omega$ is r.e.~and $A$ is a recursive subset of $\Omega$,
both $\Omega\setminus A$ and $A$ are r.e.
Obviously, we have $\Omega=(\Omega\setminus A)\cup A$ and
$(\Omega\setminus A)\cap A=\emptyset$.
Note that
$\chara{A}{\alpha}$ is an infinite sequence over $\{0,1\}$
obtained both by replacing all occurrences of elements of $\Omega\setminus A$ in $\alpha$ by $0$ and
by replacing all occurrences of elements of $A$ in $\alpha$ by $1$.
Note, moreover,
that $\charaps{P}{A}$ is a finite probability space on $\{0,1\}$ such that
$(\charaps{P}{A})(0)=P(\Omega\setminus A)$ and $(\charaps{P}{A})(1)=P(A)$.
Thus,
it follows from Theorem~\ref{contraction-Bs} that
$\chara{A}{\alpha}$ is an ensemble for $\charaps{P}{A}$. 
\end{proof}

We show that the notion of conditional probability in a discrete probability space can be
represented by an ensemble in a natural manner.
For that purpose, first we recall the notion of conditional probability in a discrete probability space.

Let $\Omega$ be a countable alphabet, and let $P\in\PS(\Omega)$.
Let $B\subset\Omega$ be an event on the discrete probability space $P$.
Suppose that $P(B)>0$.
Then, for each event $A\subset\Omega$,
the \emph{conditional probability of A given B}, denoted
$P(A|B)$, is defined as $P(A\cap B)/P(B)$.
This notion defines a finite
or discrete
probability space $P_B\in\PS(B)$ such that
$P_B(a)=P(\{a\}|B)$ for every $a\in B$.

When an infinite sequence $\alpha\in\Omega^\infty$ contains infinitely many elements from $B$,
$\cond{B}{\alpha}$ is defined as an infinite sequence in $B^\infty$ obtained from $\alpha$
by eliminating all elements
of
$\Omega\setminus B$ occurring in $\alpha$.
If $\alpha$ is an ensemble for the discrete probability space $P$ and $P(B)>0$,
then $\alpha$ contains infinitely many elements from $B$ due to Theorem~\ref{LLN-Bs}.
Therefore, $\cond{B}{\alpha}$ is
properly
defined in this case.
Note that the notion of $\cond{B}{\alpha}$ in our theory corresponds to
the notion of \emph{partition} in the theory of collectives by von Mises \cite{vM64}.

We can then show Theorem~\ref{conditional_probability-Bs} below, which states that ensembles are
\emph{closed under conditioning}.

\begin{theorem}[Closure property under conditioning]\label{conditional_probability-Bs}
Let $\Omega$ be an r.e.~infinite set, and let $P\in\PS(\Omega)$.
Let $B\subset\Omega$ be a recursive event on the discrete probability space $P$ with $P(B)>0$.
For every ensemble $\alpha$ for $P$,
it holds that $\cond{B}{\alpha}$ is an ensemble for
$P_B$.
\end{theorem}

\begin{proof}
In the case of $B=\Omega$, we have $P_B=P$ and $\cond{B}{\alpha}=\alpha$.
Therefore the result is obvious.
Thus, in what follows, we assume that
$B$ is a proper subset of $\Omega$.
In what follows, we further assume
that $B$ is an infinite set.
The case
in which
$B$ is a finite subset of $\Omega$ can be handled more easily
by simplifying the proof given below,
and thus we omit the proof
for
such a case.

First, we choose
any particular
$a\in\Omega\setminus B$ and
define
$Q\in\PS(B\cup\{a\})$ by the condition that
$Q(x):=\sum_{y\in\Omega\setminus B}P(y)$ if $x=a$ and $Q(x):=P(x)$ otherwise.
Note here that
\begin{equation}\label{1-Q=PB}
 1-Q(a)=P(B),
\end{equation}
and therefore
\begin{equation}\label{Q<1}
  Q(a)<1.
\end{equation}
Let $\beta$ be the infinite sequence over $B\cup\{a\}$
obtained by replacing all occurrences of elements of $\Omega\setminus B$ in $\alpha$ by $a$.
Since $\alpha$ is Martin-L\"of $P$-random and $\Omega\setminus B$ is r.e.,
in a similar manner to the proof of Theorem~\ref{contraction-Bs} we can show that
$\beta$ is Martin-L\"of $Q$-random.
Hence,
in order to complete the proof,
it is sufficient to show that if $\cond{B}{\alpha}$ is not Martin-L\"of $P_B$-random
then $\beta$ is not Martin-L\"of $Q$-random.

Thus, let us assume that $\cond{B}{\alpha}$ is not Martin-L\"of $P_B$-random.
Then there exists a Martin-L\"of $P_B$-test $\mathcal{C}\subset B\times\N^+$ such that
\begin{equation}\label{fBaC}
  \cond{B}{\alpha}\in\osg{\mathcal{C}_n}
\end{equation}
for every $n\in\N^+$.
For each $\sigma\in B^+$, let $F(\sigma)$ be the set of all finite strings over $B\cup\{a\}$ of the form
$a^{k_1}\sigma_1a^{k_2}\sigma_2\dots \sigma_{L-1}a^{k_L}\sigma_{L}$ for some $k_1,k_2,\dots,k_L\in\N$,
where $\sigma=\sigma_1\sigma_2\dots\sigma_L$ with $\sigma_i\in B$.
Note that $F(\sigma)$ is a prefix-free subset of $(B\cup\{a\})^*$ for every $\sigma\in B^+$.
For each $\sigma\in B^+$, we see that
\begin{equation}\label{BmQosgFs=BmPBosgs_CPEuC}
\begin{split}
  \Bm{Q}{\osg{F(\sigma)}}
  &=\sum_{k_1,k_2,\dots,k_L=0}^\infty
    \Bm{Q}{\osg{a^{k_1}\sigma_1a^{k_2}\sigma_2\dots \sigma_{L-1}a^{k_L}\sigma_{L}}} \\
  &=\sum_{k_1,k_2,\dots,k_L=0}^\infty\Bm{Q}{\osg{\sigma}}Q(a)^{k_1}Q(a)^{k_2}\dots Q(a)^{k_L} \\
  &=\Bm{Q}{\osg{\sigma}}\left(\sum_{k=0}^\infty Q(a)^k\right)^L \\
  &=\Bm{Q}{\osg{\sigma}}\frac{1}{(1-Q(a))^L} \\
  &=\Bm{Q}{\osg{\sigma}}\frac{1}{P(B)^L} \\
  &=\Bm{P_B}{\osg{\sigma}},
\end{split}
\end{equation}
where we use \eqref{Q<1} and \eqref{1-Q=PB} in the forth and fifth equalities, respectively.
We then define $\mathcal{D}$ to be a subset of $\N^+\times(B\cup\{a\})^*$ such that
$\mathcal{D}_n=\bigcup_{\sigma\in \mathcal{C}_n}F(\sigma)$ for every $n\in\N^+$.
Note here that, for each $n\in\N^+$,
$\lambda\notin\mathcal{C}_n$ since $\Bm{P}{\osg{\mathcal{C}_n}}<2^{-n}<1$.
Then, since $\mathcal{C}_n$ is a prefix-free subset of $B^*$ for every $n\in\N^+$,
we see that $\mathcal{D}_n$ is a prefix-free subset of $(B\cup\{a\})^*$ for every $n\in\N^+$.
For each $n\in\N^+$, we see that
\[
  \Bm{Q}{\osg{\mathcal{D}_n}}
  \le\sum_{\sigma\in\mathcal{C}_n}\Bm{Q}{\osg{F(\sigma)}}
  =\sum_{\sigma\in\mathcal{C}_n}\Bm{P_B}{\osg{\sigma}}
  =\Bm{P_B}{\osg{\mathcal{C}_n}}<2^{-n},
\]
where the first equality follows from \eqref{BmQosgFs=BmPBosgs_CPEuC} and
the second equality follows from the prefix-freeness of $\mathcal{C}_n$.
Moreover, since $\mathcal{C}$ is r.e.,
$\mathcal{D}$ is also r.e.
Thus, $\mathcal{D}$ is a Martin-L\"of $Q$-test.

On the other hand,
since $\cond{B}{\alpha}$ is the infinite sequence over $B$ obtained from $\beta$
by eliminating all occurrences of the symbol $a$ in $\beta$,
we see that,
for every $n\in\N^+$, if $\cond{B}{\alpha}\in\osg{\mathcal{C}_n}$ then $\beta\in\osg{\mathcal{D}_n}$.
Thus, it follows from \eqref{fBaC} that $\beta\in\osg{\mathcal{D}_n}$ for every $n\in\N^+$.
Hence, $\beta$ is not Martin-L\"of $Q$-random.
This completes the proof.
\end{proof}

Let $\Omega$ be a countable alphabet, and let $P\in\PS(\Omega)$.
For any events $A,B\subset\Omega$ on the discrete probability space $P$,
we say that $A$ and $B$ are \emph{independent on} $P$ if $P(A\cap B)=P(A)P(B)$.
In the case of $P(B)>0$, it holds that $A$ and $B$ are independent on $P$ if and only if $P(A|B)=P(A)$.

Theorem~\ref{cond-ind-Bs} below gives
operational characterizations of the notion of the independence between two events
in terms of ensembles.

Let $\Omega$ be a finite alphabet.
For any $\alpha,\beta\in\Omega^\infty$,
we say that $\alpha$ and $\beta$ are \emph{equivalent}
if there exists $P\in\PS(\Omega)$ such that $\alpha$ and $\beta$ are both an ensemble for $P$.

\begin{theorem}\label{cond-ind-Bs}
Let $\Omega$ be an r.e.~infinite set, and let $P\in\PS(\Omega)$.
Let $A,B\subset\Omega$ be recursive events on the discrete probability space $P$.
Suppose that $P(B)>0$.
Then the following conditions are equivalent to one another.
\begin{enumerate}
  \item The events $A$ and $B$ are independent on $P$.
  \item For every ensemble $\alpha$ for the discrete probability space $P$,
    it holds that $\chara{A}{\alpha}$ is equivalent to $\chara{A\cap B}{\cond{B}{\alpha}}$.
  \item There exists an ensemble $\alpha$ for the discrete probability space $P$ such that
    $\chara{A}{\alpha}$ is equivalent to $\chara{A\cap B}{\cond{B}{\alpha}}$.
\end{enumerate}
\end{theorem}

\begin{proof}
Suppose that $\alpha$ is an arbitrary ensemble for the discrete probability space $P$.
Then, on the one hand,
it follows from Theorem~\ref{charaA-Bs} that $\chara{A}{\alpha}$ is Martin-L\"of $\charaps{P}{A}$-random.
On the other hand, it follows from $P(B)>0$ and Theorem~\ref{conditional_probability-Bs} that
$\cond{B}{\alpha}$ is an ensemble for
$P_B$.
Therefore,
$\chara{A\cap B}{\cond{B}{\alpha}}$ is Martin-L\"of $\charaps{P_B}{A\cap B}$-random.
This follows from Theorem~\ref{charaA-Bs} if $B$ is an infinite set and
from Theorem~17 of Tadaki~\cite{T16arXiv} otherwise.

Assume that the condition~(i) holds. Then $P_B(A\cap B)=P(A)$.
It follows that $\charaps{P_B}{A\cap B}=\charaps{P}{A}$.
Therefore, for an arbitrary ensemble $\alpha$ for the discrete probability space $P$,
we see that $\chara{A}{\alpha}$ and $\chara{A\cap B}{\cond{B}{\alpha}}$ are equivalent.
Thus, we have the implication~(i) $\Rightarrow$ (ii).

Since there exists an ensemble $\alpha$ for the discrete probability space $P$
by Theorem~\ref{Bmae-Baire}, the implication~(ii) $\Rightarrow$ (iii) is obvious.

Finally, the implication~(iii) $\Rightarrow$ (i) is shown as follows.
Assume that the condition~(iii) holds.
Then there exist an ensemble $\alpha$ for the discrete probability space $P$ and
a finite probability space $Q\in\PS(\{0,1\})$ such that
both $\chara{A}{\alpha}$ and $\chara{A\cap B}{\cond{B}{\alpha}}$ are Martin-L\"of $Q$-random.
It follows from the consideration at the beginning of this proof
that $\chara{A}{\alpha}$ is Martin-L\"of $\charaps{P}{A}$-random,
and $\chara{A\cap B}{\cond{B}{\alpha}}$ is Martin-L\"of $\charaps{P_B}{A\cap B}$-random.
Using Corollary~\ref{uniquness} we see that $\charaps{P}{A}=Q=\charaps{P_B}{A\cap B}$,
and therefore $P(A)=P_B(A\cap B)$.
Thus, the condition~(i) holds, and the proof is completed.
\end{proof}

\section{The independence of an arbitrary number of events/random variables}
\label{IANERV}

In this section we operationally characterize the notion of the \emph{independence
of an arbitrary number of events/random variables} on a discrete probability space in terms of ensembles.

First, we consider the operational characterizations of
the notion of
the independence of an arbitrary number of random variables,
in terms of ensembles.
Let $\Omega$ be an
arbitrary
countable alphabet, and
let $P$ be an arbitrary discrete probability space on $\Omega$.
A \emph{random variable} on
$\Omega$ is a function
$X\colon\Omega\to\Omega'$ where $\Omega'$ is
a countable alphabet.
Let $X_1\colon\Omega\to\Omega_1,\dots,X_n\colon\Omega\to\Omega_n$ be random variables on $\Omega$.
For any predicate
$F(v_1,\dotsc,v_n)$ with variables $v_1,\dots,v_n$,
we use $F(X_1,\dots,X_n)$ to denote the event
$$\{a\in\Omega\mid F(X_1(a),\dots,X_n(a))\}$$
on $P$.
We say that the random variables $X_1,\dots,X_n$ are
\emph{independent on} $P$
if for every $x_1\in\Omega_1,\dots,x_n\in\Omega_n$ it holds that
$$P(X_1=x_1\;\&\;\dotsc\;\&\;X_n=x_n)=P(X_1=x_n)\dotsm P(X_n=x_n).$$
We use $X_1\times\dots\times X_n$ to denote a random variable
$Y\colon\Omega\to\Omega_1\times\dots\times\Omega_n$ on $\Omega$ such that
$$Y(a)=(X_1(a),\dots,X_n(a))$$ for every $a\in\Omega$.
Note here that $\Omega_1\times\dots\times\Omega_n$ is a countable alphabet,
since $\Omega_1,\dots,\Omega_n$ are all countable alphabets.

For any random variable $X\colon\Omega\to\Omega'$ on $\Omega$,
we use $X(P)$ to denote a discrete probability space $P'\in\PS(\Omega')$ such that
$P'(x)=P(X=x)$ for every $x\in\Omega'$.

Let $\Omega_1,\dots,\Omega_n$ be
countable alphabets.
For any $P_1\in\PS(\Omega_1),\dots,P_n\in\PS(\Omega_n)$,
we use
$$P_1\times\dots\times P_n$$
to denote a discrete probability space
$Q\in\PS(\Omega_1\times\dots\times\Omega_n)$ such that
$$Q(a_1,\dots,a_n)=P_1(a_1)\dotsm P_n(a_n)$$
for every $a_1\in\Omega_1,\dotsc,a_n\in\Omega_n$.
Then the notion of the independence of random variables can be rephrased as follows.

\begin{proposition}\label{X(P)-independent}
Let $\Omega$ be
a countable alphabet,
and let $P\in\PS(\Omega)$.
Let $X_1\colon\Omega\to\Omega_1,\dots,X_n\colon\Omega\to\Omega_n$ be random variables on $\Omega$.
Then the random variables $X_1,\dots,X_n$ are independent on $P$ if and only if
$$(X_1\times\dots\times X_n)(P)=X_1(P)\times\dots\times X_n(P).$$
\end{proposition}

\begin{proof}
Let $x_1\in\Omega_1,\dots,x_n\in\Omega_n$.
On the one hand, we have
\begin{align*}
  ((X_1\times\dots\times X_n)(P))(x_1,\dots,x_n)
  &=P((X_1\times\dots\times X_n)=(x_1,\dots,x_n)) \\
  &=P(X_1=x_1\;\&\;\dotsc\;\&\;X_n=x_n).
\end{align*}
On the other hand, we have
\begin{align*}
  (X_1(P)\times\dots\times X_n(P))(x_1,\dots,x_n)
  &= (X_1(P))(x_1)\dotsm(X_n(P))(x_n) \\
  &=P(X_1=x_n)\dotsm P(X_n=x_n).
\end{align*}
Thus, the result follows
from the definition of the independence of random variables.
\end{proof}

Let $\Omega$ be
a countable alphabet,
and let $X\colon\Omega\to\Omega'$ be a random variable on $\Omega$.
For any $\alpha\in\Omega^\infty$,
we use $X(\alpha)$ to denote an infinite sequence $\beta$ over $\Omega'$ such that
$\beta(k)=X(\alpha(k))$ for every $k\in\N^+$.
We can then show the following theorem, which states that
ensembles are \emph{closed under the mapping by a random variable}.

\begin{theorem}[Closure property under the mapping by a random variable]\label{X(p)-X(P)}
Let $\Omega$ and $\Omega'$ be r.e.~infinite sets, and let $P\in\PS(\Omega)$.
Let $X\colon\Omega\to\Omega'$ be a random variable on $\Omega$.
Suppose that
$X$ is a partial recursive function.%
\footnote{%
The domain of definition of $X$ is precisely $\Omega$
and not a proper subset of $\Omega$.}
If $\alpha$ is an ensemble for $P$ then $X(\alpha)$ is an ensemble for $X(P)$.
\end{theorem}

\begin{proof}
We show the contraposition. Suppose that $X(\alpha)$ is not Martin-L\"of $X(P)$-random.
Then there exists a Martin-L\"of $X(P)$-test $\mathcal{S}\subset\N^+\times(\Omega')^*$ such that
\begin{equation}\label{XainosgmathSn_CPuMbRV}
  X(\alpha)\in\osg{\mathcal{S}_n}
\end{equation}
for every $n\in\N^+$.
For each $\sigma\in(\Omega')^*$,
let $f(\sigma)$ be the set of all
$\tau\in\Omega^*$
such that
(i) $\abs{\tau}=\abs{\sigma}$ and (ii) $X(\tau(k))=\sigma(k)$ for every $k=1,2,\dots,\abs{\sigma}$.
Then, since
$$(X(P))(x)=\sum_{a\in X^{-1}(\{x\})}P(a)$$
for every $x\in\Omega'$, we have that
\begin{equation}\label{BmXP=os=XPs=Pfs=BmPofs_CPuMbRV}
  \Bm{X(P)}{\osg{\sigma}}=(X(P))(\sigma)=P(f(\sigma))=\Bm{P}{\osg{f(\sigma)}}
\end{equation}
for each $\sigma\in(\Omega')^*$.
We then define $\mathcal{T}$ to be a subset of $\N^+\times \Omega^*$ such that
$\mathcal{T}_n=\bigcup_{\sigma\in\mathcal{S}_n} f(\sigma)$ for every $n\in\N^+$.
Since $\mathcal{S}_n$ is a prefix-free subset of $(\Omega')^*$ for every $n\in\N^+$,
we see that $\mathcal{T}_n$ is a prefix-free subset of $\Omega^*$ for every $n\in\N^+$.
For each $n\in\N^+$, we also see that
\[
  \Bm{P}{\osg{\mathcal{T}_n}}
  \le\sum_{\sigma\in\mathcal{S}_n}\Bm{P}{\osg{f(\sigma)}}
  =\sum_{\sigma\in\mathcal{S}_n}\Bm{X(P)}{\osg{\sigma}}
  =\Bm{X(P)}{\osg{\mathcal{S}_n}}<2^{-n},
\]
where the first equality follows from \eqref{BmXP=os=XPs=Pfs=BmPofs_CPuMbRV} and
the second equality follows from the prefix-freeness of $\mathcal{S}_n$.
Moreover, since
$X$ is a partial recursive function with $\Dom X=\Omega$
and $\mathcal{S}$ is r.e.,
it follows that $\mathcal{T}$ is
r.e.
Thus, $\mathcal{T}$ is a Martin-L\"of $P$-test.

On the other hand,
note that, for every $n\in\N^+$, if $X(\alpha)\in\osg{\mathcal{S}_n}$ then $\alpha\in\osg{\mathcal{T}_n}$.
Thus, it follows from \eqref{XainosgmathSn_CPuMbRV}
that $\alpha\in\osg{\mathcal{T}_n}$ for every $n\in\N^+$.
Hence, $\alpha$ is not Martin-L\"of $P$-random.
This completes the proof.
\end{proof}

We introduce the notion of the \emph{independence} of ensembles as follows.
Let $\Omega_1,\dots,\Omega_n$ be
countable
alphabets.
For any $\alpha_1\in\Omega_1^\infty,\dots,\alpha_n\in\Omega_n^\infty$,
we use
$$\alpha_1\times\dots\times\alpha_n$$
to denote an infinite sequence
$\alpha$ over $\Omega_1\times\dots\times\Omega_n$
such that $\alpha(k)=(\alpha_1(k),\dots,\alpha_n(k))$ for every $k\in\N^+$.
Thus, $\alpha_1\times\dots\times\alpha_n\in(\Omega_1\times\dots\times\Omega_n)^\infty$
for every $\alpha_1\in\Omega_1^\infty,\dots,\alpha_n\in\Omega_n^\infty$.
For any $\sigma_1\in\Omega_1^*,\dots,\sigma_n\in\Omega_n^*$
with $\abs{\sigma_1}=\dots =\abs{\sigma_n}$,
we define $\sigma_1\times\dots\times\sigma_n$ in a similar manner,
where we define $\lambda\times\lambda$ as $\lambda$, in particular.
Thus, $\sigma_1\times\dots\times\sigma_n\in(\Omega_1\times\dots\times\Omega_n)^*$
for every $\sigma_1\in\Omega_1^*,\dots,\sigma_n\in\Omega_n^*$
with $\abs{\sigma_1}=\dots =\abs{\sigma_n}$.

\begin{definition}[Independence of ensembles]\label{Independence_of_ensembles}
Let $\Omega_1,\dotsc,\Omega_n$ be r.e.~infinite sets,
and let $P_1\in\PS(\Omega_1),\dots,P_n\in\PS(\Omega_n)$.
Let $\alpha_1,\dots,\alpha_n$ be ensembles for $P_1,\dots,P_n$, respectively.
We say that $\alpha_1,\dots,\alpha_n$ are
\emph{independent}
if $\alpha_1\times\dots\times\alpha_n$ is
an ensemble for
$P_1\times\dots\times P_n$.
\qed
\end{definition}

In Definition~\ref{Independence_of_ensembles}, note that
$\Omega_1\times\dots\times\Omega_n$ is an r.e.~infinite set,
since each of $\Omega_1,\dotsc,\Omega_n$ is an r.e.~infinite set.
Thus, the notion of Martin-L\"of $P_1\times\dots\times P_n$-randomness
given
by Definition~\ref{ML_P-randomness-Bs}
can be
properly
applied to
the infinite sequence $\alpha_1\times\dots\times\alpha_n$ over $\Omega_1\times\dots\times\Omega_n$.
Note that
the notion of the independence of ensembles in our theory corresponds to
the notion of \emph{independence} of collectives in the theory of collectives by von Mises \cite{vM64}.

Theorem~\ref{independence-independence} below gives equivalent characterizations of the notion of
the independence of random variables in terms of that of ensembles.
To prove Theorem~\ref{independence-independence}, we first show the following proposition.

\begin{proposition}\label{X1xn-alpha}
Let $\Omega$ be a countable alphabet.
Let $\alpha\in\Omega^\infty$, and
let $X_1\colon\Omega\to\Omega_1,\dots,X_n\colon\Omega\to\Omega_n$ be
random variables on $\Omega$. Then
$(X_1\times\dots\times X_n)(\alpha)=X_1(\alpha)\times\dots\times X_n(\alpha)$.
\end{proposition}

\begin{proof}
For each $k\in\N^+$, we see that
\begin{align*}
((X_1\times\dots\times X_n)(\alpha))(k)
&=(X_1\times\dots\times X_n)(\alpha(k)) \\
&=(X_1(\alpha(k)),\dots,X_n(\alpha(k))) \\
&=((X_1(\alpha))(k),\dots,(X_n(\alpha))(k)) \\
&=(X_1(\alpha)\times\dots\times X_n(\alpha))(k).
\end{align*}
This completes the proof.
\end{proof}

\begin{theorem}\label{independence-independence}
Let $\Omega$ and $\Omega_1,\dots,\Omega_n$ be r.e.~infinite sets,
and let $P\in\PS(\Omega)$.
Let $X_1\colon\Omega\to\Omega_1,\dots,X_n\colon\Omega\to\Omega_n$ be
random variables on $\Omega$.
Suppose that
all of $X_1,\dots,X_n$ are partial recursive functions.%
\footnote{%
The domain of definition of each $X_i$ is precisely $\Omega$
and not a proper subset of $\Omega$.}
Then the following conditions are equivalent to one another.
\begin{enumerate}
  \item The random variables $X_1,\dots,X_n$ are independent on $P$.
  \item For every ensemble $\alpha$ for
    $P$,
    the ensembles $X_1(\alpha),\dots,X_n(\alpha)$ are independent.
  \item There exists an ensemble $\alpha$ for $P$ such that
    the ensembles $X_1(\alpha),\dots,X_n(\alpha)$ are independent.
\end{enumerate}
\end{theorem}

\begin{proof}
Assume that the condition~(i) holds.
Let $\alpha$ be an arbitrary ensemble for the discrete probability space $P$.
First,
for each $i=1,\dots,n$,
since
$X_i$ is a partial recursive function,
it follows from Theorem~\ref{X(p)-X(P)} that $X_i(\alpha)$ is Martin-L\"of $X_i(P)$-random.
On the other hand,
since $X_i$ is a partial recursive function for every $i=1,\dots,n$,
we see that
$X_1\times\dots\times X_n$
is a partial recursive function with $\Dom f=\Omega$ and
$f(\Dom f)\subset\Omega_1\times\dots\times\Omega_n$.
Note here that $\Omega_1\times\dots\times\Omega_n$ is an r.e.~infinite set.
Thus,
it follows from Theorem~\ref{X(p)-X(P)} that
$(X_1\times\dots\times X_n)(\alpha)$ is Martin-L\"of $(X_1\times\dots\times X_n)(P)$-random.
Therefore, by Proposition~\ref{X1xn-alpha} and Proposition~\ref{X(P)-independent}, we see that
$X_1(\alpha)\times\dots\times X_n(\alpha)$ is Martin-L\"of $X_1(P)\times\dots\times X_n(P)$-random.
Thus, the ensembles $X_1(\alpha),\dots,X_n(\alpha)$ are independent.
Hence, we have the implication~(i) $\Rightarrow$ (ii).

Since there exists an ensemble $\alpha$ for the discrete probability space $P$
by Theorem~\ref{Bmae-Baire}, the implication~(ii) $\Rightarrow$ (iii) is obvious.

Finally, the implication~(iii) $\Rightarrow$ (i) is shown as follows.
Assume that the condition~(iii) holds.
Then there exists an ensemble $\alpha$ for $P$ such that
$X_1(\alpha)\times\dots\times X_n(\alpha)$ is Martin-L\"of $X_1(P)\times\dots\times X_n(P)$-random.
It follows from Proposition~\ref{X1xn-alpha} that
$(X_1\times\dots\times X_n)(\alpha)$ is Martin-L\"of $X_1(P)\times\dots\times X_n(P)$-random.
On the other hand, 
since
$X_i$
is a partial recursive function
for every $i=1,\dots,n$,
it follows from Theorem~\ref{X(p)-X(P)} that
$(X_1\times\dots\times X_n)(\alpha)$ is Martin-L\"of $(X_1\times\dots\times X_n)(P)$-random.
Thus, using Corollary~\ref{uniquness}, we have
$X_1(P)\times\dots\times X_n(P)=(X_1\times\dots\times X_n)(P)$.
Therefore, it follows from Proposition~\ref{X(P)-independent} that
the random variables $X_1,\dots,X_n$ are independent on $P$.
This completes the proof.
\end{proof}

Next, we consider the operational characterizations of the notion of
the independence of an arbitrary number of events,
in terms of ensembles.

Let $\Omega$ be an arbitrary countable alphabet, and
let $P$ be an arbitrary discrete probability space on $\Omega$.
Let $A_1,\dots,A_n$ be
arbitrary
events on the discrete probability space
$P$.
We say that the events $A_1,\dots,A_n$ are \emph{independent on} $P$ if
for every $i_1,\dots,i_k$ with $1\le i_1<\dots <i_k\le n$ it holds that
$$P(A_{i_1}\cap\dots\cap A_{i_k})=P(A_{i_1})\dotsm P(A_{i_k}).$$
For any $A\subset\Omega$, we use $\chi_A$ to denote a function $f\colon\Omega\to\{0,1\}$ such that
$f(a):=1$ if $a\in A$ and $f(a):=0$ otherwise.
Note that
$\chara{A}{\alpha}=\chi_A(\alpha)$ for every $A\subset\Omega$ and $\alpha\in\Omega^\infty$.
It is then easy to show the following proposition.

\begin{proposition}\label{independence-random-variables-events}
Let $\Omega$ be a countable alphabet, and let $P\in\PS(\Omega)$.
Let $A_1,\dots,A_n\subset\Omega$.
Then the events $A_1,\dots,A_n$ are independent on $P$ if and only if
the random variables $\chi_{A_1},\dots,\chi_{A_n}$ are independent on $P$.
\qed
\end{proposition}

Using Proposition~\ref{independence-random-variables-events},
Theorem~\ref{independence-independence} results in
Theorem~\ref{independence-independence-events} below,
which gives equivalent characterizations of the notion of
the independence of an arbitrary number of events in terms of that of ensembles.

\begin{theorem}\label{independence-independence-events}
Let $\Omega$ be an r.e.~infinite set, and let $P\in\PS(\Omega)$.
Let $A_1,\dots,A_n$ be recursive events on the discrete probability space
$P$.
Then the following conditions are equivalent to one another.
\begin{enumerate}
  \item The events $A_1,\dots,A_n$ are independent on $P$.
  \item For every ensemble $\alpha$ for $P$,
    the ensembles $\chara{A_1}{\alpha},\dots,\chara{A_n}{\alpha}$ are independent.
  \item There exists an ensemble $\alpha$ for $P$ such that
    the ensembles $\chara{A_1}{\alpha},\dots,\chara{A_n}{\alpha}$ are independent.\qed
\end{enumerate}
\end{theorem}

\section{Further equivalence of the notions of independence on computable discrete probability spaces}
\label{FENICFPS}

In the preceding section we saw that
the independence of an arbitrary number of events/random variables
and that of ensembles are equivalent to each other on an arbitrary discrete probability space.
In this section we show that
these independence notions are further equivalent to
the notion of the independence in the sense of van Lambalgen's Theorem \cite{vL87}
in the case where the underlying discrete probability space is \emph{computable}.
Thus, \emph{the three independence notions are equivalent to one another} in this case.
To show
the equivalence,
we generalize van Lambalgen's Theorem \cite{vL87} over our framework first.

\subsection{A generalization of van Lambalgen's Theorem}
\label{van Lambalgen}

To study a generalization of van Lambalgen's Theorem,
first we generalize the notion of Martin-L\"of $P$-randomness over \emph{relativized computation}
and introduce the notion of \emph{Martin-L\"of $P$-randomness relative to an oracle}.

The \emph{relativized computation} is a generalization of normal computation.
For each $k=1,\dots,\ell$, let $\beta_k$ be an arbitrary infinite sequence over an r.e.~infinite set.
In the relativized computation,
a (deterministic) Turing machine is allowed to refer to $\beta_1,\dots,\beta_\ell$ as an \emph{oracle} during the computation.
Namely,
in the relativized computation,
a Turing machine can query $(k,n)\in\{1,\dots,\ell\}\times\N^+$ at any time and then
obtains the response $\beta_k(n)$ during the computation.
Such a Turing machine is called an \emph{oracle Turing machine}.
The relativized computation is more powerful than normal computation, in general.

Let $\Omega$ be an r.e.~infinite set, and let $P\in\PS(\Omega)$.
We define
the notion of
a
\emph{Martin-L\"of $P$-test relative to $\beta_1,\dots,\beta_\ell$}
as a Martin-L\"{o}f $P$-test
where the Turing machine computing the Martin-L\"{o}f $P$-test is an oracle Turing machine
which can refer to
any elements of each of
the sequences $\beta_1,\dots,\beta_\ell$ during the computation.
Based on this notion,
we define the notion of \emph{Martin-L\"of $P$-randomness relative to $\beta_1,\dots,\beta_\ell$}
in
the same manner as
(ii) and (iii)
of Definition~\ref{ML_P-randomness-Bs}.
Formally, the notion of \emph{Martin-L\"of $P$-randomness relative to infinite sequences}
is defined as follows.

\begin{definition}[Martin-L\"of $P$-randomness relative to infinite sequences]\label{ML_P-randomness_rs}
Let $\Omega$ be an r.e.~infinite set, and let $P\in\PS(\Omega)$.
For each $k=1,\dots,\ell$, let $\beta_k$ be an infinite sequence over an r.e.~infinite set.
A subset $\mathcal{C}$ of $\N^+\times \Omega^*$ is called
a \emph{Martin-L\"{o}f $P$-test relative to $\beta_1,\dots,\beta_\ell$} if the following holds.
\begin{enumerate}
  \item
    There exists an oracle Turing machine $\mathcal{M}$ such that
    \[
       \mathcal{C}
       =\{x\in\N^+\times \Omega^*\mid
       \text{$\mathcal{M}$ accepts $x$ relative to $\beta_1,\dots,\beta_\ell$}\};
    \]
  \item For every $n\in\N^+$ it holds that
    $\mathcal{C}_n$ is a prefix-free subset of $\Omega^*$ and
    $\Bm{P}{\osg{\mathcal{C}_n}}< 2^{-n}$
    where
    $\mathcal{C}_n:=
    \left\{\,
      \sigma\bigm|(n,\sigma)\in\mathcal{C}
    \,\right\}$.
\end{enumerate}

For any $\alpha\in \Omega^\infty$,
we say that $\alpha$ is \emph{Martin-L\"{o}f $P$-random relative to $\beta_1,\dots,\beta_\ell$}
if for every Martin-L\"{o}f $P$-test $\mathcal{C}$ relative to $\beta_1,\dots,\beta_\ell$
there exists $n\in\N^+$ such that $\alpha\notin\osg{\mathcal{C}_n}$.
\qed
\end{definition}

Just like in the definition of a Martin-L\"{o}f $P$-test given in Definition~\ref{ML_P-randomness-Bs},
we require in Definition~\ref{ML_P-randomness_rs} that
the set $\mathcal{C}_n$ is prefix-free in the definition of
a Martin-L\"{o}f $P$-test $\mathcal{C}$ relative to $\beta_1,\dots,\beta_\ell$.
However, as in the case of a Martin-L\"{o}f $P$-test,
we can eliminate this
requirement
while keeping the notion of Martin-L\"of $P$-randomness relative to $\beta_1,\dots,\beta_\ell$
the same.
Namely, we can show the following theorem, corresponding to Theorem~\ref{eliminate-prefix-freeness}.

\begin{theorem}\label{eliminate-prefix-freeness-relative-to-infinite-sequences}
Let $\Omega$ be an r.e.~infinite set, and let $P\in\PS(\Omega)$.
For each $k=1,\dots,\ell$, let $\beta_k$ be an infinite sequence over an r.e.~infinite set.
Suppose that a subset $\mathcal{C}$ of $\N^+\times \Omega^*$ satisfies the following two conditions:
\begin{enumerate}
  \item
    There exists an oracle Turing machine $\mathcal{M}$ such that
    \[
       \mathcal{C}
       =\{x\in\N^+\times \Omega^*\mid
       \text{$\mathcal{M}$ accepts $x$ relative to $\beta_1,\dots,\beta_\ell$}\};
    \]
  \item For every $n\in\N^+$ it holds that
    $\Bm{P}{\osg{\mathcal{C}_n}}< 2^{-n}$
    where
    $\mathcal{C}_n:=
    \left\{\,
      \sigma\bigm|(n,\sigma)\in\mathcal{C}
    \,\right\}$.
\end{enumerate}
Then
there exists a Martin-L\"{o}f $P$-test
$\mathcal{D}$ relative to $\beta_1,\dots,\beta_\ell$
such that
$\osg{\mathcal{C}_n}=\osg{\mathcal{D}_n}$ for every $n\in\N^+$.
\qed
\end{theorem}

From Theorem~\ref{eliminate-prefix-freeness-relative-to-infinite-sequences}
we have the following theorem,
corresponding to Theorem~\ref{ML_P-randomness_eliminated-prefix-freeness}.

\begin{theorem}\label{ML_P-randomness_eliminated-prefix-freeness-relative-to-infinite-sequences}
Let $\Omega$ be an r.e.~infinite set, and let $P\in\PS(\Omega)$.
For each $k=1,\dots,\ell$, let $\beta_k$ be an infinite sequence over an r.e.~infinite set.
Let $\alpha\in\Omega^\infty$.
Then the following conditions are equivalent to each other.
\begin{enumerate} 
  \item The infinite sequence $\alpha$ is Martin-L\"{o}f $P$-random relative to $\beta_1,\dots,\beta_\ell$.
  \item For every subset $\mathcal{C}$ of $\N^+\times \Omega^*$, if
    \begin{enumerate}
      \item
        there exists an oracle Turing machine $\mathcal{M}$ such that
        \[
           \mathcal{C}
           =\{x\in\N^+\times \Omega^*\mid
           \text{$\mathcal{M}$ accepts $x$ relative to $\beta_1,\dots,\beta_\ell$}\}, and
        \]
      \item for every $n\in\N^+$ it holds that
        $\Bm{P}{\osg{\mathcal{C}_n}}< 2^{-n}$,
    \end{enumerate}
    then there exists $n\in\N^+$ such that $\alpha\notin\osg{\mathcal{C}_n}$.\qed
\end{enumerate}
\end{theorem}

The following holds, obviously.

\begin{proposition}
Let $\Omega$ be an r.e.~infinite set, and let $P\in\PS(\Omega)$.
For each $k=1,\dots,\ell$, let $\beta_k$ be an infinite sequence over an r.e.~infinite set.
For every $\alpha\in\Omega^\infty$,
if $\alpha$ is Martin-L\"of $P$-random relative to $\beta_1,\dots,\beta_\ell$
then $\alpha$ is Martin-L\"of $P$-random.
\qed
\end{proposition}

The converse does not necessarily hold.
In the case where $\alpha$ is Martin-L\"of $P$-random, the converse means that
the Martin-L\"of $P$-randomness of $\alpha$ is \emph{independent} of $\beta_1,\dots,\beta_\ell$
in a certain sense.

We here recall van Lambalgen's Theorem.
Let $\beta$ be an infinite binary sequence.
For any $\alpha\in\XI$,
we say that $\alpha$ is \emph{Martin-L\"{o}f random relative to $\beta$} if
$\alpha$ is Martin-L\"{o}f $U$-random relative to $\beta$ where
$U$ is a discrete probability space on $\N$
such that (i) $U(0)=U(1)=1/2$ and (ii) $U(n)=0$ for every $n\ge 2$.
Based on this notion of \emph{Martin-L\"of randomness relative to an
infinite sequence},
van Lambalgen's Theorem is stated as follows.

\begin{theorem}[van Lambalgen's Theorem, van Lambalgen \cite{vL87}]
Let $\alpha,\beta\in\{0,1\}^\infty$,
and let $\alpha\oplus\beta$ denote the infinite binary sequence
\[
  \alpha(1)\beta(1)\alpha(2)\beta(2)\alpha(3)\beta(3)\dotsc\dotsc.
\]
Then the following conditions are equivalent.
\begin{enumerate}
  \item $\alpha\oplus\beta$ is Martin-L\"of random.
  \item $\alpha$ is Martin-L\"of random relative to $\beta$ and $\beta$ is Martin-L\"of random.\qed
\end{enumerate}
\end{theorem}

We generalize van Lambalgen's Theorem as follows.

\begin{theorem}[Generalization of van Lambalgen's Theorem I]\label{gvL}
Let $\Omega_1$ and $\Omega_2$ be r.e.~infinite sets,
and let $P_1\in\PS(\Omega_1)$ and $P_2\in\PS(\Omega_2)$.
Let $\alpha_1\in\Omega_1^\infty$ and $\alpha_2\in\Omega_2^\infty$.
For each $k=1,\dots,\ell$,
let $\beta_k$ be an infinite sequence over an r.e.~infinite set.
Suppose that $P_1$ is computable.
Then $\alpha_1\times\alpha_2$ is Martin-L\"{o}f $P_1\times P_2$-random relative to
$\beta_1,\dots,\beta_\ell$
if and only if
$\alpha_1$ is Martin-L\"{o}f $P_1$-random relative to
$\alpha_2,\beta_1,\dots,\beta_\ell$
and $\alpha_2$ is Martin-L\"{o}f $P_2$-random relative to $\beta_1,\dots,\beta_\ell$.
\qed
\end{theorem}

The proof of Theorem~\ref{gvL} is obtained by
generalizing and elaborating
the proof of van Lambalgen's Theorem given in Nies~\cite[Section 3.4]{N09}.
The detail
of the proof of Theorem~\ref{gvL}
is given in the subsequent two subsections.
Note that in Theorem~\ref{gvL},
the computability of $P_1$ is
assumed
while that of $P_2$ is not required.

We have Theorem~\ref{cor_gvL} below based on Theorem~\ref{gvL}.
Note that the computability of $P_n$ is not required in Theorem~\ref{cor_gvL}.

\begin{theorem}[Generalization of van Lambalgen's Theorem II]\label{cor_gvL}
Let $n\ge 2$.
Let $\Omega_1,\dots,\Omega_n$ be r.e.~infinite sets,
and let $P_1\in\PS(\Omega_1),\dots,P_n\in\PS(\Omega_n)$.
Let $\alpha_1\in\Omega_1^\infty,\dots,\alpha_n\in\Omega_n^\infty$.
For each $k=1,\dots,\ell$, let $\beta_k$ be an infinite sequence over an r.e.~infinite set.
Suppose that $P_1,\dots,P_{n-1}$ are computable.
Then $\alpha_1\times\dots\times\alpha_n$ is
Martin-L\"of $P_1\times\dots\times P_n$-random relative to $\beta_1,\dots,\beta_\ell$ if and only if
for every $k=1,\dots,n$ it holds that
$\alpha_k$ is Martin-L\"of $P_k$-random relative to
$\alpha_{k+1},\dots, \alpha_n,\beta_1,\dots,\beta_\ell$.
\end{theorem}

\begin{proof}
We show the result by induction on $n\ge 2$.
In the case of $n=2$, the result holds since it is precisely Theorem~\ref{gvL}.

For an arbitrary $m\ge 2$, assume that the result holds for $n=m$.
Let $\Omega_1,\dots,\Omega_{m+1}$ be r.e.~infinite sets,
and let $P_1\in\PS(\Omega_1),\dots,P_{m+1}\in\PS(\Omega_{m+1})$.
Let $\alpha_1\in\Omega_1^\infty,\dots,\alpha_{m+1}\in\Omega_{m+1}^\infty$.
For each $k=1,\dots,\ell$, let $\beta_k$ be an infinite sequence over an r.e.~infinite set.
Suppose that $P_1,\dots,P_m$ are computable.
Then, by applying Theorem~\ref{gvL} with $P_1\times\dots\times P_m$ as $P_1$, $P_{m+1}$ as $P_2$,
$\alpha_1\times\dots\times \alpha_m$ as $\alpha_1$, and $\alpha_{m+1}$ as $\alpha_2$
in Theorem~\ref{gvL},
we have that $(\alpha_1\times\dots\times\alpha_m)\times\alpha_{m+1}$ is
Martin-L\"{o}f $(P_1\times\dots\times P_m)\times P_{m+1}$-random relative to $\beta_1,\dots,\beta_\ell$
if and only if
$\alpha_1\times\dots\times\alpha_m$ is Martin-L\"{o}f $P_1\times\dots\times P_m$-random
relative to $\alpha_{m+1},\beta_1,\dots,\beta_\ell$ and
$\alpha_{m+1}$ is Martin-L\"{o}f $P_{m+1}$-random relative to $\beta_1,\dots,\beta_\ell$.
Thus, by applying the result for $n=m$ we have the result for $n=m+1$.
This completes the proof.
\end{proof}

\subsection{The proof of the ``only if'' part of Theorem~\ref{gvL}}
\label{section-only-if-part}

We prove the following theorem, from which the ``only if'' part of Theorem~\ref{gvL} follows.

\begin{theorem}\label{only-if-part}
Let $\Omega_1$ and $\Omega_2$ be r.e.~infinite sets,
and let $P_1\in\PS(\Omega_1)$ and $P_2\in\PS(\Omega_2)$.
Let $\alpha_1\in\Omega_1^\infty$ and $\alpha_2\in\Omega_2^\infty$.
For each $k=1,\dots,\ell$, let $\beta_k$ be an infinite sequence over an r.e.~infinite set.
Suppose that $P_1$ is right-computable.
If $\alpha_1\times\alpha_2$ is Martin-L\"{o}f $P_1\times P_2$-random relative to
$\beta_1,\dots,\beta_\ell$ then $\alpha_1$ is Martin-L\"of $P_1$-random relative to
$\alpha_2,\beta_1,\dots,\beta_\ell$
and $\alpha_2$ is Martin-L\"{o}f $P_2$-random relative to $\beta_1,\dots,\beta_\ell$.
\qed
\end{theorem}

In order to prove Theorem~\ref{only-if-part},
we use the notion of \emph{universal Martin-L\"of $P$-test relative to infinite sequences}.

\begin{definition}[Universal Martin-L\"of $P$-test relative to infinite sequences]\label{UML_P-test_rs}
Let $\Omega$ be an r.e.~infinite set, and let $P\in\PS(\Omega)$.
Let $\ell\in\N^+$, and let $\Theta_1,\dots,\Theta_{\ell}$ be r.e.~infinite sets.
An oracle Turing machine $\mathcal{M}$ is called a
\emph{universal Martin-L\"of $P$-test relative to
$\ell$ infinite sequences over $\Theta_1,\dots,\Theta_{\ell}$}
if for every
$\beta_1\in\Theta_1^\infty,\dots,\beta_\ell\in\Theta_{\ell}^\infty$
there exists
$\mathcal{C}$
such that
\begin{enumerate}
  \item 
    $\mathcal{C}=\{x\in\N^+\times \Omega^*\mid
    \text{$\mathcal{M}$ accepts $x$ relative to $\beta_1,\dots,\beta_\ell$}\}$,
  \item for every $n\in\N^+$ it holds that $\mathcal{C}_n$ is a prefix-free subset of $\Omega^*$ and
    $\Bm{P}{\osg{\mathcal{C}_n}}<2^{-n}$
    where
    $\mathcal{C}_n:=
    \left\{\,
      \sigma\bigm|(n,\sigma)\in\mathcal{C}
    \,\right\}$, and
  \item for every Martin-L\"{o}f $P$-test $\mathcal{D}$ relative to $\beta_1,\dots,\beta_\ell$,
    $$\bigcap_{n=1}^\infty \osg{\mathcal{D}_n}\subset \bigcap_{n=1}^\infty \osg{\mathcal{C}_n}.$$\qed
\end{enumerate}
\end{definition}

It is then easy to show the following theorem.

\begin{theorem}\label{ExistsUML_P-test_rs}
Let $\Omega$ be an r.e.~infinite set, and let $P\in\PS(\Omega)$.
Let $\ell\in\N^+$, and let $\Theta_1,\dots,\Theta_{\ell}$ be r.e.~infinite sets.
Suppose that $P$ is right-computable.
Then there exists a universal Martin-L\"of $P$-test relative to
$\ell$ infinite sequences over $\Theta_1,\dots,\Theta_{\ell}$.
\qed
\end{theorem}

Then, using Theorems~\ref{ExistsUML_P-test_rs}, we can prove Theorem~\ref{only-if-part} as follows.

\begin{proof}[Proof of Theorem~\ref{only-if-part}]
Let $\Omega_1$ and $\Omega_2$ be r.e.~infinite sets,
and let $P_1\in\PS(\Omega_1)$ and $P_2\in\PS(\Omega_2)$.
Let $\alpha_1\in\Omega_1^\infty$ and $\alpha_2\in\Omega_2^\infty$.
Let $\ell\in\N^+$, and let $\Theta_1,\dots,\Theta_{\ell}$ be r.e.~infinite sets.
Let $\beta_1\in\Theta_1^\infty,\dots,\beta_\ell\in\Theta_{\ell}^\infty$.

First, we show that
if $\alpha_1\times\alpha_2$ is Martin-L\"{o}f $P_1\times P_2$-random relative to
$\beta_1\dots,\beta_\ell$
then $\alpha_2$ is Martin-L\"of $P_2$-random relative to $\beta_1,\dots,\beta_\ell$.
Actually, we prove the contraposition.
Thus, let us assume that
$\alpha_2$ is not Martin-L\"of $P_2$-random relative to $\beta_1,\dots,\beta_\ell$.
Then there exists a Martin-L\"of $P_2$-test $\mathcal{S}$
relative to $\beta_1,\dots,\beta_\ell$
such that
\begin{equation}\label{a2inosgmathCn_MarginalProb}
  \alpha_2\in\osg{\mathcal{S}_n}
\end{equation}
for every $n\in\N^+$.
For each $\sigma_2\in\Omega_2^*$,
we use $F(\sigma_2)$ to denote the set
$$\{\sigma_1\times\sigma_2\mid \sigma_1\in\Omega_1^*\;\&\;\abs{\sigma_1}=\abs{\sigma_2}\}.$$
Then, since $P_2(a_2)=\sum_{a_1\in\Omega_1}(P_1\times P_2)(a_1,a_2)$ for every $a_2\in\Omega_2$,
we have that
\begin{equation}\label{BP2os2=P2s2=P1tP2Fs2=BP1tP2oFs2_MarginalProb}
  \Bm{P_2}{\osg{\sigma_2}}=P_2(\sigma_2)
  =(P_1\times P_2)(F(\sigma_2))=\Bm{P_1\times P_2}{\osg{F(\sigma_2)}}
\end{equation}
for each $\sigma_2\in\Omega_2^*$.
We then define $\mathcal{T}$ to be a subset of $\N^+\times (\Omega_1\times\Omega_2)^*$ such that
$\mathcal{T}_n=\bigcup_{\sigma_2\in\mathcal{S}_n} F(\sigma_2)$ for every $n\in\N^+$.
Since $\mathcal{S}_n$ is a prefix-free subset of $\Omega_2^*$ for every $n\in\N^+$,
we see that $\mathcal{T}_n$ is a prefix-free subset of $(\Omega_1\times\Omega_2)^*$
for every $n\in\N^+$.
For each $n\in\N^+$, we also see that
\[
  \Bm{P_1\times P_2}{\osg{\mathcal{T}_n}}
  \le\sum_{\sigma_2\in\mathcal{S}_n}\Bm{P_1\times P_2}{\osg{F(\sigma_2)}}
  =\sum_{\sigma_2\in\mathcal{S}_n}\Bm{P_2}{\osg{\sigma_2}}
  =\Bm{P_2}{\osg{\mathcal{S}_n}}<2^{-n},
\]
where the first equality follows from \eqref{BP2os2=P2s2=P1tP2Fs2=BP1tP2oFs2_MarginalProb} and
the second equality follows from the prefix-freeness of $\mathcal{S}_n$.
Moreover, since
$\mathcal{S}$ is r.e.~relative to $\beta_1,\dots,\beta_\ell$
and $\Omega_1$ is r.e.,
we see that $\mathcal{T}$ is
r.e.~relative to $\beta_1,\dots,\beta_\ell$.
Thus, $\mathcal{T}$ is a Martin-L\"of $P_1\times P_2$-test relative to $\beta_1,\dots,\beta_\ell$.
On the other hand, note that, for every $n\in\N^+$,
if $\alpha_2\in\osg{\mathcal{S}_n}$ then $\alpha_1\times\alpha_2\in\osg{\mathcal{T}_n}$.
Thus, it follows from \eqref{a2inosgmathCn_MarginalProb}
that $\alpha_1\times\alpha_2\in\osg{\mathcal{T}_n}$ for every $n\in\N^+$.
Hence, $\alpha_1\times\alpha_2$ is not Martin-L\"of $P_1\times P_2$-random
relative to $\beta_1,\dots,\beta_\ell$.

Next, we show that if $\alpha_1\times\alpha_2$ is Martin-L\"{o}f $P_1\times P_2$-random relative to
$\beta_1\dots,\beta_\ell$
then $\alpha_1$ is Martin-L\"of $P_1$-random relative to $\alpha_2,\beta_1,\dots,\beta_\ell$.
Since $P_1$ is right-computable, it follows from Theorem~\ref{ExistsUML_P-test_rs} that
there exists a universal Martin-L\"of $P_1$-test
relative to $\ell+1$ infinite sequences over $\Omega_2,\Theta_1,\dots,\Theta_{\ell}$.
Thus, there exists an oracle Turing machine $\mathcal{M}$ such that
for every
$\gamma\in\Omega_2^\infty$
there exists $\mathcal{C}$
such that
\begin{enumerate}
  \item
    $\mathcal{C}
    =\{x\in\N^+\times \Omega_1^*\mid
    \text{$\mathcal{M}$ accepts $x$ relative to $\gamma,\beta_1,\dots,\beta_\ell$}\}$,
  \item for every $n\in\N^+$ it holds that $\mathcal{C}_n$ is a prefix-free subset of $\Omega_1^*$ and
    $\Bm{P_1}{\osg{\mathcal{C}_n}}<2^{-n}$, and
  \item for every Martin-L\"{o}f $P_1$-test $\mathcal{D}$ relative to $\gamma,\beta_1,\dots,\beta_\ell$,
    \begin{equation*}%
      \bigcap_{n=1}^\infty \osg{\mathcal{D}_n}\subset \bigcap_{n=1}^\infty \osg{\mathcal{C}_n}.
    \end{equation*}
\end{enumerate}
We choose any particular $a\in\Omega_2^*$.
Then,
for each $\sigma\in\Omega_2^*$,
let $\mathcal{U}^\sigma$ be the set of all $x\in\N^+\times \Omega_1^*$ such that
$\mathcal{M}$ accepts $x$ relative to $\sigma a^\infty,\beta_1,\dots,\beta_\ell$
with oracle access only to the prefix of $\sigma a^\infty$ of length $\abs{\sigma}$
in the first infinite sequence.
Here, $\sigma a^\infty$ denotes the infinite sequence over $\Omega_2$
which is the concatenation of the finite string $\sigma$ and the infinite sequence consisting only of $a$.
It follows that
\begin{equation}\label{BmP1osgmUsn<2-n_RtIS}
  \Bm{P_1}{\osg{\mathcal{U}^\sigma_n}}<2^{-n}
\end{equation}
for every $\sigma\in\Omega_2^*$ and every $n\in\N^+$,
where
$\mathcal{U}^\sigma_n:=
    \left\{\,
      \tau\bigm|(n,\tau)\in\mathcal{U}^\sigma
    \,\right\}$.
For each $k,n\in\N^+$, let
$$G_n(k)=\{u\times\sigma\mid u\in\Omega_1^k\text{ \& }\sigma\in\Omega_2^k\text{ \& Some prefix of $u$ is in $\mathcal{U}^\sigma_n$}\}.$$
Then,
it is easy to see that
$G_n(k)$ is r.e.~relative to $\beta_1,\dots,\beta_\ell$ uniformly in $n$ and $k$.
Note that
\[
  \osg{G_n(k)}=\bigcup_{\sigma\in\Omega_2^k}\bigcup_{u\in S_n(k,\sigma)}\osg{u\times\sigma},
\]
for every $n,k\in\N^+$,
where $S_n(k,\sigma):=\{u\in\Omega_1^k\mid\text{Some prefix of $u$ is in $\mathcal{U}^\sigma_n$}\}$.
Therefore, for each $n,k\in\N^+$,
we see that
\begin{align*}
  \Bm{P_1\times P_2}{\osg{G_n(k)}}
  &=\sum_{\sigma\in\Omega_2^k}\sum_{u\in S_n(k,\sigma)}\Bm{P_1\times P_2}{\osg{u\times\sigma}}
  =\sum_{\sigma\in\Omega_2^k}\Bm{P_2}{\osg{\sigma}}\sum_{u\in S_n(k,\sigma)}\Bm{P_1}{\osg{u}} \\
  &=\sum_{\sigma\in\Omega_2^k}\Bm{P_2}{\osg{\sigma}}
      \Bm{P_1}{\osg{\mathcal{U}^\sigma_n\cap\Omega_1^{\le k}}}
  \le \sum_{\sigma\in\Omega_2^k}\Bm{P_1}{\osg{\mathcal{U}^\sigma_n}}\Bm{P_2}{\osg{\sigma}} \\
  &<\sum_{\sigma\in\Omega_2^k}2^{-n}\Bm{P_2}{\osg{\sigma}}
  = 2^{-n},
\end{align*}
where the last inequality follows from \eqref{BmP1osgmUsn<2-n_RtIS} (and
the fact that $\Bm{P_2}{\osg{\sigma}}>0$ for some $\sigma\in\Omega_2^k$). 
On the other hand, it follows that $\osg{G_n(k)}\subset\osg{G_n(k+1)}$
for every $n,k\in\N^+$.
For each $n\in\N^+$, let $G_n=\bigcup_{k=1}^\infty G_n(k)$.
Then $G_n$ is r.e.~relative to $\beta_1,\dots,\beta_\ell$ uniformly in $n$, and
\[
  \Bm{P_1\times P_2}{\osg{G_n}}\le 2^{-n}
\]
for every $n\in\N^+$.
We define $\mathcal{A}$ to be a subset of $\N^+\times (\Omega_1\times\Omega_2)^*$ such that
$\mathcal{A}_n=G_{n+1}$ for every $n\in\N^+$.
Then  $\mathcal{A}$ is r.e.~relative to $\beta_1,\dots,\beta_\ell$ and
\[
  \Bm{P_1\times P_2}{\osg{\mathcal{A}_n}}< 2^{-n}
\]
for every $n\in\N^+$.

Now, assume that $\alpha_1$ is not Martin-L\"of $P_1$-random relative to $\alpha_2,\beta_1,\dots,\beta_\ell$.
Then there exists
$\mathcal{C}$
such that
\begin{enumerate}
  \item
    $\mathcal{C}
    =\{x\in\N^+\times \Omega_1^*\mid
    \text{$\mathcal{M}$ accepts $x$ relative to $\alpha_2,\beta_1,\dots,\beta_\ell$}\}$, and
  \item $$\alpha_1\in\bigcap_{n=1}^\infty \osg{\mathcal{C}_n}.$$
\end{enumerate}
Let $n\in\N^+$.
Then there exists $m\in\N^+$ such that $\rest{\alpha_1}{m}\in\mathcal{C}_{n+1}$.
Then, there exists $k\ge m$ such that
$\mathcal{M}$ accepts $(n+1,\rest{\alpha_1}{m})$ relative to $\alpha_2,\beta_1,\dots,\beta_\ell$
with oracle access only to the prefix of $\alpha_2$ of length $k$ in the first infinite sequence $\alpha_2$.
It follows that $\rest{\alpha_1}{m}\in\mathcal{U}^{\reste{\alpha_2}{k}}_{n+1}$.
Thus, $\rest{\alpha_1}{k}\times\rest{\alpha_2}{k}\in G_{n+1}(k)$,
and therefore $\alpha_1\times\alpha_2\in\osg{G_{n+1}(k)}\subset\osg{G_{n+1}}=\osg{\mathcal{A}_n}$.
Hence,
it follows from
Theorem~\ref{ML_P-randomness_eliminated-prefix-freeness-relative-to-infinite-sequences} that
$\alpha_1\times\alpha_2$ is not Martin-L\"{o}f $P_1\times P_2$-random relative to
$\beta_1,\dots,\beta_\ell$.
This completes the proof.
\end{proof}

\subsection{The proof of the ``if'' part of Theorem~\ref{gvL}}
\label{section-if-part}

Next, we prove the following theorem, from which the ``if'' part of Theorem~\ref{gvL} follows.

\begin{theorem}\label{if-part}
Let $\Omega_1$ and $\Omega_2$ be r.e.~infinite sets,
and let $P_1\in\PS(\Omega_1)$ and $P_2\in\PS(\Omega_2)$.
Let $\alpha_1\in\Omega_1^\infty$ and $\alpha_2\in\Omega_2^\infty$.
For each $k=1,\dots,\ell$, let $\beta_k$ be an infinite sequence over an r.e.~infinite set.
Suppose that $P_1$ is left-computable.
If $\alpha_1$ is Martin-L\"of $P_1$-random relative to $\alpha_2,\beta_1,\dots,\beta_\ell$
and $\alpha_2$ is Martin-L\"{o}f $P_2$-random relative to $\beta_1,\dots,\beta_\ell$,
then
$\alpha_1\times\alpha_2$ is Martin-L\"{o}f $P_1\times P_2$-random relative to
$\beta_1,\dots,\beta_\ell$. 
\end{theorem}

\begin{proof}
Suppose that $\alpha_1\times\alpha_2$ is not Martin-L\"{o}f $P_1\times P_2$-random relative to
$\beta_1,\dots,\beta_\ell$. 
Then there exists a Martin-L\"of $P$-test $\mathcal{V}$ relative to $\beta_1,\dots,\beta_\ell$
such that
\begin{enumerate}
  \item $\mathcal{V}_d$ is prefix-free for every $d\in\N^+$,
  \item $\Bm{P_1\times P_2}{\osg{\mathcal{V}_d}}<2^{-2d}$ for every $d\in\N^+$, and
  \item $\alpha_1\times\alpha_2\in\osg{\mathcal{V}_d}$ for every $d\in\N^+$.
\end{enumerate}

On the one hand, for each $x\in\Omega_2^*$, we use $[\emptyset\times x]$ to denote the set
\[
  \{\gamma_1\times\gamma_2\mid
  \text{$\gamma_1\in\Omega_1^\infty$, $\gamma_2\in\Omega_2^\infty$,
  and $x$ is a prefix of $\gamma_2$}\}.
\]
On the other hand, for each $x\in\Omega_2^*$ and $W\subset(\Omega_1\times\Omega_2)^*$,
we use $F(W,x)$ to denote the set of all
$\sigma_1\in\Omega_1^*$ such that there exists $\sigma_2\in\Omega_2^*$ for which
(i) $\abs{\sigma_1}=\abs{\sigma_2}$,
(ii) $\sigma_1\times\sigma_2\in W$, and
(iii) $\sigma_2$ is a prefix of $x$.
It is then easy to see that
\begin{equation}\label{FWx}
  P_1(F(W,x))P_2(x)=\Bm{P_1\times P_2}{\osg{W}\cap[\emptyset\times x]}
\end{equation}
for every $x\in\Omega_2^*$ and every prefix-free subset $W$ of $(\Omega_1\times\Omega_2)^{\le\abs{x}}$.
For each $d\in\N^+$, let
\[
  S_d=\bigl\{x\in\Omega_2^*\bigm|
  2^{-d}<P_1(F(\mathcal{V}_d\cap (\Omega_1\times\Omega_2)^{\le\abs{x}},x))\bigr\}.
\]
Since $P_1$ is left-computable, $S_d$ is r.e.~relative to $\beta_1,\dots,\beta_\ell$ uniformly in $d$.

Let $d\in\N^+$.
Let $\{x_i\}$ be a listing of the minimal strings in $S_d$.
It follows
from \eqref{FWx}
that
\begin{align*}
  2^{-d}\Bm{P_2}{\osg{x_i}}&=2^{-d}P_2(x_i)
  \le P_1(F(\mathcal{V}_d\cap (\Omega_1\times\Omega_2)^{\le\abs{x_i}},x_i))P_2(x_i) \\
  &=\Bm{P_1\times P_2}{\osg{\mathcal{V}_d\cap (\Omega_1\times\Omega_2)^{\le\abs{x_i}}}\cap[\emptyset\times x_i]} \\
  &\le\Bm{P_1\times P_2}{\osg{\mathcal{V}_d}\cap[\emptyset\times x_i]}.
\end{align*}
Since the sets
$\{\osg{\mathcal{V}_d}\cap[\emptyset\times x_i]\}_i$
are pairwise disjoint, we have
\[
  \sum_{i}2^{-d}\Bm{P_2}{\osg{x_i}}
  \le\sum_{i}\Bm{P_1\times P_2}{\osg{\mathcal{V}_d}\cap[\emptyset\times x_i]}
  \le\Bm{P_1\times P_2}{\osg{\mathcal{V}_d}}<2^{-2d}.
\]
Thus, since $\Bm{P_2}{\osg{S_d}}=\sum_{i}\Bm{P_2}{\osg{x_i}}$ and
$d$ is an arbitrary positive integer,
we have that
$$\Bm{P_2}{\osg{S_d}}<2^{-d}$$
for every $d\in\N^+$.
For each $d\in\N^+$, let $T_d=\bigcup_{c=d}^\infty S_{c+1}$.
It follows that
\[
  \Bm{P_2}{\osg{T_d}}\le\sum_{c=d}^\infty\Bm{P_2}{\osg{S_{c+1}}}<2^{-d}
\]
for each $d\in\N^+$, and
$T_d$
is r.e.~relative to $\beta_1,\dots,\beta_\ell$
uniformly in $d$.

Now, let us assume that
$\alpha_2$ is Martin-L\"of $P_2$-random relative to $\beta_1,\dots,\beta_\ell$.
We will then show that
$\alpha_1$ is not Martin-L\"of $P_1$-random relative to $\alpha_2,\beta_1,\dots,\beta_\ell$,
in what follows.
If $\alpha_2\in\osg{S_d}$ for infinitely many $d$,
then we have that
$\alpha_2\in\osg{T_d}$
for every $d$,
and therefore
using Theorem~\ref{ML_P-randomness_eliminated-prefix-freeness-relative-to-infinite-sequences}
we have that
$\alpha_2$ is not Martin-L\"of $P_2$-random relative to $\beta_1,\dots,\beta_\ell$.
This contradicts the assumption.
Thus, there must exists $d_0\in\N^+$ such that $\alpha_2\notin\osg{S_d}$ for every $d>d_0$.

For each $d,n\in\N^+$, let
\[
  H_d(n)=\{w\in\Omega_1^n\mid
  \osg{w\times\rest{\alpha_2}{n}}\subset\osg{\mathcal{V}_d\cap(\Omega_1\times\Omega_2)^{\le n}}\}.
\]
Let $d,n\in\N^+$, and let ${w_1,\dots,w_m}$ be a listing of all elements of $H_d(n)$.
Since
\[
  \osg{w_i\times\rest{\alpha_2}{n}}\subset
  \osg{\mathcal{V}_d\cap(\Omega_1\times\Omega_2)^{\le n}}\cap[\emptyset\times\rest{\alpha_2}{n}]
\]
for every $i=1,\dots,m$,
and the sets
$\{\osg{w_i\times\rest{\alpha_2}{n}}\}_i$
are pairwise disjoint,
we see that
\begin{equation}\label{postFWx}
\begin{split}
  \Bm{P_1}{\osg{H_d(n)}}\Bm{P_2}{\osg{\rest{\alpha_2}{n}}}
  &=\left(\sum_{i=1}^m\Bm{P_1}{\osg{w_i}}\right)\Bm{P_2}{\osg{\rest{\alpha_2}{n}}} \\
  &=\sum_{i=1}^m\Bm{P_1}{\osg{w_i}}\Bm{P_2}{\osg{\rest{\alpha_2}{n}}} \\
  &=\sum_{i=1}^m\Bm{P_1\times P_2}{\osg{w_i\times\rest{\alpha_2}{n}}} \\
  &=\Bm{P_1\times P_2}{\bigcup_{i=1}^m\osg{w_i\times\rest{\alpha_2}{n}}} \\
  &\le\Bm{P_1\times P_2}{\osg{\mathcal{V}_d\cap(\Omega_1\times\Omega_2)^{\le n}}\cap[\emptyset\times\rest{\alpha_2}{n}]}.
\end{split}
\end{equation}
Assume that $d>d_0$.
Then, since $\alpha_2\notin\osg{S_d}$, we have $\rest{\alpha_2}{n}\notin S_d$.
It follows from \eqref{FWx} that
\begin{align*}
  \Bm{P_1\times P_2}{\osg{\mathcal{V}_d\cap(\Omega_1\times\Omega_2)^{\le n}}\cap[\emptyset\times\rest{\alpha_2}{n}]}
  &=P_1(F(\mathcal{V}_d\cap(\Omega_1\times\Omega_2)^{\le n},\rest{\alpha_2}{n}))P_2(\rest{\alpha_2}{n}) \\
  &\le 2^{-d}\Bm{P_2}{\osg{\rest{\alpha_2}{n}}}.
\end{align*}
Therefore, using \eqref{postFWx} we have
\[
  \Bm{P_1}{\osg{H_d(n)}}\Bm{P_2}{\osg{\rest{\alpha_2}{n}}}\le 2^{-d}\Bm{P_2}{\osg{\rest{\alpha_2}{n}}}.
\]
Since $\alpha_2$ is Martin-L\"of $P_2$-random relative to $\beta_1,\dots,\beta_\ell$,
we can show that $\Bm{P_2}{\osg{\rest{\alpha_2}{n}}}>0$,
in a similar manner to the proof of Theorem~\ref{zero_probability-Bs}.
Hence, we see that
\begin{equation}\label{PHdn}
  \Bm{P_1}{\osg{H_d(n)}}\le 2^{-d}
\end{equation}
for every $d>d_0$ and $n$.

On the other hand, we see that $\osg{H_d(n)}\subset\osg{H_d(n+1)}$ for every $d$ and $n$.
For each $d\in\N^+$, let $H_d=\bigcup_{n=1}^\infty H_{d+d_0}(n)$.
It follows from \eqref{PHdn} that $\Bm{P_1}{\osg{H_d}}<2^{-d}$ for every $d\in\N^+$.
It is easy to show that
\[
  H_d(n)=\{w\in\Omega_1^n\mid
  \text{Some prefix of $w\times\rest{\alpha_2}{n}$ is in $\mathcal{V}_d$}\}
\]
for every $d,n\in\N^+$.
It follows that $H_d$ is r.e.~relative to $\alpha_2,\beta_1,\dots,\beta_\ell$ uniformly in $d$.

Let $d\in\N^+$.
Since $\alpha_1\times\alpha_2\in\osg{\mathcal{V}_{d+d_0}}$, there exists $n\in\N^+$ such that
$\rest{(\alpha_1\times\alpha_2)}{n}\in\mathcal{V}_{d+d_0}$.
It follows that
$\rest{\alpha_1}{n}\times\rest{\alpha_2}{n}\in\mathcal{V}_{d+d_0}\cap(\Omega_1\times\Omega_2)^{\le n}$,
and therefore $\rest{\alpha_1}{n}\in H_{d+d_0}(n)$.
It follows that $\alpha_1\in\osg{H_{d+d_0}(n)}\subset\osg{H_d}$.
Therefore, $\alpha_1\in\osg{H_d}$ for every $d\in\N^+$.
Hence,
using Theorem~\ref{ML_P-randomness_eliminated-prefix-freeness-relative-to-infinite-sequences}
we have that
$\alpha_1$ is not Martin-L\"of $P_1$-random relative to $\alpha_2,\beta_1,\dots,\beta_\ell$,
as desired.
This completes the proof.
\end{proof}

\subsection{Equivalence between the three independence notions on computable discrete probability spaces}

Theorem~\ref{mgvL} below
gives an equivalent characterization of the notion of the independence of
ensembles in terms of Martin-L\"of $P$-randomness relative to an oracle.

\begin{theorem}[Generalization of van Lambalgen's Theorem III]\label{mgvL}
Let $n\ge 2$.
Let $\Omega_1,\dots,\Omega_n$ be r.e.~infinite sets,
and let $P_1\in\PS(\Omega_1),\dots,P_n\in\PS(\Omega_n)$.
Let $\alpha_1,\dots,\alpha_n$ be ensembles for $P_1,\dots,P_n$, respectively.
Suppose that $P_1,\dots,P_{n-1}$ are computable.%
\footnote{The computability of $P_n$ is not required in the theorem.}
Then the ensembles $\alpha_1,\dots,\alpha_n$ are independent if and only if
for every $k=1,\dots,n-1$ it holds that
$\alpha_k$ is Martin-L\"of $P_k$-random relative to $\alpha_{k+1},\dots,\alpha_n$.\qed
\end{theorem}

\begin{proof}
Theorem~\ref{mgvL} follows immediately from Theorem~\ref{cor_gvL}.
\end{proof}

Combining Theorem~\ref{independence-independence} with Theorem~\ref{mgvL},
we obtain the following theorem.

\begin{theorem}\label{ind-vL}
Let $\Omega$ and $\Omega_1\dots,\Omega_n$ be r.e.~infinite sets,
and let $P\in\PS(\Omega)$.
Let $X_1\colon\Omega\to\Omega_1,\dots,X_n\colon\Omega\to\Omega_n$ be
random variables on $\Omega$.
Suppose that (i) all of $X_1,\dots,X_n$ are partial recursive functions
and (ii) $X_1(P),\dots,X_{n-1}(P)$ are computable.
Then the following conditions are equivalent to one another.
\begin{enumerate}
  \item The random variables $X_1,\dots,X_n$ are independent on $P$.
  \item For every ensemble $\alpha$ for $P$ and every $k=1,\dots,n-1$ it holds that
    $X_k(\alpha)$ is Martin-L\"of $X_k(P)$-random relative to $X_{k+1}(\alpha),\dots,X_n(\alpha)$.
  \item There exists an ensemble $\alpha$ for $P$ such that for every $k=1,\dots,n-1$ it holds that
    $X_k(\alpha)$ is Martin-L\"of $X_k(P)$-random relative to $X_{k+1}(\alpha),\dots,X_n(\alpha)$
\end{enumerate}
\end{theorem}

\begin{proof}
Let $\alpha$ be an arbitrary ensemble for $P$.
Then it follows from Theorem~\ref{X(p)-X(P)} that
$X_1(\alpha),\dots,X_n(\alpha)$ are ensembles for $X_1(P),\dots,X_n(P)$, respectively.
Therefore, in the case where $X_1(P),\dots,X_{n-1}(P)$ are computable,
using Theorem~\ref{mgvL} we have that
the ensembles $X_1(\alpha),\dots,X_n(\alpha)$ are independent if and only if
for every $k=1,\dots,n-1$ it holds that
$X_k(\alpha)$ is Martin-L\"of $X_k(P)$-random relative to $X_{k+1}(\alpha),\dots,X_n(\alpha)$.
Thus, Theorem~\ref{ind-vL} follows from Theorem~\ref{independence-independence}.
\end{proof}

In Theorem~\ref{ind-vL}, the computability of $X_1(P),\dots,X_{n-1}(P)$ is required. 
The computability of $X_1(P),\dots,X_{n-1}(P)$
follows
from
the computability of $P$ together with the partial recursiveness of $X_1,\dots,X_{n-1}$,
as the following theorem states.

\begin{theorem}\label{computable-P-and-recursive-X-implies-computable-X(P)}
Let $\Omega$ and $\Omega'$ be r.e.~infinite sets, and
let $X\colon\Omega\to\Omega'$ be random variables on $\Omega$.
Suppose that $X$ is a partial recursive function. 
For every $P\in\PS(\Omega)$, if $P$ is computable then $X(P)$ is computable.
\end{theorem}

\begin{proof}
Since $X$ is a partial recursive function and $\Omega'$ is r.e.,
it is easy to see that $X(P)$ is left-computable for every computable $P\in\PS(\Omega)$.
Thus, the result follows from Proposition~\ref{computable-left-computable-right-computable-equivalent}.
\end{proof}

Note that the converse of Theorem~\ref{computable-P-and-recursive-X-implies-computable-X(P)}
does not holds.
Namely, even under the partial recursiveness of $X$, the computability of $X(P)$ does not necessarily
imply the computability of $P$.

Theorem~\ref{ind-vL} results in the following theorem,
using Theorem~\ref{computable-P-and-recursive-X-implies-computable-X(P)}.

\begin{theorem}\label{ind-vL2}
Let $\Omega$ and $\Omega_1\dots,\Omega_n$ be r.e.~infinite sets,
and let $P\in\PS(\Omega)$.
Let $X_1\colon\Omega\to\Omega_1,\dots,X_n\colon\Omega\to\Omega_n$ be
random variables on $\Omega$.
Suppose that (i) all of $X_1,\dots,X_n$ are partial recursive functions
and (ii) $P$ is computable.
Then the following conditions are equivalent to one another.
\begin{enumerate}
  \item The random variables $X_1,\dots,X_n$ are independent on $P$.
  \item For every ensemble $\alpha$ for $P$ and every $k=1,\dots,n-1$ it holds that
    $X_k(\alpha)$ is Martin-L\"of $X_k(P)$-random relative to $X_{k+1}(\alpha),\dots,X_n(\alpha)$.
  \item There exists an ensemble $\alpha$ for $P$ such that for every $k=1,\dots,n-1$ it holds that
    $X_k(\alpha)$ is Martin-L\"of $X_k(P)$-random relative to $X_{k+1}(\alpha),\dots,X_n(\alpha)$
\end{enumerate}
\end{theorem}

Theorem~\ref{independence-independence} and Theorem~\ref{ind-vL2}
together show that
\emph{the three independence notions we have considered so far:
the independence of random variables, the independence of
ensembles, and the independence in the sense of van Lambalgen's Theorem,
are equivalent to one another}
on an arbitrary \emph{computable} discrete probability space.

Now, Theorem~\ref{ind-events-vL} below follows from Theorem~\ref{ind-vL}.

\begin{theorem}\label{ind-events-vL}
Let $\Omega$ be an r.e.~infinite set, and let $P\in\PS(\Omega)$.
Let $A_1,\dots,A_n$ be recursive events on the discrete probability space $P$.
Suppose that the finite probability space $\charaps{P}{A_k}$ is computable for every $k=1,\dots,n-1$.
Then the following conditions are equivalent to one another.
\begin{enumerate}
  \item The events $A_1,\dots,A_n$ are independent on $P$.
  \item For every ensemble $\alpha$ for $P$ and every $k=1,\dots,n-1$ it holds that
    $\chara{A_k}{\alpha}$ is Martin-L\"of $\charaps{P}{A_k}$-random relative to $\chara{A_{k+1}}{\alpha},\dots,\chara{A_n}{\alpha}$.
  \item There exists an ensemble $\alpha$ for $P$ such that for every $k=1,\dots,n-1$ it holds that
    $\chara{A_k}{\alpha}$ is Martin-L\"of $\charaps{P}{A_k}$-random relative to $\chara{A_{k+1}}{\alpha},\dots,\chara{A_n}{\alpha}$.
\end{enumerate}
\end{theorem}

\begin{proof}
The result is obtained by applying Theorem~\ref{ind-vL}
to the random variables
$\chi_{A_1},\dots,\chi_{A_n}$
as $X_1,\dots,X_n$, respectively,
and then using Proposition~\ref{independence-random-variables-events}.
\end{proof}

Theorem~\ref{ind-events-vL} results in the following theorem,
using Theorem~\ref{computable-P-and-recursive-X-implies-computable-X(P)}.

\begin{theorem}\label{ind-events-vL2}
Let $\Omega$ be an r.e.~infinite set, and let $P\in\PS(\Omega)$.
Let $A_1,\dots,A_n$ be recursive events on the discrete probability space $P$.
Suppose that
$P$ is computable.
Then the following conditions are equivalent to one another.
\begin{enumerate}
  \item The events $A_1,\dots,A_n$ are independent on $P$.
  \item For every ensemble $\alpha$ for $P$ and every $k=1,\dots,n-1$ it holds that
    $\chara{A_k}{\alpha}$ is Martin-L\"of $\charaps{P}{A_k}$-random relative to $\chara{A_{k+1}}{\alpha},\dots,\chara{A_n}{\alpha}$.
  \item There exists an ensemble $\alpha$ for $P$ such that for every $k=1,\dots,n-1$ it holds that
    $\chara{A_k}{\alpha}$ is Martin-L\"of $\charaps{P}{A_k}$-random relative to $\chara{A_{k+1}}{\alpha},\dots,\chara{A_n}{\alpha}$.
\end{enumerate}
\end{theorem}

Theorem~\ref{independence-independence-events} and Theorem~\ref{ind-events-vL2} together
show that the three independence notions are equivalent
to one another
for recursive events
on
an arbitrary \emph{computable}
discrete probability space.

\section{Concluding remarks}
\label{Concluding}

In this paper we have developed an operational characterization of the notion of probability
for a \emph{discrete probability space}.

In our former work~\cite{T14,T15,T16arXiv},
as the first step
for developing a framework for an operational characterization of the notion of probability,
we considered
the case of
a
\emph{finite probability space}, where the sample space is \emph{finite}.
In this paper, as the next step of the research
of this line,
we
have considered the case of \emph{discrete probability space},
where the sample space is \emph{countably infinite}.
Actually,
in this case
we have been able to develop
a framework for
an
operational characterization of the notion of probability,
in
the same manner as the case of
a
finite probability space.

The
major application of our framework is \emph{to quantum mechanics}.
The notion of probability plays a crucial role in quantum mechanics.
It appears in quantum mechanics as the so-called \emph{Born rule}, i.e.,
the \emph{probability interpretation of the wave function}.
In modern mathematics which describes quantum mechanics, however,
probability theory means nothing other than measure theory,
and therefore any operational characterization of the notion of probability is still missing
in quantum mechanics.
In this sense, the current form of quantum mechanics is considered to be \emph{imperfect}
as a physical theory which must stand on operational means.

In a series of works~\cite{T15Kokyuroku,T15WiNF-Tadaki_rule,T16QIT35,T18arXiv},
as a \emph{major application} of
the framework
introduced and
developed by our former work~\cite{T14,T15,T16arXiv},
we presented
a \emph{refinement of the Born rule},
based on the notion of ensemble
for a finite probability space,
for the purpose of making quantum mechanics \emph{perfect},
in the case where the number of possible measurement outcomes is \emph{finite}.
Specifically,
we
used the notion of ensemble
for a finite probability space,
in order
to state
the refined rule of the Born rule,
for specifying the property of the results of quantum measurements \emph{in an operational way}.
We
then presented
a refinement of the Born rule
for mixed states,
based on the notion of ensemble
for a finite probability space.
In particular, we gave a
\emph{precise}
definition for the notion of mixed state.
Finally, we
showed that
\emph{all} of the refined rules of the Born rule for both pure states and mixed states
can be derived
from a
\emph{single}
postulate, called the \emph{principle of typicality}, in a unified manner.
We did this from the point of view of the \emph{many-worlds interpretation of quantum mechanics}
\cite{E57}.

In the works~\cite{T15Kokyuroku,T15WiNF-Tadaki_rule,T16QIT35,T18arXiv} above,
for simplicity,
we considered only the case of finite-dimensional quantum systems and measurements over them.
As the next step of the research,
it is natural to consider the case of infinite-dimensional quantum systems,
and measurements over them where the set of possible measurement outcomes is
\emph{countably infinite}.
Actually,
in this case,
based on the framework developed by this paper
we can
certainly
develop
a framework for an operational refinement of the Born rule and the principle of typicality,
using
the notion of ensemble for a discrete probability space.
We can do this
in \emph{almost the same manner} as the finite case developed
through the works~\cite{T15Kokyuroku,T15WiNF-Tadaki_rule,T16QIT35,T18arXiv}.
A full paper which describes the detail of
the application
of our framework
to infinite-dimensional quantum systems
is in preparation.

\section*{Acknowledgments}

This work was supported by JSPS KAKENHI Grant Numbers 15K04981, 18K03405.


\end{document}